\documentclass[10pt, a4paper]{amsart}  
\usepackage{amssymb} 
\usepackage{mathabx} 
\usepackage{epic}  
\usepackage{xypic}  
\usepackage[vcentermath]{youngtab} 
\usepackage[backref=page, colorlinks=false]{hyperref} 
\usepackage{amsthm}
\usepackage[capitalise]{cleveref} 
\usepackage[font=footnotesize,margin={8pt,8pt}]{caption}
\usepackage{url} 
\renewcommand*{\backref}[1]{}
\renewcommand*{\backrefalt}[4]{\quad \tiny 
    \ifcase #1 (Not cited.)%
    \or        (Cited on page~#2.)%
    \else      (Cited on pages~#2.)%
    \fi}

\makeatletter

\def\myMRbibitem{\@ifnextchar[\my@lbibitem\my@bibitem}

\def\mybiblabel#1#2{\@biblabel{{\hyperref{http://www.ams.org/mathscinet-getitem?mr=#1}{}{}{#2}}}}

\def\myhyperanchor#1{\Hy@raisedlink{\hyper@anchorstart{cite.#1}\hyper@anchorend}}

\def\my@lbibitem[#1]#2#3#4\par{%
    \item[\mybiblabel{#2}{#1}\myhyperanchor{#3}\hfill]#4%
    \@ifundefined{ifbackrefparscan}{}{\BR@backref{#3}}%
    \if@filesw{\let\protect\noexpand\immediate
       \write\@auxout{\string\bibcite{#3}{#1}}}\fi\ignorespaces%
}

\def\my@bibitem#1#2#3\par{%
    \refstepcounter\@listctr
    \item[\mybiblabel{#1}{\the\value\@listctr}\myhyperanchor{#2}\hfill]#3%
    \@ifundefined{ifbackrefparscan}{}{\BR@backref{#2}}%
    \if@filesw\immediate\write\@auxout
        {\string\bibcite{#2}{\the\value\@listctr}}\fi\ignorespaces%
}

\makeatother

\newtheoremstyle{note}
  {5pt}
  {5pt}
  {\small}
  {}
  {\bfseries}
  {.}
  {.5em}
  {}

\theoremstyle{plain}
    \newtheorem{lemma}{Lemma}[section]
    \newtheorem{thm}[lemma]{Theorem}
    \newtheorem{prop}[lemma]{Proposition}
    \newtheorem{corol}[lemma]{Corollary}
    \newtheorem{otherthm}[lemma]{Theorem}

    \newtheorem{scho}[lemma]{Scholium}
	\newtheorem{claim}[lemma]{Claim}
	\newtheorem*{repeatedthm}{Theorem \ref{t.schubert}}
\theoremstyle{definition}

\theoremstyle{remark}
    \newtheorem*{ack}{Acknowledgements}
\theoremstyle{note}  
    \newtheorem{rem}[lemma]{Remark}
	\newtheorem{example}[lemma]{Example}

\crefname{thm}{Theorem}{Theorems} 
\Crefname{thm}{Theorem}{Theorems} 
\crefname{otherthm}{Theorem}{Theorems}  
\Crefname{otherthm}{Theorem}{Theorems} 
\crefname{prop}{Proposition}{Propositions} 
\Crefname{prop}{Proposition}{Propositions}

\numberwithin{equation}{section}         

\hypersetup{bookmarksdepth=subsection} 

\crefname{subsection}{\S}{\S\S} 
\Crefname{subsection}{\S}{\S\S} 

\newcommand{\comment}[1]{}

\newcommand{\C}{\mathbb{C}}
\newcommand{\R}{\mathbb{R}}
\newcommand{\Q}{\mathbb{Q}}
\newcommand{\Z}{\mathbb{Z}}
\newcommand{\N}{\mathbb{N}}

\newcommand{\K}{\mathbb{K}}
\renewcommand{\P}{\mathbb{P}}

\newcommand{\Si}{\Sigma}

\renewcommand{\cD}{\mathcal{D}}
\newcommand{\cE}{\mathcal{E}}\newcommand{\cF}{\mathcal{F}}
\newcommand{\cG}{\mathcal{G}}
\newcommand{\cN}{\mathcal{N}}\newcommand{\cO}{\mathcal{O}}
\newcommand{\cP}{\mathcal{P}}
\newcommand{\cS}{\mathcal{S}}\newcommand{\cU}{\mathcal{U}}
\newcommand{\cX}{\mathcal{X}}
\newcommand{\cZ}{\mathcal{Z}}
\newcommand{\bA}{\mathbf{A}}

\newcommand{\bJ}{\mathbf{J}}

\newcommand{\tC}{\mathtt{C}}
\newcommand{\tE}{\mathtt{E}}

\newcommand{\tJ}{\mathtt{J}}

\newcommand{\tR}{\mathtt{R}}

\renewcommand{\setminus}{\smallsetminus}
\renewcommand{\emptyset}{\varnothing}

\newcommand{\GL}{\mathrm{GL}}

\newcommand{\gl}{\mathfrak{gl}}

\newcommand{\Id}{\mathrm{Id}}
\newcommand{\Mat}{\mathrm{Mat}}
\newcommand{\Diag}{\mathrm{Diag}}
\newcommand{\Ad}{\mathrm{Ad}}
\newcommand{\id}{\mathrm{id}}

\newcommand{\RP}{\mathbb{R}\mathrm{P}}
\newcommand{\CP}{\mathbb{C}\mathrm{P}}
\newcommand{\KP}{\mathbb{K}\mathrm{P}}
\renewcommand{\Im}{\mathrm{Im}\;}
\DeclareMathOperator{\codim}{codim}
\DeclareMathOperator{\Ker}{Ker}

\DeclareMathOperator*{\spa}{span}
\newcommand{\sorb}{\mathfrak{R}}
\DeclareMathOperator{\pop}{wgt}

\DeclareMathOperator{\rig}{rig}
\DeclareMathOperator{\rank}{rank}
\DeclareMathOperator{\acyc}{acyc}
\DeclareMathOperator{\col}{col}
\newcommand{\prow}{\pi_\mathrm{r}}
\newcommand{\pcol}{\pi_\mathrm{c}}

\renewcommand{\pitchfork}{\;\;\makebox[0pt]{$\top$}\makebox[0pt]{\small $\cap$}\;\;} 

\newcommand{\cupro}{\smallsmile}
\newcommand{\brickpattern}{\drawline(.25,0)(.25,.25)\drawline(.75,.25)(.75,.5)\drawline(.25,.5)(.25,.75)\drawline(.75,.75)(.75,1)\put(0,0){\grid(1,1)(1,.25)}}
\newcommand{\diagonalpattern}{\drawline(.75,0)(1,.25)\drawline(.5,0)(1,.5)\drawline(.25,0)(1,.75)\drawline(0,0)(1,1)\drawline(0,.25)(.75,1)\drawline(0,.5)(.5,1)\drawline(0,.75)(.25,1)}
\newcommand{\lightdiagonalpattern}{\drawline(.5,0)(1,.5)\drawline(0,0)(1,1)\drawline(0,.5)(.5,1)}

\title[Universal Regular Control]{Universal Regular Control for Generic Semilinear Systems}
\author[Bochi]{Jairo Bochi}
\address{Pontif\'icia Universidade Cat\'olica do Rio de Janeiro (PUC--Rio)}
\email{jairo@mat.puc-rio.br}
\author[Gourmelon]{Nicolas Gourmelon}
\address{Institut de Math\'{e}matiques de Bordeaux, Universit\'{e} Bordeaux I}
\email{ngourmel@math.u-bordeaux1.fr}
\date{\today}
\thanks{J.B.\ was partially supported by CNPq, Universit\'{e} Bordeaux I, Rede Franco--Brasileira em Matem\'atica, and FAPERJ. N.G.\ was partially supported by FAPERJ}

\begin{document}

\begin{abstract} 
We consider discrete-time projective semilinear control systems $\xi_{t+1} = A(u_t) \cdot \xi_t$, where the states $\xi_t$ are in projective space $\RP^{d-1}$, inputs $u_t$ are in a manifold $\cU$ of arbitrary finite dimension, and $A \colon \cU \to \GL(d,\R)$ is a differentiable mapping.

An input sequence $(u_0,\ldots,u_{N-1})$ is called universally regular if for any initial state $\xi_0 \in \RP^{d-1}$, the derivative of the time-$N$ state with respect to the inputs is onto.

In this paper we deal with the universal regularity of constant input sequences $(u_0, \dots, u_0)$. Our main result states that generically in the space of such systems, for sufficiently large $N$, all constant inputs of length $N$ are universally regular, with the exception of a discrete set. More precisely, the conclusion holds for a $C^2$-open and $C^\infty$-dense set of maps $A$, and $N$ only depends on $d$ and on the dimension of $\cU$. We also show that the inputs on that discrete set are nearly universally regular; indeed there is a unique non-regular initial state, and its corank is $1$.

In order to establish the result, we study the spaces of bilinear control systems. We show that the codimension of the set of systems for which the zero input is not universally regular coincides with the dimension of the control space. The proof is based on careful matrix analysis and some elementary algebraic geometry. Then the main result follows by applying standard transversality theorems.
\end{abstract}

\maketitle

\section{Introduction}\label{s.intro}

\subsection{Basic definitions and some questions}

Consider discrete-time control systems of the form:
\begin{equation}\label{e.general CS}
x_{t+1} = F(x_t,u_t), \qquad (t = 0,1,2, \dots)
\end{equation}
where $F \colon \cX \times \cU \to \cX$ is a map.
We will always assume that the space $\cX$ of states and the space $\cU$ of controls are manifolds, and that the map $F$ is continuously differentiable.

A sequence $(x_0, \dots, x_N ; u_0, \dots, u_{N-1})$ satisfying \eqref{e.general CS} is called a trajectory of length $N$; it is uniquely determined by the initial state $x_0$ and the 
input 
$(u_0,\dots,u_{N-1})$.
Let $\phi_N$ denote the time-$N$ transition map,
which gives the final state as a function of the initial state and the input:
\begin{equation}\label{e.final state}
x_N = \phi_N(x_0; u_0, \dots, u_{N-1}).
\end{equation}

We say that the system \eqref{e.general CS} is \emph{accessible} from $x_0$ in time $N$ if 
the set $\phi_N(\{x_0\} \times \cU^N)$ 
of final states that can be reached from the initial state $x_0$ 
has nonempty interior.

The implicit function theorem gives a sufficient condition for accessibility.
If the derivative of the map $\phi_N(x_0; \cdot)$ at input $(u_0,\dots,u_{N-1})$
is an onto linear map 
then we say that the trajectory determined by $(x_0; u_0, \dots, u_{N-1})$ is \emph{regular}.
So the existence of such a regular trajectory implies that the system
is accessible from $x_0$ in time $N$.

\medskip

Let us call an input $(u_0, \dots, u_{N-1})$ \emph{universally regular}
if for every $x_0 \in \cX$, the trajectory determined by $(x_0; u_0, \dots, u_{N-1})$ is regular;
otherwise the input is called \emph{singular}.

The concept of universal regularity is central in this paper; 
it was introduced by Sontag 
in \cite{Sontag_92} in the context of continuous-time control systems. The discrete-time analogue was considered by Sontag and Wirth in \cite{Sontag_Wirth_98}.
They showed that if the system \eqref{e.general CS} is accessible from every initial condition $x_0$ 
in uniform time $N$ then universally regular inputs do exist, provided one assumes the map $F$ to be analytic. In fact, under those hypotheses they showed that universally regular inputs are abundant: in the space of inputs of sufficiently large length, singular ones form a set of positive codimension.

\medskip

In this paper, we are interested in control systems \eqref{e.general CS} where
the next state $x_{t+1}$ depends linearly on the previous state $x_t$ 
(but non-linearly on $u_t$, in general). 
This means that the state space is $\K^d$,
where $\K$ is either $\R$ or $\C$,
and that \eqref{e.general CS} now takes the form: 
\begin{equation}\label{e.semilin CS}
x_{t+1} = A(u_t) \cdot x_t, \qquad \text{where } A \colon \cU \to \Mat_{d \times d}(\K).
\end{equation}
Following \cite{CK_93}, we call this a \emph{semilinear control system}.

In the case that the map $A$ above takes values in the set $\GL(d,\K)$ of invertible matrices
of size $d \ge 2$, we consider the corresponding projectivized control system:
\begin{equation}\label{e.proj semilin CS}
\xi_{t+1} = A(u_t) \cdot \xi_t, 
\end{equation}
where the states $\xi_t$ take value in the projective space $\KP^{d-1} = \K^d_* / \K_*$.
We call this a \emph{projective semilinear control system}.
The projectivized system is also a useful tool
for the study of the original system \eqref{e.semilin CS}:
see e.g.\ \cite{Wirth_98,CK_book}.



Universally regular inputs for projective semilinear control systems
were first considered by Wirth in \cite{Wirth_98}.
Under his working hypotheses, the existence and abundance of such inputs 
is guaranteed by the aforementioned result of \cite{Sontag_Wirth_98};
then he uses universally regular inputs to obtain global controllability properties.

\medskip

The purpose of this paper is to establish results on the existence and abundance of 
universally regular inputs for projective semilinear control systems. 
Differently from \cite{Sontag_Wirth_98, Wirth_98}, we will not necessarily assume our systems to be analytic.
Let us consider systems \eqref{e.proj semilin CS} with $\K=\R$ 
and $A \colon \cU \to \GL(d,\R)$ a map of class $C^r$, for some fixed $r\ge 1$.
To compensate for less rigidity, we do not try to obtain results that work for all $C^r$ maps $A$, 
but only for \emph{generic} ones, i.e., those maps in a residual (dense $G_\delta$) 
subset, or, even better, in an open dense subset.

To make things more precise, assume $\cU$ is a $C^\infty$ (real) manifold without boundary.
All manifolds are assumed to be Hausdorff paracompact with a countable base of open sets, and of finite dimension.
We will always consider the space $C^r(\cU, \GL(d,\R))$ endowed with the strong $C^r$ topology
(which coincides with the usual uniform $C^r$ topology in the case that $\cU$ is compact).

Hence the first question we pose is this: 
\begin{quote}
Taking $N$ sufficiently large, is it true that for $C^r$-generic maps $A$, the set of universally regular inputs in $\cU^N$ is itself generic?
\end{quote}
It turns out that this question has a positive answer.
Actually, in a work in preparation we show that for generic maps $A$, 
all inputs in $\cU^N$ are universally regular, except for those in 
a stratified closed set of positive codimension.
So another natural question is this:
\begin{quote}
Fixed parameters $d$, $\dim \cU$, $N$, and $r$,
what is the minimum codimension of the set of singular inputs in $\cU^N$
that can occur for $C^r$-generic maps $A \colon \cU \to \GL(d,\R)$?
\end{quote}
In full generality, this question seems to be very difficult.
A simpler setting would be to restrict to \emph{non-resonant inputs}, 
namely those inputs $(u_0,\dots,u_{N-1})$ such that $u_i \neq u_j$ whenever $i \neq j$.  
In this paper we consider the most resonant case.
Define a \emph{constant} input of length $N$ 
as an element of $\cU^N$ of the form $(u_0, u_0, \dots, u_0)$.
We propose ourselves to study universal regularity of inputs of this form.

\subsection{The main result}\label{ss.main_statements}

We prove that generically the singular constant inputs form a very small set:
 
\begin{thm}\label{t.main}
Given $d\ge 2$ and $m \ge 1$, there exists an integer $N$ with $1 \le N \le d^2$ 
such that the following properties hold.
Let $\cU$ be a smooth $m$-dimensional manifold without boundary.
Then there exists a $C^2$-open $C^\infty$-dense subset $\cO$ of $C^2(\cU, \GL(d,\R))$ 
such that for every system \eqref{e.proj semilin CS} with $A \in \cO$,
all constant inputs of length $N$ are universally regular, except for 
those in a zero-dimensional (i.e., discrete) set.
\end{thm}

By saying that a subset $\cO$ of $C^2(\cU, \GL(d,\R))$ is  $C^\infty$-dense, 
we mean that for all $r \ge 2$, the intersection of $\cO$ with $C^r(\cU, \GL(d,\R))$ 
is dense in $C^r(\cU, \GL(d,\R))$.

It is remarkable that the generic dimension of the set of singular constant inputs (namely, $0$) does not depend on the dimension $m$ of the control space $\cU$, neither on the dimension $d-1$ of the state space. A partial explanation for this phenomenon is the following: 
First, the obstruction to universal regularity of the input $(u,u,\dots,u)$ is the combined degeneracy of the matrix $A(u)$ and of the derivatives of $A$ at $u$.
If $m$ is small then the image of the generic map $A$ will avoid too degenerate matrices, which increases the chances of obtaining universal regularity. If $m$ is large then more degenerate matrices $A(u)$ will inevitably appear; however the large number of control parameters compensates, so universal control is still likely.

\medskip

The singular inputs that appear in \cref{t.main} are not only rare;
we also show that they are ``almost'' universally regular:

\begin{thm}[Addendum to \cref{t.main}]\label{t.addendum}
The set $\cO \subset C^2(\cU,\GL(d,\R))$ in \cref{t.main} 
can be taken with the following additional properties:
If $A \in \cO$ and a constant input $(u,\dots,u)$ of length $N$ is singular then:
\begin{enumerate}
\item \label{i.addendum_1} 
There is a single direction $\xi_0 \in \RP^{d-1}$ for which
the corresponding trajectory of system \eqref{e.proj semilin CS} is not regular.
\item \label{i.addendum_2}
The derivative of the map $\phi_N(\xi_0; \cdot)$ at input $(u,\dots,u)$
has corank $1$.
\end{enumerate}
\end{thm}

To sum up, for generic systems \eqref{e.proj semilin CS}, the universal regularity of constant inputs 
can fail only in the weakest possible way: there is at most one non-regular state, 
which can be moved in all directions but one.

We actually describe precisely in \cref{a.generic singular} the singular inputs that appear in \cref{t.addendum}. 
We show that these singular inputs can be unremovable by perturbations,
and therefore \cref{t.main} is optimal in the sense that there are $C^2$-open (actually even $C^1$-open) sets of maps $A$ for which the set of singular constant inputs is nonempty.
Also, by $C^1$-perturbing any $A$ in those $C^2$-open sets,
one can obtain an infinite number of singular constant inputs.
In particular, the set $\cO$ in the statement of the \cref{t.main} is not $C^1$-open in general.


\subsection{Reduction to the study of the set of poor data}

The bulk of the proof of \cref{t.main} consists on the computation 
of the dimension of certain canonical sets,
as we now explain.

\medskip


We fix $A \colon \cU \to \GL(d,\K)$ and consider
the projective semilinear system \eqref{e.proj semilin CS}.
By the chain rule, the universal regularity of an input 
$(u_0, u_1, \ldots, u_{N-1})$
depends only on the $1$-jets of $A$ at points $u_0$, \ldots, $u_{N-1}$,
i.e., on the first order Taylor approximations of $A$ around those points.


Let us discuss the case of constant inputs $(u_0, \ldots, u_0)$.
If we take local coordinates such that $u_0 = 0$
and replace the matrix map $A \colon \cU \to \GL(d,\K)$
by its linear approximation, system \eqref{e.proj semilin CS} becomes:
\begin{equation}\label{e.proj bilin CS}
\xi_{t+1} = \left(A + \sum_{j=1}^m u_{t,i} C_j \right) \xi_t  
\, , \quad (t = 0, 1, 2, \dots) ,
\end{equation}
where $A = A(u_0)$ 
and $C_1$, \dots, $C_m$ are the partial derivatives at $u_0=0$.
This is the projectivization of a \emph{bilinear control system} (see \cite{Elliott}).
For these systems, the zero input is a distinguished one and the focus of more attention.

To study system \eqref{e.proj bilin CS} it is actually more convenient to consider
\emph{normalized derivatives} $B_j = C_j A^{-1}$,
which intrinsically take values in the Lie algebra $\gl(d,\K)$.
Consider the matrix datum $\bA = (A, B_{1}, \ldots, B_{m})$.
We will explain how the universal regularity of the zero input
is expressed in linear algebraic terms.
Recall that the \emph{adjoint operator} of $A$ acts on $\gl(d,\K)$
by the formula $\Ad_A(B) = A B A^{-1}$.
Consider the linear subspace $\Lambda_N(\bA)$ of $\gl(d,\K)$ spanned by the matrices
$$
\Id \quad \text{and} \quad
(\Ad_A)^i(B_j), \quad (i=0,\ldots,n-1, \ j=1,\ldots, m).
$$
(The identity matrix appears because of the projectivization.)
This is nothing but the reachable set from $0$ 
for the linear control system $(\Ad_A, \Id, B_1, \dots, B_m)$.
Then:

\begin{prop}\label{p.reg transitivity}
The constant input $(0,\dots,0)$ of length $N$ is universally regular 
for system \eqref{e.proj bilin CS}
if and only if the space $\Lambda_N(\bA)$ is transitive.
\end{prop}

Here we say that a subspace of $d \times d$ matrices with entries in the field $\K$
is \emph{transitive} if it acts transitively in the set $\K^d_*$ of nonzero vectors.

Clearly, the spaces $\Lambda_N(\bA)$ form a nested sequence that stabilizes 
to a space $\Lambda(\bA)$ at some time $N \le d^2$.
If $\Lambda(\bA)$ is transitive then the datum $\bA$ is called \emph{rich}; 
otherwise it is called \emph{poor}.
Let $\cP_m^{(\K)} = \cP_{m,d}^{(\K)}$ denote the set of poor data.
A major part of our work is to study these sets.
We prove:

\begin{thm}\label{t.cod_data_R} 
The set $\cP_{m}^{(\R)}$ is closed and semialgebraic,
and its codimension in $\GL(d,\R) \times (\gl(d,\R))^m$
is $m$.
\end{thm}

\begin{thm}\label{t.cod_data_C} 
The set $\cP_{m}^{(\C)}$ is algebraic,
and its (complex) codimension in $\GL(d,\C) \times (\gl(d,\C))^m$
is $m$.
\end{thm}

So \cref{t.cod_data_R,t.cod_data_C} 
say how frequent universal regularity of the zero input is 
in the space of projective bilinear control systems \eqref{e.proj bilin CS}.

\subsection{Overview of the proofs}\label{ss.overview}

\Cref{t.main} follows rather directly from \cref{t.cod_data_R} 
by applying standard results from transversality theory.
More precisely, the fact that the set $\cP_m^{(\R)}$ is semialgebraic
implies that it has a canonical stratification.
This permits us to apply Thom's jet transversality theorem and obtain \cref{t.main}.

On the other hand, \Cref{t.cod_data_R}
follows from its complex version \cref{t.cod_data_C} 
by simple abstract arguments.

Thus everything is based on \cref{t.cod_data_C}.
One part of the result is easily obtained:
we give examples of small disks of codimension $m$ formed by poor data,
so concluding that the codimension of $\cP_{m}^{(\C)}$ is at most $m$.

To prove the other inequality, 
one could try to exhibit an explicit codimension~$m$ set containing all poor data.
For $m=1$ this task is feasible
(and we actually perform it, because with these conditions we can 
actually check universal regularity in concrete examples).
However, for $m=2$ already the task would be very laborious,
and to expect to find a general solution seems unrealistic.

Our actual approach to prove the lower bound on the codimension of $\cP_{m}^{(\C)}$ is indirect.
Crudely speaking, after careful matrix computations,
we find some sets in the complement of $\cP_{m}^{(\C)}$
that are reasonably ``large'' (basically in terms of dimension).
Then, by using some abstract results of algebraic geometry,
we are able to show that $\cP_{m}^{(\C)}$ is ``small'', 
thus proving the other half of  \cref{t.cod_data_C}.

Let us give more detail about this strategy.
We decompose the set $\cP_m = \cP_{m}^{(\C)}$ into fibers: 
$$
\cP_m =\bigcup_{A \in \GL(d,\C)} \{A\} \times \cP_m(A) , \qquad \cP_m(A) \subset [\gl(d,\C)]^m.
$$
It is not very difficult to show that for generic $A$ in $\GL(d,\C)$, 
the fiber $\cP_m(A)$ has precisely the wanted codimension $m$.
However, for degenerate matrices $A$, the fiber $\cP_m(A)$ may be much bigger.
(For example, one can show that if $A$ is an homothecy and $m \le 2d-3$ then $\cP_m(A)$ is the whole $[\gl(d,\C)]^m$.)
In order to show that $\codim \cP_m \ge m$, we need to make sure 
that those degenerate matrices do not form a large set.
More precisely, we show that:
\begin{equation}\label{e.everything}
\forall k\in\{0,\ldots,m\}, \ 
\codim \big\{ A \in \GL(d,\C) ; \;  \codim \cP_m(A) \le m-k  \big\} \ge k. 
\end{equation}

Let us explain how we prove~\eqref{e.everything}.
In order to estimate the dimension of $\cP_m(A)$ for any matrix $A \in \GL(d,\C)$,
we consider a quantity $r = r(A)$ which is the least number such that
a rich datum of the form $(A,C_1,\ldots,C_r)$ exists.
In particular, if $r = r(A) \le m$ 
then the following affine space 
\begin{equation}\label{e.affine}
\big\{ (C_1, C_2, \dots, C_r , B_{r+1}, \dots, B_m ) ; \; B_j \in \gl(d,\C) \big\}
\end{equation}
is contained in the complement of $\cP_m(A)$.

In certain situations,  if two algebraic subsets have large enough dimensions 
then they necessarily intersect;
for example, two algebraic curves in the complex projective plane $\CP^2$ always intersect.
This kind of phenomenon happens here:
the dimension of the affine space \eqref{e.affine} 
forces a lower bound for the codimension of $\cP_m(A)$, namely:
\begin{equation}\label{e.algebraic}
\codim \cP_m(A) \ge m+1-r(A).
\end{equation}

So 
we need to show that matrices $A$ with large $r(A)$ are rare.
A careful matrix analysis provides
an upper bound to $r(A)$ based on the numbers and sizes of the Jordan blocks of $A$,
and on the occasional algebraic relations between the eigenvalues.
This bound together with \eqref{e.algebraic}
implies \eqref{e.everything} and therefore concludes the proof of \cref{t.cod_data_C}.

In fact, the results of this analysis are even better,
and we conclude that the codimension inequality \eqref{e.everything}
is strict when $k \ge 1$.
This implies that poor data $(A, B_1, \dots, B_m)$ 
for which the matrix $A$ is degenerate form a subset of $\cP_{m}^{(\C)}$ 
with strictly bigger codimension.
Thus we can show that the poor data that appear generically are well-behaved,
which leads to \cref{t.addendum}.

\subsection{Holomorphic setting}

In the case of complex matrices (i.e., $\K = \C$), we have a 
corresponding version of \cref{t.main} where the maps $A$ are holomorphic.
Given an open subset $\cU \subset \C^m$, we denote by 
$\mathcal{H}(\cU, \GL(d,\C))$ 
the set of holomorphic mappings $A \colon \cU\to \GL(d,\C)$ endowed with the usual topology of uniform convergence on compact sets.

\begin{thm}\label{t.main_C}
Given integers $d\ge 2$ and $m \ge 1$, there exists an integer $N\ge 1$ with the following properties.
Let $\cU\subset \C^m$ be open, and let $K\subset \cU$ be compact.
Then there exists an open and dense subset $\cO$ of $\mathcal{H}(\cU, \GL(d,\C))$
such that for any $A \in \cO$ the constant inputs in $K^N$ are all universally regular for the  system~\eqref{e.proj semilin CS}, except for a finite subset.
\end{thm}

We have the straightforward corollary:

\begin{corol}
Given integers $d\ge 2$ and $m \ge 1$, there exists an integer $N\ge 1$ with the following properties.
Let $\cU\subset \C^m$ be an open subset. There exists a residual subset $\mathcal{R}$ of $\mathcal{H}(\cU, \GL(d,\C))$
such that for any $A \in \mathcal{R}$ the constant inputs in $\cU^N$ are all universally regular for  the  system~\eqref{e.proj semilin CS}, except for a discrete subset.
\end{corol}

\subsection{Directions for future research}

One can also study uniform regularity of periodic inputs of higher period.
Using our results for constant inputs, it is not difficult 
to derive some (non-sharp) codimension bounds for singular periodic inputs for generic systems.
However, for highly resonant non-periodic inputs, we have no idea on how to obtain reasonable dimension estimates. 

To obtain good estimates for the codimension of \emph{non-resonant}
singular inputs for generic systems is relatively simpler from the point of view of matrix computations,
but needs more sophisticated transversality theorems (e.g., multijet transversality).
Since highly resonant inputs have large codimension themselves,
it seems possible to obtain reasonably good codimension estimates for general inputs for generic systems.

Another interesting direction of research is to consider other Lie groups of matrices.

%

\subsection{Organization of the paper}\label{ss.organization}

\Cref{s.prelim_poor} contains some basic results about 
transitivity of spaces of matrices and its relation with universal regularity.
We also obtain the easy parts of \cref{t.cod_data_R,t.cod_data_C}, namely
(semi)algebraicity and the upper codimension inequalities.

In \cref{s.rig} we introduce the concept of rigidity, which is related to the quantity $r(A)$
mentioned above. We state the central rigidity estimates (\cref{t.rig}), which
consist into two parts. The first and easier part is proved in the same \cref{s.rig},
while the whole \cref{s.rig proof} is devoted to the proof of the second part.

\Cref{s.cod proof} starts with some preliminaries in elementary algebraic geometry.
Then we use the rigidity estimates to prove \cref{t.cod_data_C}, following the strategy 
outlined above (\cref{ss.overview}).
\Cref{t.cod_data_R} follows easily.
We also obtain a lemma that is needed for the proof of \cref{t.addendum}.

In \cref{s.main proof} we deduce \cref{t.main}
from  previous results and standard theorems stratifications and transversality.

The paper also has some appendices:

\Cref{a.dim 1} basically reobtains the major results in the special case $m=1$,
where we actually gain additional information of practical value:
as mentioned in \cref{ss.overview}, it is possible to describe explicitly what $1$-jets
the map $A$ should avoid in order to satisfy the conclusions of \cref{t.main,t.addendum}.
The arguments necessary for the $m=1$ case are much simpler and more elementary
than those in \cref{s.rig,s.rig proof,s.cod proof}.
Therefore the \lcnamecref{a.dim 1} is also useful to give the reader some intuition about the
general problem, and as a source of examples.
\Cref{a.dim 1} is written in a slightly informal way, and it can be read after \cref{s.prelim_poor}
(though the final part requires \cref{l.sum}).

\Cref{a.algebraic} contains the proofs of necessary algebraic-geometric results,
especially the one that allows us to obtain estimate \eqref{e.algebraic}.

\Cref{a.strat_trans} reviews the necessary concepts and results on stratifications,
and proves a prerequisite transversality proposition.

In \cref{a.complex} we apply \cref{t.cod_data_C} 
to prove a version of \cref{t.main} for holomorphic mappings.

In \cref{a.generic singular} we study the singular constant inputs of generic type,
proving \cref{t.addendum} and the other assertions made at the end of \S~\ref{ss.main_statements}
concerning the sharpness of \cref{t.main}.
We also discuss the generic validity of some control-theoretic properties 
related to accessibility and regularity.

\section{Preliminary facts on the poor data}\label{s.prelim_poor}

In this section, we review some basic properties 
related to poorness, and prove the easy inequalities in \cref{t.cod_data_R,t.cod_data_C}.

\subsection{Transitive spaces}\label{ss.transitivity}

Let $E$ and $F$ be finite-dimensional vector spaces over the field $\K$.
Let $\mathcal{L}(E,F)$ be the space of linear maps from $E$ to $F$.
A vector subspace $\Lambda$ of $\mathcal{L}(E,F)$ is called \emph{transitive}
if for every $v \in E \setminus \{0\}$, we have $\Lambda \cdot v = F$,
where $\Lambda \cdot v = \{ L(v) ; \; L \in \Lambda \}$.

Under the identification $\mathcal{L}(\K^n, \K^m) = \Mat_{m \times n}(\K)$,
we may also speak of transitive spaces of matrices.

The following examples illustrate the concept; they will also be needed in later considerations.

\begin{example}\label{ex.toeplitz}
Recall that a \emph{Toeplitz matrix}, resp.\ a \emph{Hankel matrix}, is a matrix of the form
$$
\left(
\vcenter{
\xymatrix @-2pc {
t_0    \ar@{.}[rrrddd] & t_1 \ar@{.}[rrdd] & \cdots & t_{d-1} \\
t_{-1} \ar@{.}[rrdd]   &                   &        & \vdots  \\
\vdots                 &                   &        & t_1     \\
t_{-d+1}               & \cdots            & t_{-1} & t_0
}}
\right) ,
\quad\text{resp.}\quad
\left(
\vcenter{
\xymatrix @-2pc {
h_1       & \cdots  & h_{d-1} \ar@{.}[lldd] & h_d     \ar@{.}[lllddd] \\
\vdots    &         &                       & h_{d+1} \ar@{.}[lldd]   \\
h_{d-1}   &         &                       & \vdots                  \\
h_d       & h_{d+1} & \cdots                & h_{2d-1}
}}
\right),
$$
The set of Toeplitz matrices 
and the set of complex Hankel matrices 
constitute examples transitive subspaces of $\gl(d,\K)$.
Transitivity of the Toeplitz space is a particular case of \cref{ex.gen Toeplitz},
and transitivity of Hankel space follows from \cref{r.transitivity trick}.
For $\K = \C$, these spaces are optimal, in the sense that they have the least possible dimension; see \cite{Azoff}.
\end{example}

\begin{example}\label{ex.gen Toeplitz}
A \emph{generalized Toeplitz space} is a subspace $\Lambda$ of $\Mat_{d\times d}(\K)$ 
(where $d \ge 2$)
with the following property:
For any two matrix entries $(i_1,j_1)$ and $(i_2,j_2)$ which are not in the same diagonal
(i.e., $i_1-j_1 \neq i_2-j_2$), 
the linear map $(b_{i,j})_{i,j} \in \Lambda \mapsto (b_{i_1,j_1},b_{i_2,j_2}) \in \C^2$ is onto.
Equivalently, a space is generalized Toeplitz if 
it can be defined by a number of linear relations between the matrix coefficients
so that each relation involves only the entries on a same diagonal,
and so that the relations do not force any matrix entry to be zero.
We will prove later (see \cref{ss.sudoku})
that \emph{every generalized Toeplitz space is transitive}.
\end{example}

\begin{rem}\label{r.transitivity trick}
If $\Lambda$ is a transitive subspace of $\mathcal{L}(E,F)$
and $P \in \mathcal{L}(E,E)$, $Q \in \mathcal{L}(F,F)$ are invertible operators then
$P \cdot \Lambda \cdot Q := \{PLQ ; \;  L \in \Lambda\}$ 
is a transitive subspace of $\mathcal{L}(E,F)$.
\end{rem}

\medskip

Let us see that transitivity is a semialgebraic or algebraic property, according to the field.
Recall that:
\begin{itemize}
\item A subset of $\K^n$ is called \emph{algebraic} if it is expressed by 
polynomial equations with coefficients in $\K$.
\item A subset of $\R^n$ is called \emph{semialgebraic} if it is 
the union of finitely many sets, each of them defined by finitely many  real polynomial equations
and inequalities (see \cite{BR,BCR}).
\end{itemize}

\begin{prop}\label{p.NT algebraicity}
Let $\cN_{m,n,k}^{(\K)}$ be the set of $(B_1, \dots, B_k) \in [\Mat_{m \times n}(\K)]^k=\K^{mnk}$ such that
$\spa\{B_1, \dots, B_k\}$ is not transitive.
Then:
\begin{enumerate}
\item\label{i.semialgebraic}
The set $\cN_{m,n,k}^{(\R)}$ is semialgebraic.
\item\label{i.algebraic} 
The set $\cN_{m,n,k}^{(\C)}$ is algebraic.
\end{enumerate}
\end{prop}

\begin{proof} 
Consider the set of $(B_1, \dots, B_k, v) \in [\Mat_{m \times n}(\K)]^k \times \K^n_*$ 
such that 
$$
\spa\{B_1, \dots, B_k\}\cdot v \neq \K^m \, .
$$
For $\K = \R$, this is a semialgebraic set,
because it is expressed by the vanishing of certain determinants
plus the condition $v \neq 0$.
Projecting this set along the $\R^n_*$ fiber we obtain $\cN_{m,n,k}^{(\R)}$;
so, by the Tarski--Seidenberg theorem (see \cite[p.~60]{BR} or \cite[p.~26]{BCR}),
this set is semialgebraic, proving part~\ref{i.semialgebraic}.

To see part~\ref{i.algebraic},
we take $\K=\C$ and projectivize the $\C^n_*$ fiber, obtaining an
algebraic subset $[\Mat_{m \times n}(\C)]^k \times \CP^{n-1}$
whose projection along the $\CP^{n-1}$ fiber is  $\cN_{m,n,k}^{(\C)}$.
So part~\ref{i.algebraic} follows from the fact that projections 
along projective fibers take algebraic sets to algebraic sets (see \cite[p.~58]{Shafa}).
\end{proof}

Complex transitivity of real matrices is a stronger property than real transitivity:

\begin{prop}\label{p.NT inclusion}
The real part of $\cN_{m,n,k}^{(\C)}$ (that is, its intersection with $[\Mat_{m \times n}(\R)]^k$)
contains $\cN_{m,n,k}^{(\R)}$.
\end{prop}

The proof is an easy exercise.



\subsection{Universal regularity for constant inputs and richness}\label{ss.characterization}

In this subsection we prove \cref{p.reg transitivity};
in fact we prove a more precise result, and also fix some notation.

Given a linear operator $H \colon E \to E$, 
where $E$ is a finite-dimensional vector space over the field $\K$,
and vectors $v_1$, \dots, $v_m \in E$, 
we denote by $\sorb^N_H(v_1,\dots,v_m)$ the space spanned by 
the family of vectors $H^t(v_i)$, where $1\le i \le m$ and $0 \le t < N$.
In other words,  $\sorb^N_H(v_1,\dots,v_m)$ is the reachable set from $0$ of the linear
control system 
$$
\xi_{t+1} = H\xi_t + \sum_i u_{t,i} v_i \, .
$$

The sequence of spaces $\sorb^N_H(v_1,\dots,v_m)$ is nested nondecreasing, and
thus stabilize to a space $\sorb_H(v_1,\dots,v_m)$ 
after $N \le \dim H$ steps.


%

\medskip


If $A \colon \cU \to \GL(d,\C)$ is a differentiable map
then the \emph{normalized derivative} of $A$ at a point $u$
is the linear map $T_u \cU \to \gl(d,\R)$
given by 
$h \mapsto (DA(u)\cdot h)\circ A^{-1}(u)$.

Let $\phi_N(\xi_0,\hat{u})$ be the state $\xi_N\in \KP^d$ of the system~\eqref{e.proj semilin CS} 
determined by the initial state $\xi_0$ and the input sequence $\hat{u}\in \cU^N$. 
Let $\partial_2 \phi_N(\xi_0,\hat{u})$ be the derivative of the map $\phi_N(\xi_0, \cdot)$ at $\hat{u}$.

Fix a constant input $\hat{u}=(u,\ldots, u)\in \cU^N$, 
and local coordinates on $\cU$ around~$u$. 
Let $B_j$ be the normalized partial derivatives of the map $A$ at $u$ with respect to the $i^\mathrm{th}$ coordinate. 
Consider the datum $\bA=(A,B_1,\ldots, B_m)$, where $A = A(u)$.
Define the following subspace of $\gl(d,\K)$:
\begin{equation}\label{e.Lambda_N}
\Lambda_N(\bA) =  \sorb_{\Ad_A}^N (\Id, B_1, \dots, B_m) \, ,
\end{equation}
where $\Ad_A(B) = ABA^{-1}$.

\begin{prop}\label{p.ranks}
For all $\xi_0\in \KP^{d-1}$ and any $x_0\in \K^d\setminus \{0\}$ representing $\xi_0$, 
$$
\rank \partial_2 \phi_N (\xi_0,\hat{u}) = \dim\left[\Lambda_N(\bA)\cdot (A^N x_0)\right]-1.
$$
\end{prop}

In particular (since $A =A(u)$ is invertible), 
the input $\hat{u}$ is universally regular 
if and only if $\Lambda_N(\bA)$ is a transitive space,
which is the statement of \cref{p.reg transitivity}.

\begin{proof}[Proof of \cref{p.ranks}]
Let $\xi_0 = [x_0]$, where $x_0\in \K^d_*$.
Let $\psi_N(x_0,\hat{u})$ be the final state 
of the non projectivized system~\eqref{e.semilin CS}  determined by the initial state $x_0$ 
and by the sequence of controls $\hat{u}\in \cU^N$.
Using local coordinates with $u$ in the origin,
we have the following first order approximation for $\hat{u}\simeq 0$:
\begin{align*}
\psi_N(x_0,\hat{u}) &\simeq A^N x_0 + 
\sum_{1\leq j\leq m \atop 0 \leq t<N}u_{t,j}A^{N-t-1}B_{j} A^{t+1} x_0 \\
&= 
\left(\Id + \sum_{1\leq j\leq m \atop 0 \leq n < N} u_{N-1-n,j}\Ad_A^n(B_{j})\right) x_N \, ,
\end{align*}
where $x_N = \psi_N(x_0,0) = A^N x_0$.  
Therefore the image of 
$\partial_2 \psi_N(x_0,\hat{u})$
is the following subspace of $T_{A^N x_0} \K^d$:
\begin{align*}
V = \left(\spa_{1\leq j\leq m\atop 0\leq n<N}\Ad_A^n B_{j} \right) \cdot x_N ,
\end{align*}

The image of $\partial_2 \phi_N(\xi_0,\hat{u})$
equals $D\pi(x_N)(V)$, where $\pi: \K^d_* \to \KP^{d-1}$ is the canonical projection.
Notice that $\Ker D\pi(x) = \K x$ for any $x \in \K^d_*$.
It follows that 
\begin{multline*}
\rank \partial_2 \phi_N(\xi_0,\hat{u})
 = \dim \left[ D\pi(x_N) (V) \right]
\\
 = \dim \left[ D\pi(x_N) \big(\K x_N + V\big) \right]
 = \dim [\K x_N + V] - 1
\end{multline*}
Since $\K x_N + V = \Lambda_N(\bA)\cdot x_N$,
the \lcnamecref{p.ranks} is proved.
\end{proof}

\subsection{The sets of poor data}\label{ss.poor_set}

For emphasis, we repeat the definition already given in the introduction:
The datum $\bA = (A, B_1, \dots, B_m) \in \GL(d,\K) \times [\gl(d,\K)]^m$
is \emph{rich} if the space $\Lambda(\bA) = \Lambda_{d^2}(\bA)$  is transitive,
and \emph{poor} otherwise.
The concept in fact depends on the field under consideration.
The set of such poor data is denoted by $\cP_{m,d}^{(\K)}$.

It follows immediately from \cref{p.NT algebraicity} that 
$\cP_{m,d}^{(\R)}$ is a closed and semialgebraic subset of $\GL(d,\R) \times [\gl(d,\R)]^m$
and $\cP_{m,d}^{(\C)}$ is an algebraic subset of $\GL(d,\C) \times [\gl(d,\C)]^m$.
This proves part of \cref{t.cod_data_R,t.cod_data_C}.

Also, by \cref{p.NT inclusion} the real poor data are contained in the real part of the complex poor data, i.e.,
\begin{equation}\label{e.RC_inclusion}
\cP_{m,d}^{(\R)} \subset \cP_{m,d}^{(\C)} \cap \big[ \GL(d,\R) \times [\gl(d,\R)]^m \big] \, .
\end{equation}

\medskip

For later use, we note that the sets of poor data
are saturated in the sense of the following definition: 
A set $\cZ \subset [\Mat_{d\times d}(\K)]^{1+m}$
will be called \emph{saturated}
if $(A, B_1, \dots, B_m) \in \cZ$ implies that:
$(A, B_1, \dots, B_m) \in \cZ$ implies that:
\begin{itemize}
\item 
for all $P \in \GL(d,\K)$ we have $(P^{-1}AP, P^{-1}B_1 P, \dots, P^{-1}B_m P) \in \cZ$;
\item 
for all $Q = (q_{ij}) \in \GL(m,\K)$, letting $B'_i = \sum_j q_{ij} B_j$,
we have $(A, B_1', \dots, B_m') \in \cZ$.
\end{itemize}


\subsection{The easy codimension inequality of \cref{t.cod_data_R,t.cod_data_C}}
\label{ss.cod_data_easy_half}

Here we will discuss the simplest examples of poor data.

To begin, notice that if $A \in \GL(d,\C)$ is diagonalizable
then so is $\Ad_A$.
Indeed, assume without loss of generality that $A = \Diag(\lambda_1, \dots, \lambda_d)$.
Consider the basis $\{E_{i,j} ; \; i,j \in \{1,\dots,d\}\} $ of $\gl(d,\C)$,
where 
\begin{equation}\label{e.Eij}
\text{$E_{i,j}$ is the matrix whose only nonzero entry is a $1$ in the $(i,j)$ position.}
\end{equation}
Then $\Ad_A (E_{i,j}) = \lambda_i \lambda_j^{-1} E_{i,j}$.
So if $f$ is a polynomial and $B = (b_{ij})$ then 
\begin{equation}\label{e.poly entry}
\text{the $(i,j)$-entry of the matrix $(f(\Ad_A))(B)$ is  
$f(\lambda_i\lambda_j^{-1})b_{ij}$.}
\end{equation}

\medskip

The datum $\bA = (A, B_1, \dots, B_m) \in \GL(d,\K) \times \gl(d,\K)^m$
is called \emph{conspicuously poor}
if there exists a change of bases $P \in \GL(d,\K)$ such that:
\begin{itemize}
\item 
the matrix $P^{-1} A P$ is diagonal;
\item 
the matrices $P^{-1} B_k P$ have a zero entry in a common off-diagonal position;
more precisely, 
there are indices $i_0$, $j_0 \in \{1,\dots,d\}$ with $i_0 \neq j_0$ 
such that for each $k \in \{1,\dots,m\}$, the $(i_0,j_0)$ entry of the matrix $P^{-1} B_k P$ vanishes.
\end{itemize}
(As in the definition of poorness, the concept depends on the field $\K$.)

\begin{lemma}\label{l.easy_poor_data}
Conspicuously poor data are poor.
\end{lemma}

\begin{proof}
Let $\bA = (A, B_1, \dots, B_m)$ be conspicuously poor.
With a change of basis we can assume that $A$ is diagonal.
Let $(e_1,\dots,e_d)$ be the canonical basis of $\K^d$.
Let $(i,j)$ be the entry position where all $B_i$'s have a zero entry. 
By \eqref{e.poly entry}, all matrices in the space 
$\Lambda(\bA) = \sorb_{\Ad_A}(\Id, B_1, \dots, B_m)$ have a zero entry in the $(i_0,j_0)$ position.
In particular, there is no $L \in \Lambda(\bA)$ such that $L \cdot e_{j_0} = e_{i_0}$,
showing that this space is not transitive.
\end{proof}

The converse of this \lcnamecref{l.easy_poor_data} is certainly false.
(Many examples appear in \cref{a.dim 1}; see also \cref{ex.toroidal toeplitz}.)
However, we will see in \cref{ss.unconstrained} that the converse holds for generic $A$.

\medskip

We will use \cref{l.easy_poor_data} to prove the easy codimension inequalities for \cref{t.cod_data_R,t.cod_data_C};
first we need to recall the following:

\begin{prop}\label{p.eigen_smooth}
Suppose $A \in \Mat_{d\times d}(\K)$ is diagonalizable over $\K$ and with simple eigenvalues only.	
Then there is a neighborhood of $A$ where the eigenvalues vary smoothly, 
and where the eigenvectors can be chosen to vary smoothly.
\end{prop}

\begin{prop}[Easy half of \cref{t.cod_data_R,t.cod_data_C}]\label{p.cod_data_easy_half}
For both $\K =\R$ or $\C$, we have 
$\codim_{\K} \cP^{(\K)}_m \le m$.
\end{prop}

\begin{proof}
Using \cref{p.eigen_smooth},
we can exhibit smoothly embedded disks of codimension $m$ 
inside $\GL(d,\K) \times \gl(d,\K)^m$
formed by conspicuously poor data.
\end{proof}

\section{Rigidity}\label{s.rig}

The aim of this section is to state \cref{t.rig} and prove its first part.
Along the way we will establish several lemmas which will be reused 
in the proof of the second part of the theorem in \cref{s.rig proof}.

\subsection{Acyclicity}\label{ss.acyclicity}

Consider a linear operator $H \colon  E \to E$, 
where $E$ is a finite-dimensional complex vector space.
The \emph{acyclicity} of $H$ is defined as the least number 
$n$ of vectors $v_1$, \dots, $v_n \in E$
such that $\sorb_H(v_1, \ldots, v_n) = E$.
We denote $n = \acyc H$.
If $n = 1$ then $H$ is called a \emph{cyclic operator},
and $v_1$ is called a \emph{cyclic vector}.

\begin{lemma}\label{l.sum}
Let $E$ be a finite-dimensional complex vector space and let 
$H \colon E \to E$ be a linear operator.
Assume that $E_1$, \ldots, $E_k \subset E$ are $H$-invariant subspaces
and that the spectra of $A|E_i$ ($1 \le i \le k$) are pairwise disjoint.
If $v_1 \in E_1$, \ldots, $v_k \in E_k$ then  
$$
\sorb_H(v_1, \ldots, v_k) = \sorb_H(v_1 + \cdots + v_k) \, .
$$
\end{lemma}

\begin{proof}
View $E$ as a module over the ring of polynomials $\C[x]$ by defining $xv=H(v)$ for $v \in E$.
Then the lemma follows from \cite[Theorem~6.4]{Roman}.
\end{proof}

\medskip

The \emph{geometric multiplicity} of an eigenvalue $\lambda$ of $H$ is the 
dimension of the kernel of $H - \lambda \Id$
(or, equivalently, the number of corresponding Jordan blocks).

\begin{prop}\label{p.acyc}   
The acyclicity of an operator equals the maximum of the geometric multiplicities of its eigenvalues.
\end{prop}

\begin{proof} 
This follows from the Primary Cyclic Decomposition Theorem 
together with \cref{l.sum}.
%
%
\end{proof}

\begin{rem}\label{r.conjugacy_class}
The operators which interest us most are $H = \Ad_A$, where $A \in \GL(d,\C)$.
It is useful to observe that \emph{the geometric multiplicity of $1$ 
as an eigenvalue of $\Ad_A$ equals the codimension of the conjugacy class of $A$ inside $\GL(d,\C)$}.
To prove this, consider the map $\Psi_A \colon \GL(d,\C) \to \GL(d,\C)$ given by 
$\Psi_A(X) = \Ad_X(A)$.
The derivative at $X=\Id$ is $H \mapsto HA-AH$; so
$\Ker D\Psi_A (\Id) = \Ker (\Ad_A - \id)$.
Therefore when $X=\Id$, the rank of $D\Psi_A (X)$ equals 
the geometric multiplicity of $1$ as an eigenvalue of $\Ad_A$.
To see that this is true for any $X$, notice that $\Psi_A = \Psi_{\Ad_X(A)} \circ R_{X^{-1}}$
(where $R$ denotes a right-multiplication diffeomorphism of $\GL(d,\C)$).

We will see later (\cref{l.acyc and pop1}) that $1$ is the eigenvalue of $\Ad_A$
with the biggest geometric multiplicity.
By \cref{p.acyc}, we conclude that $\acyc \Ad_A$ equals 
the codimension of the conjugacy class of $A$.
\end{rem}

\subsection{Definition of rigidity, and the main rigidity estimate}

Let $E$ and $F$ be finite-dimensional complex vector spaces.
Let $H$ be a linear operator action on the space $\mathcal{L}(E,F)$.
We define the \emph{rigidity} of $H$, denoted $\rig H$,
as the least $n$ such that there exist $L_1$, \dots, $L_n \in \mathcal{L}(E,F)$ so that
$\sorb_H(L_1 , \dots, L_n)$ is transitive.
Therefore
$$
1 \le \rig H \le \acyc H \, .
$$

For technical reasons, 
we also define a \emph{modified rigidity} of $H$, 
denoted $\rig_+ H$.
The definition is the same, with the difference that if $E = F$ then 
$L_1$ is required to be the identity map in $\mathcal{L}(E,E)$.
Of course,
$$
\rig H \le \rig_+ H \le \rig H + 1.
$$  

\medskip

We want to give a reasonably good estimate 
of the modified rigidity of $\Ad_A$ for any fixed $A \in \GL(d,\C)$.
(This will be achieved in \cref{l.rig world}.)
We assume that $d \ge 2$; so $\rig_+ \Ad_A \ge 2$.
The next example shows that ``most'' matrices $A$ have the lowest possible $\rig_+ \Ad_A$.

\begin{example}\label{ex.unconstrained_rig}
If $A\in \GL(d,\C)$ is unconstrained (see \cref{ss.unconstrained})
then $\rig_+ \Ad_A = 2$.
Indeed if we take a matrix $B \in \gl(d,\C)$ whose expression in the base that diagonalizes $A$
has no zeros off the diagonal then, by \cref{l.easy_fiber}, 
$\Lambda(A,B) = \sorb_{\Ad_A}(\Id,B)$ is rich.

More generally, if $A\in \GL(d,\C)$ is little constrained (see \cref{a.dim 1})
then it follows from \cref{p.rich_pair} that
$\rig_+ \Ad_A = 2$.
\end{example}

\begin{example}\label{ex.toroidal toeplitz}
Consider $A = \Diag (1, \alpha, \alpha^2)$ where $\alpha = e^{2\pi i /3}$.
(In the terminology of \cref{ss.unconstrained}, 
$A$ has constraints of type 1.) 
Since $\Ad_A^3$ is the identity, we have $\dim \sorb_{\Ad_A}(\Id,B) \le 4$ for any $B \in \gl(3,\C)$.
By the result of Azoff \cite{Azoff} already mentioned at \cref{ex.toeplitz}, 
the minimum dimension of a transitive subspace of $\gl(3,\C)$ is $5$.
This shows that $\rig_+ \Ad_A \ge 3$.
(Actually, equality holds, as we will see in \cref{ex.toroidal toeplitz again} below.)
\end{example}

\medskip

Let $T$ be the set of roots of unity.
Define an equivalence relation $\asymp$ on the set $\C^*$ of nonzero complex numbers by:
\begin{equation}\label{e.equiv rel}
\lambda \asymp \lambda' \ \Leftrightarrow \ \lambda / \lambda' \in T.
\end{equation}
We also say that $\lambda$, $\lambda'$ are \emph{equivalent mod~$T$}.

For $A \in \GL(d,\C)$, we denote
\begin{equation}\label{e.num classes}
c(A) := \text{number of different classes mod $T$ of the eigenvalues of $A$.}
\end{equation}

We now state a technical result
which has a central role in our proofs,
as explained informally in \cref{ss.overview}:

\begin{thm}\label{t.rig}
Let $d \ge 2$ and $A \in \GL(d,\C)$.
Then:
\begin{enumerate}
\item\label{i.rig easy}
If $c(A) = d$ then $\rig_+ \Ad_A = 2$.
\item\label{i.rig hard}
If $c(A) < d$ then $\rig_+ \Ad_A \le \acyc \Ad_A - c(A) + 1$.
\end{enumerate}
\end{thm}


\begin{rem}
When $c(A) = d$, we have $\acyc \Ad_A = d$ (this will follow from \cref{l.acyc and pop1});
so the conclusion of part~\ref{i.rig hard} does not hold in this case.
\end{rem}

\begin{rem}
The conditions of $A$ being unconstrained and $A$ having $c(A)=d$ 
both mean that $A$ is ``non-degenerate''.
Both of them imply small rigidity, according to \cref{ex.unconstrained_rig}
and  part~\ref{i.rig easy} of \cref{t.rig}.
It is important, however, not to confuse the two properties;
in fact, none implies the other.
\end{rem}

\begin{example}\label{ex.toroidal toeplitz again}
Consider again $A$ as in \cref{ex.toroidal toeplitz}.
The eigenvalues of $\Ad_A$ are $1$, $\alpha$, and $\alpha^2$, each with multiplicity $3$;
so \cref{p.acyc} gives $\acyc \Ad_A = 3$.
So \cref{t.rig} tell us that $\rig_+ \Ad_A \le 3$,
which is actually sharp.
\end{example}

\medskip

The proof of part~\ref{i.rig easy} of \cref{t.rig} will be given in \cref{ss.rig easy}
after a few preliminaries (\cref{ss.sudoku,ss.preorder}).
These preliminaries are also used in the proof of the harder part~\ref{i.rig hard},
which will be given in \cref{s.rig proof}.

\subsection{A criterion for transitivity} \label{ss.sudoku}

We will show the transitivity of 
certain spaces of matrices that remotely resemble Toeplitz matrices.  

Let $t$ and $s$ be positive integers.
Let $\mathcal{R}_1$ be a partition of the interval $[1,t] = \{1, \ldots, t\}$ into intervals,
and let $\mathcal{R}_2$ be a partition of $[1,s]$ into intervals.
Let $\mathcal{R}$ be the product partition.
We will be interested in matrices of the following special form:
\begin{equation}\label{e.sudoku}
M=(m_{i,j})_{1\leq i\leq  t \atop 1\leq j\leq s }=
\left(
\begin{array}{ccc|c|ccc}
&   & &   &  &   & \\
& * & & 0 &  & 0 & \\
&   & &   &  &   & \\
\hline 
& 0 & &M_\tR & & 0 & \\
\hline
&   & &   &  &   & \\
& 0 & & 0 &  & * & \\
&   & &   &  &   & 
\end{array}
\right),
\end{equation}
where $\tR$ is an element of the product partition $\mathcal{R}$, and $M_{\tR}$ is the submatrix 
$(m_{i,j})_{(i,j)\in \tR}$. 

Let $\Lambda$ be a vector space of $t \times s$ matrices.
For each $\tR \in \mathcal{R}$, say of size $k \times \ell$, 
we define the following space of matrices:
\begin{equation}\label{e.sudoku space}
\Lambda^{[\tR]} = \big\{
N \in \Mat_{k \times \ell}(\C) \; ;
\exists \ M \in \Lambda
\text{ of the form \eqref{e.sudoku} with $M_\tR = N$} 
\big\}.
\end{equation}
We regard $\Lambda$ as a subspace of $\mathcal{L}(\C^s,\C^t)$.
If the rectangle $\tR$ is $[p, p+k-1] \times [q, q+\ell-1]$,
we regard the space $\Lambda^{[\tR]}$ as a subspace of  
$$
\mathcal{L} \big( \{0\}^{q-1} \times \C^\ell \times \{0\}^{t-q-\ell+1}, \{0\}^{p-1} \times \C^k \times \{0\}^{t-p-k+1} \big).
$$ 

\begin{lemma}\label{l.sudoku}
Assume that $\Lambda^{[\tR]}$ is transitive for each $\tR \in \mathcal{R}$.
Then $\Lambda$ is transitive.
\end{lemma}

An interesting feature of the lemma which will be useful later is that 
it can be applied recursively.
Before giving the proof of the lemma, 
we illustrate its usefulness by showing the transitivity
of generalized Toeplitz spaces:

\begin{proof}[Proof of \cref{ex.gen Toeplitz}]
Consider the partition of $[1,d]^2$ into $1 \times 1$ ``rectangles''.
If $\Lambda$ is a generalized Toeplitz space then 
$\Lambda^{[\tR]} = \Mat_{1\times 1}(\C) = \C$ for each rectangle $\tR$.
These are transitive spaces, so \cref{l.sudoku} implies that $\Lambda$
is transitive.
\end{proof}

Before proving \cref{l.sudoku}, notice the following dual characterization of transitivity, whose proof is immediate:

\begin{lemma}\label{l.duality}
A subspace $\Lambda \subset \mathcal{L}(\C^s,\C^t)$ is transitive iff
for any non-zero vector $u \in \C^s$ and any non-zero linear functional $\phi \in (\C^t)^*$
there exists $M \in \Lambda$ such that $\phi (M \cdot u) \neq 0$.
\end{lemma}

\begin{proof}[Proof of \cref{l.sudoku}]
Take any non-zero vector $u = (u_1,\ldots,u_s)$ in $\C^s$
and a non-zero functional $\phi(v_1,\ldots,v_t) = \sum_{i=1}^t \phi_i v_i$ in $(\C^t)^*$.
By \cref{l.duality}, we need to show that there exists $M = (x_{ij}) \in \Lambda$ such that 
\begin{equation}\label{e.sum}
\phi (M \cdot u) = \sum_{i=1}^t \sum_{j=1}^s \phi_i x_{ij} u_j
\end{equation}
is non-zero.

Let $j_0$ be the least index such that $u_{j_0} \neq 0$,
and let $i_0$ be the greatest index such that $\phi_{i_0} \neq 0$.
Let $\tR$ be the element of $\mathcal{R}$ that contains $(i_0,j_0)$. 
Notice that if $M$ is of the form \eqref{e.sudoku} then
the $(i,j)$-entries of $M$ that are above left (resp.\ below right) of $\tR$
do not contribute to the sum \eqref{e.sum}, because $u_j$ (resp.\ $\phi_i$) vanishes.
That is,
$\phi (M \cdot u)$ depends only on $M_{\tR}$ 
and is given by $\sum_{(i,j)\in \tR} \phi_i x_{ij} u_j$;
Since $\Lambda^{[\tR]}$ is transitive, by \cref{l.duality}
there is a choice of a matrix $M \in \Lambda$ of the form \eqref{e.sudoku}
so that $\phi (M \cdot u) \neq 0$.
So we are done.
\end{proof}

\subsection{Preorder in the complex plane}\label{ss.preorder}


We consider the set $\C_* / T$ of equivalence classes 
of the relation~\eqref{e.equiv rel}.
Since $T$ is the torsion subgroup of $\C_*$,
the quotient $\C_* / T$ is an abelian torsion-free group.

\begin{prop}
There exists a multiplication-invariant total order $\preccurlyeq$  on $\C_* / T$. 
\end{prop}

The proposition follows from a result of Levi~\cite{Levi},
but nevertheless let us give a direct proof:

\begin{proof}
There is an isomorphism between $\R \oplus (\R / \Q)$ and $\C_*/T$, namely
$(x,y) \mapsto \exp(x + 2\pi i y)$. 
So it suffices to find a multiplication-invariant order in $\R/\Q$
(and then take the lexicographic order).
Take a Hamel basis $B$ of the $\Q$-vector space $\R$ so that $1 \in B$.
Then $\R/\Q$ is a direct sum of abelian groups $\bigoplus_{x\in B, \ x \neq 1} x\Q$.
Order each $x\Q$ in the usual way and take any total order on $B$.
Then the induced lexicographic order on $\R/\Q$ is multiplication-invariant,
and the proof is concluded.
\end{proof}

Let $[z]\in \C_*/T$ denote the equivalence class of $z\in \C_*$.
Let us extend the notation, writing $z \preccurlyeq z'$ if $[z] \preccurlyeq [z']$.
Then $\preccurlyeq$ becomes a multiplication-invariant total preorder on $\C_*$ 
that induces the equivalence relation $\asymp$.
In other words, for all $z$, $z'$, $z'' \in \C_*$ we have:
\begin{itemize}
	\item $z \preccurlyeq z'$ or $z' \preccurlyeq z$;
	\item $z \preccurlyeq z'$ and $z' \preccurlyeq z$ $\Longleftrightarrow$ $z \asymp z'$;
	\item $z \preccurlyeq z'$ and $z' \preccurlyeq z''$ $\Longrightarrow$ $z \preccurlyeq z''$;
	\item $z \preccurlyeq z'$ $\Longrightarrow$ $z z'' \preccurlyeq z' z''$.
\end{itemize}
It follows that:
\begin{itemize}
	\item $z \preccurlyeq z'$ $\Longrightarrow$ $(z')^{-1} \preccurlyeq z^{-1}$.
\end{itemize}
We write $z \prec z'$ when $z \preccurlyeq z'$ and $z \not \asymp z'$.

\subsection{Proof of the easy part of \cref{t.rig}}\label{ss.rig easy}

\begin{proof}[Proof of part~\ref{i.rig easy} of \cref{t.rig}]
If $c(A)=d$ then in particular all eigenvalues are different and so the matrix $A$ is diagonalizable.
So with a change of basis we can assume that
$A = \Diag(\lambda_1, \dots, \lambda_d)$.
We can also assume that the eigenvalues are 
increasing with respect to the preorder introduced in \cref{ss.preorder}:
$$
\lambda_1 \prec \lambda_2 \prec \cdots \prec \lambda_d \, .
$$

Fix any matrix $B$ with only nonzero entries,
and consider the space $\Lambda = \sorb_{\Ad_A}(B)$,
which is described by \eqref{e.poly entry}.
We will use \cref{l.sudoku} to show that $\Lambda$ is transitive.
Let $\mathcal{R}$ be the partition of $[1,d]^2$ into $1 \times 1$ rectangles.
Given a cell $\tR = \{(i_0,j_0)\} \in \mathcal{R}$ and a coefficient $t \in \C$,
there exists a polynomial $f$ such that 
$f(\lambda_i \lambda_j^{-1})$
equals $t$ if $\lambda_i \lambda_j^{-1} = \lambda_{i_0} \lambda_{j_0}^{-1}$
and equals $0$ otherwise.
Because the eigenvalues are ordered,
$M = f(\Ad_A)\cdot B$ is a matrix in $\Lambda$ of the form \eqref{e.sudoku}.
Also, $M_{\tR} = (t)$.
So $\Lambda^{[\tR]} = \C$, which is transitive.
This shows that $\rig \Ad_A = 1$, and  $\rig_+ \Ad_A \leq 2$. 
Thus, as $d\geq 2$, we have $\rig_+ \Ad_A = 2$.
\end{proof}

\section{Proof of the hard part of the rigidity estimate}\label{s.rig proof}

This section is wholly devoted to proving part~\ref{i.rig hard} of \cref{t.rig}.
In the course of the proof we need to introduce some terminology 
and to establish several intermediate results.
None of these are used in the rest of the paper, apart form a simple consequence,
which is \cref{r.jordan type and acyc}.

\subsection{The normal form}\label{ss.normal form}

Let $A \in \GL(d,\C)$.
In order to describe the estimate on $\rig_+ \Ad_A$, 
we need to put $A$ in a certain normal form, which we now explain. 
Fix a preorder $\preccurlyeq$ on $\C_*$ as in \cref{ss.preorder}.

\medskip

List the eigenvalues of $A$ without repetitions as
\begin{equation}\label{e.order lambdas}
\lambda_1 \preccurlyeq \cdots \preccurlyeq \lambda_r
\end{equation}
Write each eigenvalue in polar coordinates:
$$
\lambda_k = \rho_k \exp ( \theta_k \sqrt{-1}), \quad
\text{where } \rho_k > 0 \text { and } 0 \leq \theta_k <  2 \pi.
$$
Up to reordering, we may assume
$$
\left.
\begin{array}{l}
\lambda_k \asymp \lambda_\ell \\
k < \ell
\end{array}
\right\} \ \Rightarrow \ 
\theta_k < \theta_\ell \, .
$$

With a change of basis, we can assume that 
$A$ has Jordan form: 
\begin{equation}\label{e.order blocks}
A =
\begin{pmatrix}
A_1  &        &     \\
     & \ddots &     \\
     &        & A_r      
\end{pmatrix}, \quad 
A_k =
\begin{pmatrix}
J_{t_{k,1}}(\lambda_k) &        &                         \\
                       & \ddots &                         \\
                       &        & J_{t_{k,\tau_k}}(\lambda_k)      
\end{pmatrix},
\end{equation}
where $t_{k,1} + \cdots + t_{k,\tau_k} = s_k$ is the multiplicity of the eigenvalue $\lambda_k$,
and $J_t(\lambda)$ is the following $t \times t$ Jordan block:
\begin{equation}\label{e.block}
J_t(\lambda) = 
\left(
\vcenter{
\xymatrix @-1.5pc {
\lambda \ar@{.}[rrdd]  & 1 \ar@{.}[rd] &         \\
                 &               & 1       \\
                 &               & \lambda \\
}}
\right).
\end{equation}
The matrix $A$ will be fixed from now on.

\subsection{Rectangular partitions}\label{ss.geo defs}

This subsection contains several definitions that will be fundamental in all 
arguments until the end of the section.
We will define certain subregions of the set $\{1,\dots,d\}^2$ of matrix entry positions
that depend on the normal form of the matrix $A$.
Later we will see they are related to $\Ad_A$-invariant subspaces.
Those regions will be $\rm c$-rectangles, $\rm e$-rectangles, and $\rm j$-rectangles 
(where $\rm c$ stands for classes of eigenvalues, $\rm e$ for eigenvalues and $\rm j$ for Jordan blocks).
Regions will have some numerical attributes (banners and weights) coming from their geometry and from the eigenvalues of $A$ they will be associated to. Those attributes will be related to 
numerical invariants of $\Ad_A$ (eigenvalues and geometric multiplicities), but we use different names so that we remember their geometric meaning and so that they are not mistaken for the corresponding invariants of $A$. We also introduce positional attributes of the regions (arguments and latitudes) which will be useful fundamental later in the proofs of our rigidity estimates.

\medskip

Recall $A$ is a matrix in normal form as explained in \cref{ss.normal form}.
Define three partitions $\cP_{\rm c}$, $\cP_{\rm e}$, $\cP_{\rm j}$ of the set $[1,d] = \{1,\dots,d\}$ into intervals:
\begin{itemize}
\item The partition $\cP_{\rm c}$ corresponds to equivalence classes of eigenvalues 
under the relation $\asymp$, that is, 
the right endpoints of its atoms are the numbers 
$s_1 + \dots + s_k$
where $k=r$ or $k$ is such that $\lambda_k \prec \lambda_{k+1}$.
\item The partition $\cP_{\rm e}$ corresponds to eigenvalues:
the right endpoints of its atoms are the numbers 
$s_1 + \dots + s_k$, where $1 \le k \le r$.
So $\cP_{\rm e}$ refines $\cP_{\rm c}$.
\item The partition $\cP_{\rm j}$ corresponds to Jordan blocks:
the right endpoints of its atoms are the numbers 
$s_1 + \dots + s_{k-1} + t_{k,1} + \dots + t_{k,\ell}$, where $1 \le k \le r$
and $1 \le \ell \le \tau_k$.
So $\cP_{\rm j}$ refines $\cP_{\rm e}$.
\end{itemize}
For $*= \rm c$, $\rm e$, $\rm j$, let 
$\cP_*^2$ be the partition of the square $[1,d]^2$ 
into rectangles that are products of atoms of $\cP_*$.
The elements of $\cP_{\rm c}^2$ are called \emph{$\rm c$-rectangles}, 
the elements of $\cP_{\rm e}^2$ are called \emph{$\rm e$-rectangles},
and elements of $\cP_{\rm j}^2$ are called \emph{$\rm j$-rectangles}.
Thus the square $[1,d]^2$ is a disjoint union $\rm c$-rectangles,
each of them is a disjoint union of $\rm e$-rectangles,
each of them is a disjoint union of $\rm j$-rectangles.

\begin{example}\label{ex.big matrix}
Suppose $d=17$,
$A$ has $r=5$ eigenvalues 
$$
\lambda_1 =  \exp { \tfrac{1}{2}\pi i} , \quad
\lambda_2 =  \exp { \tfrac{7}{6}\pi i} , \quad
\lambda_3 =  \exp {\tfrac{11}{6}\pi i} , \quad
\lambda_4 = 2\exp { \tfrac{1}{6}\pi i} , \quad
\lambda_5 = 2\exp { \tfrac{5}{6}\pi i} 
$$
with $\lambda_1 \asymp \lambda_2 \asymp \lambda_3 \prec \lambda_4 \asymp \lambda_5$
and  respective Jordan blocks of sizes 
$4$, $2$, $1$; $3$, $2$;  $2$;  $2$; $1$. 
Then there are $4$ $\rm c$-rectangles, $25$ $\rm e$-rectangles, and $64$ $\rm j$-rectangles.
See \cref{f.big matrix}. 
\begin{figure}[hbt]  
\setlength{\unitlength}{.6cm}
\begin{picture}(17,17)(0,-1)  

\thicklines  
\put(0,-1){\framebox(17,0)} 
\put(0,2){\framebox(17,0)}
\put(0,16){\framebox(17,0)} 
\put(0,-1){\framebox(0,17)}  
\put(14,-1){\framebox(0,17)}
\put(17,-1){\framebox(0,17)} 

\thinlines
\put(0,0){\line(1,0){17}}
\put(0,4){\line(1,0){17}}
\put(0,9){\line(1,0){17}}
\put(7,-1){\line(0,1){17}}
\put(12,-1){\line(0,1){17}}
\put(16,-1){\line(0,1){17}}

\put(0,6){\dashbox{.2}(17,0)}
\put(0,10){\dashbox{.2}(17,0)}
\put(0,12){\dashbox{.2}(17,0)}
\put(4,-1){\dashbox{.2}(0,17)}
\put(6,-1){\dashbox{.2}(0,17)}
\put(10,-1){\dashbox{.2}(0,17)}


\put( 7.2,13.1){\tiny wgt.~$3$}
\put( 7.2,12.6){\tiny lat.~$0$}

\put(10.2,13.1){\tiny wgt.~$2$}
\put(10.2,12.6){\tiny lat.~$1$}

\put( 7.2,11.1){\tiny wgt.~$2$}
\put( 7.2,10.6){\tiny lat.~$-1$}

\put(10.2,11.1){\tiny wgt.~$2$}
\put(10.2,10.6){\tiny lat.~$0$}

\put( 7.2, 9.6){\tiny wgt.~$1$}
\put( 7.2, 9.1){\tiny lat.~$-2$}

\put(10.2, 9.6){\tiny wgt.~$1$}
\put(10.2, 9.1){\tiny lat.~$-1$}

\put(   0.05,15.4){{\footnotesize  $15$\,}\framebox(.5,.35){\tiny $\bullet${\sf a}}\put(-.51,0){\line(0,-1){.3}}}
\put(   7.05, 8.4){{\footnotesize   $9$\,}\framebox(.5,.35){\tiny $\bullet${\sf a}}\put(-.51,0){\line(0,-1){.3}}}
\put(  12.05, 3.4){{\footnotesize   $2$\,}\framebox(.5,.35){\tiny $\bullet${\sf a}}\put(-.51,0){\line(0,-1){.3}}}
\put(   7.05,15.4){{\footnotesize  $11$\,}\framebox(.5,.35){\tiny $\bullet${\sf b}}\put(-.51,0){\line(0,-1){.3}}}
\put(  12.05, 8.4){{\footnotesize   $4$\,}\framebox(.5,.35){\tiny $\bullet${\sf b}}\put(-.51,0){\line(0,-1){.3}}}
\put(   0.05, 3.4){{\footnotesize   $5$\,}\framebox(.5,.35){\tiny $\bullet${\sf b}}\put(-.51,0){\line(0,-1){.3}}}
\put(   0.05, 2.7){{\footnotesize $\ominus$}}
\put(  12.05,15.4){{\footnotesize   $5$\,}\framebox(.5,.35){\tiny $\bullet${\sf c}}\put(-.51,0){\line(0,-1){.3}}}
\put(   0.05, 8.4){{\footnotesize  $11$\,}\framebox(.5,.35){\tiny $\bullet${\sf c}}\put(-.51,0){\line(0,-1){.3}}}
\put(   0.05, 7.7){{\footnotesize $\ominus$}}
\put(   7.05, 3.4){{\footnotesize   $4$\,}\framebox(.5,.35){\tiny $\bullet${\sf c}}\put(-.51,0){\line(0,-1){.3}}}
\put(   7.05, 2.7){{\footnotesize $\ominus$}}

\put(  14.05, 1.4){{\footnotesize   $2$\,}\framebox(.5,.35){\tiny $\bullet${\sf a}}\put(-.51,0){\line(0,-1){.3}}}
\put(  16.05, -.6){{\footnotesize   $1$\,}\framebox(.5,.35){\tiny $\bullet${\sf a}}\put(-.51,0){\line(0,-1){.3}}}
\put(  14.05, -.6){{\footnotesize   $1$\,}\framebox(.5,.35){\tiny $\bullet${\sf c}}\put(-.51,0){\line(0,-1){.3}}}
\put(  15.05, -.6){{\footnotesize $\ominus$}}
\put(  16.05, 1.4){{\footnotesize   $1$\,}\framebox(.5,.35){\tiny $\bullet${\sf b}}\put(-.51,0){\line(0,-1){.3}}}

\put(  14.05,15.4){{\footnotesize   $5$\,}\framebox(.5,.35){\tiny $\mathord{\Uparrow}${\sf a}}\put(-.51,0){\line(0,-1){.3}}}
\put(  14.05,14.7){{\footnotesize $\ominus$}}
\put(  16.05,15.4){{\footnotesize   $3$\,}\framebox(.5,.35){\tiny $\mathord{\Uparrow}${\sf b}}\put(-.51,0){\line(0,-1){.3}}}
\put(  14.05, 8.4){{\footnotesize   $4$\,}\framebox(.5,.35){\tiny $\mathord{\Uparrow}${\sf c}}\put(-.51,0){\line(0,-1){.3}}}
\put(  14.05, 7.7){{\footnotesize $\ominus$}}
\put(  16.05, 8.4){{\footnotesize   $2$\,}\framebox(.5,.35){\tiny $\mathord{\Uparrow}${\sf a}}\put(-.51,0){\line(0,-1){.3}}}
\put(  16.05, 7.7){{\footnotesize $\ominus$}}
\put(  14.05, 3.4){{\footnotesize   $2$\,}\framebox(.5,.35){\tiny $\mathord{\Uparrow}${\sf b}}\put(-.51,0){\line(0,-1){.3}}}
\put(  14.05, 2.7){{\footnotesize $\ominus$}}
\put(  16.05, 3.4){{\footnotesize   $1$\,}\framebox(.5,.35){\tiny $\mathord{\Uparrow}${\sf c}}\put(-.51,0){\line(0,-1){.3}}}
\put(  16.05, 2.7){{\footnotesize $\ominus$}}

\put(   0.05, 1.4){{\footnotesize $5$\,}\framebox(.5,.35){\tiny $\mathord{\Downarrow}${\sf a}}\put(-.51,0){\line(0,-1){.3}}}
\put(   0.05, -.6){{\footnotesize $3$\,}\framebox(.5,.35){\tiny $\mathord{\Downarrow}${\sf c}}\put(-.51,0){\line(0,-1){.3}}}
\put(   1.05, -.6){{\footnotesize $\ominus$}}
\put(   7.05, 1.4){{\footnotesize $4$\,}\framebox(.5,.35){\tiny $\mathord{\Downarrow}${\sf b}}\put(-.51,0){\line(0,-1){.3}}}
\put(   7.05, -.6){{\footnotesize $2$\,}\framebox(.5,.35){\tiny $\mathord{\Downarrow}${\sf a}}\put(-.51,0){\line(0,-1){.3}}}
\put(  12.05, 1.4){{\footnotesize $2$\,}\framebox(.5,.35){\tiny $\mathord{\Downarrow}${\sf c}}\put(-.51,0){\line(0,-1){.3}}}
\put(  12.05, -.6){{\footnotesize $1$\,}\framebox(.5,.35){\tiny $\mathord{\Downarrow}${\sf b}}\put(-.51,0){\line(0,-1){.3}}}

\end{picture}
\caption{\footnotesize The partitions of the square $[1,d]^2$ corresponding to \cref{ex.big matrix}.
Thick (resp., thin, dashed) lines represent $\rm c$-rectangles (resp., $\rm e$-, $\rm j$-)  borders.
Weight and latitude of each $\rm j$-rectangle inside a selected $\rm e$-rectangle are indicated.
The weight of each $\rm e$-rectangle is recorded in its upper left corner, along with 
a symbolic representation of its banner.
There are three banner classes
($\bullet = [1]$, $\mathord{\Downarrow}=[2]$ and $\mathord{\Uparrow}=[1/2]$),
each of them with $3$ different banners. The $\rm e$-rectangles with negative arguments are marked with $\ominus$.
} 
\label{f.big matrix}
\end{figure}
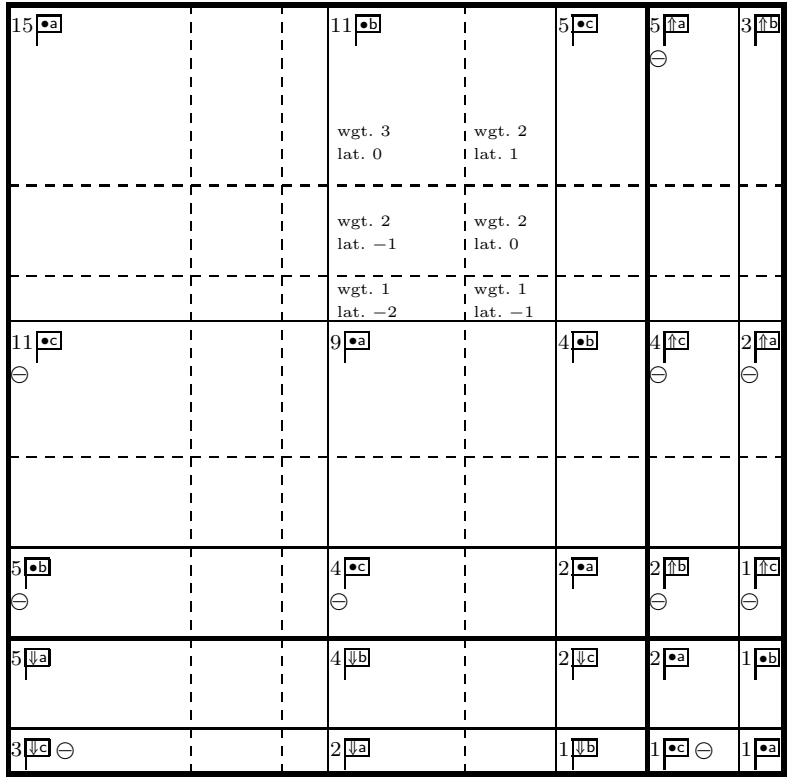
\end{example}

For each $\rm e$-rectangle we define its \emph{row eigenvalue} and its \emph{column eigenvalue}
in the obvious way: 
If an $\rm e$-rectangle $\tE$ equals $I_k \times I_\ell$ where $I_k$ and $I_\ell$
are intervals with right endpoints $s_1 + \dots + s_k$ and $s_1 + \dots + s_\ell$, respectively,
then the row eigenvalue of $\tE$ is $\lambda_k$ and the column eigenvalue of $\tE$ is $\lambda_\ell$.
The row and column eigenvalues of a $\rm j$-rectangle $\tJ$ are defined respectively
as the row and column eigenvalues of the $\rm e$-rectangle that contains it.

Let $\tE$ be an $\rm e$-rectangle with row eigenvalue $\lambda_k$ and column eigenvalue $\lambda_\ell$.
The \emph{banner} of $\tE$ is defined by $\lambda_k^{-1}\lambda_\ell$. The \emph{argument} of the $\rm e$-rectangle is the quantity $\theta_\ell - \theta_k \in (-2\pi,2\pi)$. 
It coincides modulo $2\pi$ with the argument of the banner, but it contains more information than the argument of the banner. 

Each $\rm j$-rectangle $\tJ$ has an address of the type ``$i^\text{th}$ row, $j^\text{th}$ column, $\rm e$-rectangle $\tE$''; then the \emph{latitude} of the $\rm j$-rectangle $\tJ$ within the $\rm e$-rectangle $\tE$ is defined as $j-i$. See an example in \cref{f.big matrix}.

If two $\rm e$-rectangles lie in the same $\rm c$-rectangle then their banners are equivalent mod~$T$.
Thus every $\rm c$-rectangle has a well-defined \emph{banner class} in $\C^*/T$.

If a $\rm j$-rectangle, $\rm e$-rectangle, or $\rm c$-rectangle 
intersects the diagonal $\{(1,1), \dots, (d,d)\}$
then we call it \emph{equatorial}.
Equatorial regions are always square.
Thus every equatorial $\rm e$-rectangle has banner $1$.

The \emph{weight} of a $\rm j$-rectangle is defined as the minimum of its sides.
The weight of a union $R$ of $\rm j$-rectangles in $[1,d]^2$ is defined as
the sum of the weights of those $\rm j$-rectangles. We denote it by $\pop R$.
We can in particular consider the weights of $\rm e$ and $\rm c$-rectangles, and of the complete square $[1,d]^2$.
\medskip

Let us notice some facts on the location of the banners
(which will be useful to apply \cref{l.sudoku}):

\begin{lemma}\label{l.geo}
Let $\tE$ be an $\rm e$-rectangle in a $\rm c$-rectangle $\tC$. Consider the divisions of the square $[1,d]^2$ and the $\rm c$-rectangle $\tC$ as in \cref{f.divisions}. 

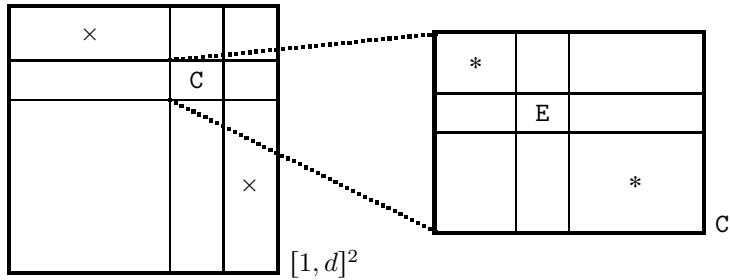
\begin{figure}[hbt]  
\setlength{\unitlength}{.35cm}
\begin{picture}(27, 10)(0,0)

\linethickness{0.4mm}
\put(0,0){\framebox(10,10)} 
\put(10.5,0){$[1,d]^2$}

\put(16,1.5){\framebox(10,7.5)} 
\put(26.5,1.5){$\tC$}


\dottedline{.25}(6,8)(16,9)
\dottedline{.25}(6,6.5)(16,1.5)

\linethickness{0.1mm}
\put(0,0  ){\framebox(6,6.5){$ $}}  
\put(0,6.5){\framebox(6,1.5){$ $}}  
\put(0,8  ){\framebox(6,2  ){$\times$}}     
\put(6,0  ){\framebox(2,6.5){$ $}}  
\put(6,6.5){\framebox(2,1.5){$\tC$}}     
\put(6,8  ){\framebox(2,2  ){$ $}}  
\put(8,0  ){\framebox(2,6.5){$\times$}}     
\put(8,6.5){\framebox(2,1.5){$ $}}  
\put(8,8  ){\framebox(2,2  ){$ $}}  

%

\linethickness{0.1mm}
\put(16,1.5 ){\framebox(3,3.75){$ $}}  
\put(16,5.25){\framebox(3,1.5 ){$ $}}         
\put(16,6.75){\framebox(3,2.25){\large$*$}}      
\put(19,1.5 ){\framebox(2,3.75){$ $}}         
\put(19,5.25){\framebox(2,1.5 ){$\tE$}}       
\put(19,6.75){\framebox(2,2.25){$ $}}         
\put(21,1.5){\framebox(5,3.75){\large$*$}}        
\put(21,5.25){\framebox(5,1.5 ){$ $}}         
\put(21,6.75){\framebox(5,2.25){$ $}}   

\end{picture}
\caption{\footnotesize The divisions of $[1,d]^2$ and $\tC$ in \cref{l.geo}.} 
\label{f.divisions}
\end{figure}

Let $\beta$ be the banner of the $\rm e$-rectangle $\tE$, and let $[\beta]$ be the banner class of the 
$\rm c$-rectangle $\tC$.
Then: 
\begin{enumerate}
\item\label{i.geo0} All the $\rm c$-rectangles with banner class $[\beta]$ are inside the rectangles marked with $\times$.
\item\label{i.geo2} If the $\rm e$-rectangle $\tE$ has nonnegative (resp.\ negative) argument 
then the all the $\rm e$-rectangles with nonnegative (resp.\ negative) argument and with same banner $\beta$ are inside the rectangles marked with $*$.
\end{enumerate}
\end{lemma}

\begin{proof}
In view of the ordering of the eigenvalues \eqref{e.order lambdas},
the banner class increases strictly (with respect to the order $\prec$, of course)
when we move rightwards or upwards to another $\rm c$-rectangle.
So Claim (\ref{i.geo0}) follows.

The argument of an $\rm e$-rectangle takes values in the interval $(-2\pi,2\pi)$.
It increases strictly by moving rightwards or upwards inside $\tC$. 
If two $\rm e$-rectangles in the same $\rm c$-rectangle have both nonnegative or negative argument then 
they have the same banner if and only if they have the same argument.
So Claim (\ref{i.geo2}) follows.
\end{proof}

\subsection{The action of the adjoint of $A$}\label{ss.ad_geo}

Given any $d\times d$ matrix $X = (x_{i,j})$ and 
a $\rm j$-rectangle, $\rm e$-rectangle or $\rm c$-rectangle $\tR = [p,p+t-1] \times [q, q+s-1]$
we define the \emph{submatrix} of $X$ corresponding to $\tR$ 
as $(x_{i,j})_{(i,j) \in \tR}$.
We regard the space of $\tR$-submatrices as 
$\mathcal{L} \big( \{0\}^{q-1} \times \C^s \times \{0\}^{d-q-s+1} , \{0\}^{p-1} \times \C^t \times \{0\}^{d-p-t+1} \big)$,
or as the set of $d \times d$ matrices whose entries outside $\tR$ are all zero.
Such spaces are denoted by $\tR^\square$, and are invariant under $\Ad_A$.
Indeed, if $\tR = \tJ$ is a $\rm j$-rectangle 
then identifying $\tJ^\square$ with $\Mat_{t\times s}(\C)$, the action
of $\Ad_A | \tJ^\square$ is given by
$$
X \mapsto J_{t}(\lambda_k) \cdot X \cdot J_{s}(\lambda_\ell)^{-1},
$$
where $\lambda_k$ and $\lambda_\ell$ 
are respectively the row and the column eigenvalues of $\tJ$
and $J$ denotes Jordan blocks as defined by \eqref{e.block}.

\medskip

\begin{lemma}\label{l.min}
For each $\rm j$-rectangle $\tJ$, 
the only eigenvalue of $\Ad_A | \tJ^\square$ is the banner of the $\rm e$-rectangle that contains $\tJ$.
Moreover, the geometric multiplicity of the eigenvalue is the weight of the $\rm j$-rectangle.
\end{lemma}

\begin{proof}
The matrix of the the linear operator $\Ad_A | \tJ^\square$
can be described using the Kronecker product: see \cite[Lemma~4.3.1]{HJ}.
The Jordan form of this operator is then described by \cite[Theorem~4.3.17(a)]{HJ}.
The assertions of the lemma follow.
\end{proof}

Some immediate consequences are the following:
\begin{itemize}
\item The eigenvalues of $\Ad_A$ are the banners of $\rm e$-rectangles.
\item The geometric multiplicity of the eigenvalue $\beta$ for $\Ad_A$ is the total weight of $\rm e$-rectangles of banner~$\beta$.
\end{itemize}

\medskip

If $\tR$ is an equatorial $\rm j$-rectangle, $\rm e$-rectangle, or $\rm c$-rectangle
we will refer to 
the $d\times d$-matrix in $\tR^\square$ whose $\tR$-submatrix is the identity
as the \emph{identity on $\tR^\square$}.
The following observation will be useful:

\begin{lemma}\label{l.1x1}
If $\tJ$ is an equatorial $\rm j$-rectangle then the identity on $\tJ^\square$
is an eigenvector of the operator $\Ad_A | \tJ^\square$ corresponding to a Jordan block of size $1 \times 1$. 
\end{lemma}

\begin{proof}
Suppose $\tJ$ has size $t \times t$ and row (or column) eigenvalue $\lambda$.
Assume that the claim is false.
This means that there exists a matrix $X \in \Mat_{t \times t}(\C)$ such that $J_t(\lambda) X J_t(\lambda)^{-1} = X + \Id$, which is impossible because $X$ and $X + \Id$ have different spectra.
\end{proof}

\subsection{Rigidity estimates for $\rm j$-rectangles and $\rm e$-rectangles}

\begin{lemma}\label{l.rig district}
For any $\rm j$-rectangle $\tJ$, we have $\rig_+ (\Ad_A | \tJ^\square)  \le \pop \tJ$.
\end{lemma}

\begin{proof}
By \cref{l.min} (and \cref{p.acyc}), 
$\Ad_A|\tJ^\square$ has acyclicity $n = \pop \tJ$,
that is, there are matrices $X_1$, \dots, $X_n \in \tJ^\square$ 
such that 
$\sorb_{\Ad_A}(X_1, \dots, X_n)$ is the whole $\tJ^\square$
(and, in particular, is transitive in $\tJ^\square$).
So $\rig (\Ad_A|\tJ^\square) \le n$,
which proves the lemma for non-equatorial $\rm j$-rectangles.

If $\tJ$ is an equatorial $\rm j$-rectangle then, 
by \cref{l.1x1}, $\tJ^\square$ splits invariantly into two subspaces, one
of them spanned by the the identity matrix on $\tJ^\square$.
So we can choose the matrices $X_i$ above so that $X_1$ is the identity.
This shows that $\rig_+ (\Ad_A|\tJ^\square) \le n$.
\end{proof}

In all that follows, we adopt the convention $\max \emptyset = 0$.

\begin{lemma}\label{l.rig city} 
For any $\rm e$-rectangle $\tE$,
$$
\rig_+ (\Ad_A | \tE^\square) \le
\sum_{\text{\rm{$\ell$ latitude}}} 
\max_{\text{\rm{$\tJ \subset \tE$ is a $\rm j$-rectangle}} \atop \text{\rm{with latitude $\ell$}}}
\rig_+ (\Ad_A | \tJ^\square) \, .
$$
\end{lemma}

\begin{proof}
For each $\rm j$-rectangle $\tJ$ contained in $\tE$,
let $r(\tJ) = \rig_+ (\Ad_A|\tJ^\square)$.
Take matrices $X_{\tJ, 1}$, \dots, $X_{\tJ, r(\tJ)}$
such that 
$\Lambda_\tJ := \sorb_{\Ad_A} \big( X_{\tJ, 1}, \ldots, X_{\tJ, r(\tJ)} \big)$
is a transitive subspace of $\tJ^\square$,
and $X_{\tJ,1}$ is the identity matrix in $\tJ^\square$ if $\tJ$ is an equatorial ${\rm j}$-rectangle.
Define $X_{\tJ, i} = 0$ for $i>r(\tJ)$.		
For each latitude $\ell$, 
let $n_\ell$ be the maximum of $r(\tJ)$
over the $\rm j$-rectangles $\tJ$ of $\tE$ with latitude $\ell$,
and let
$$
Y_{\ell, i} = \sum_{\text{\rm{$\tJ \subset \tE$ is a $\rm j$-rectangle}} \atop \text{\rm{with latitude $\ell$}}} 
X_{\tJ, i} \, ,
\quad \text{for $1 \le i \le n_\ell$.}
$$
Notice that if $\tE$ is an equatorial $\rm e$-rectangle then $Y_{0,1}$ is 
the identity matrix in $\tE^\square$.
Consider the space 
$$
\Delta = \sorb_{\Ad_A} \big\{Y_{\ell,i} ; \; \ell \text{ is a latitude, } 1 \le i \le n_\ell \big\}.
$$

We claim that for every $\rm j$-rectangle $\tJ$ in $\tE$ and for every $M \in \Lambda_\tJ$,
we can find some $N \in \Delta$ with the following properties:
\begin{itemize} 
\item the submatrix $N_\tJ$ equals $M$;
\item for every $\rm j$-rectangle $\tJ'$ in $\tE$ that has a different latitude 
than $\tJ$, the submatrix $N_{\tJ'}$ vanishes.  
\end{itemize}
Indeed, if 
$M = \sum_{i=1}^{r(\tJ)} f_i(\Ad_A) X_{\tJ,i}$ for certain polynomials $f_i$,
we simply take 
$N = \sum_{i=1}^{r(\tJ)} f_i(\Ad_A) Y_{\ell,i}$, where $\ell$ is the latitude of $\tJ$.

In notation~\eqref{e.sudoku space}, 
the claim we have just proved means that 
$\Delta^{[\tJ]} \supset \Lambda_\tJ$.
So we can apply \cref{l.sudoku} and conclude that
$\Delta$ is a transitive subspace of $\tE^\square$.
Therefore $\rig_+ (\Ad_A | \tE^\square) \le \sum n_\ell$,
as we wanted to show.
\end{proof}

\begin{example}\label{ex.continued}
Using \cref{l.rig district,l.rig city},
we see that the $\rm e$-rectangle $\tE$ 
whose $\rm j$-rectangle weights are indicated in \cref{f.big matrix}
has $\rig_+ (\Ad_A | \tE^\square) \le 5$. 
\end{example}

In fact, we will not use \cref{l.rig district,l.rig city} directly, 
but only the following immediate consequence:
\begin{lemma}\label{l.crude rig city}
For every $\rm e$-rectangle $\tE$ we have $\rig_+ (\Ad_A | \tE^\square) \le \pop \tE$.
The inequality is strict if $\tE$ has more than one row of $\rm j$-rectangles and more that one column of $\rm j$-rectangles.
\end{lemma}


\subsection{Comparison of weights}

If $\tR$ is a $\rm j$-rectangle, $\rm e$-rectangle or $\rm c$-rectangle,
we define its \emph{row projection} $\prow(\tR)$
as the unique equatorial $\rm j$-rectangle, $\rm e$-rectangle or $\rm c$-rectangle (respectively)
that is in the same row as $\tR$.
Analogously, we define the \emph{column projection} $\pcol(\tR)$.

\begin{lemma}\label{l.city proj}
For any $\rm e$-rectangle $\tE$, we have
$$
\pop \tE \le \frac{\pop \prow(\tE) + \pop \pcol(\tE)}{2} \, .
$$
Moreover, if equality holds then 
the number of rows of $\rm j$-rectangles for $\tE$ equals the number of columns of $\rm j$-rectangles.
\end{lemma}

This is a clear consequence of the abstract lemma below,
taking $x_\alpha$, $\alpha\in F_0$ (resp.\ $\alpha\in F_1$) 
as the sequence of heights (resp.\ widths) of $\rm j$-rectangles in $\tE$,
counting repetitions.

\begin{lemma}\label{l.combin}
Let $F$ be a nonempty finite set,
and let $x_\alpha$ be positive numbers indexed by $\alpha \in F$.
Take any partition $F = F_0 \sqcup F_1$, where $\sqcup$ stands for disjoint union.
For $\epsilon$, $\delta \in \{0,1\}$, let
$$
\Sigma_{\epsilon\delta} = \sum_{(\alpha,\beta) \in F_\epsilon \times F_\delta} \min(x_\alpha, x_\beta) \, .
$$
Then
$$
\Sigma_{01} = \Sigma_{10} \le \frac{\Sigma_{00} + \Sigma_{11}}{2} \, .
$$
Moreover, equality implies that $F_0$ and $F_1$ have the same cardinality.
\end{lemma}

\begin{proof}
We will in fact prove the stronger fact:
\begin{equation}\label{e.induction}
\Sigma_{00} - 2\Sigma_{01} + \Sigma_{11} \ge \big(|F_0| -|F_1|\big)^2 \min_{\alpha \in F} x_\alpha \, ,
\end{equation}
where $|\mathord{\cdot}|$ denotes set cardinality.
The proof is by induction on $|F|$.
It clearly holds for $|F|=1$.
Fix some $n$ and assume that \eqref{e.induction} always holds when $|F| = n$.
Take a set $F$ with $|F|=n+1$, and take positive numbers $x_\alpha$, $\alpha \in F$.
We can assume that $F=\{1, \ldots, n+1\}$ and that $x_1 \ge \cdots \ge x_{n+1}$.
Take any partition $F = F_0 \sqcup F_1$.
Without loss of generality, assume that $n+1 \in F_0$.
Apply the induction hypothesis to $F'=\{1, \ldots, n\}$, obtaining
$$
\Sigma'_{00} - 2\Sigma'_{01} + \Sigma'_{11}  \ge  \big(|F_0| - 1 - |F_1|)^2 x_n.
$$
We have 
$$
\Sigma_{00} = \Sigma_{00}' + \big( 2|F_0| -1 \big) x_{n+1} \, ,
\quad 
\Sigma_{01} = \Sigma'_{01} + |F_1| x_{n+1} \, ,
\quad \text{and} \quad
\Sigma_{11} = \Sigma'_{11} \, ,
$$
so \eqref{e.induction} follows.
\end{proof}

If $\tR$ is a $\rm c$-rectangle or the entire square $[1,d]^2$, let $\pop_1 \tR$ denote 
the sum of the weights of the $\rm e$-rectangles in $\tR$ with banner $1$.

Let us give the following useful consequence of \cref{l.city proj}:

\begin{lemma}\label{l.acyc and pop1}
$\acyc \Ad_A = \pop_1 [1,d]^2$.
\end{lemma}

\begin{proof}
By \cref{p.acyc}, $\acyc \Ad_A$ 
is the maximum of the geometric multiplicities of the eigenvalues of $\Ad_A$.
Those eigenvalues are the banners $\beta$, and the geometric multiplicity of each $\beta$
is the total weight with banner $\beta$.
Thus, to prove the lemma we have to show that banner $1$ has biggest total weight.

Let $\beta$ be a banner.
Then, using \cref{l.city proj},
$$
\sum_{\text{$\tE$ is an $\rm e$-rectangle}\atop\text{with banner $\beta$}} \pop \tE 
\le \frac12 \sum_{\text{$\tE$ is an $\rm e$-rectangle}\atop\text{with banner $\beta$}} \pop \prow(\tE)  
+   \frac12 \sum_{\text{$\tE$ is an $\rm e$-rectangle}\atop\text{with banner $\beta$}} \pop \pcol(\tE) \, .
$$
Since no two $\rm e$-rectangles in the same row (resp.\ column) can have the same banner,
the restriction of $\prow$ (resp.\ $\pcol$) to the set of $\rm e$-rectangles with banner $\beta$ is 
a one-to-one map. 
This allows us to conclude.
\end{proof}

\begin{rem}\label{r.jordan type and acyc}
The \emph{Jordan type} of a matrix $A \in \Mat_{d \times d}(\C)$ consists 
on the following data: 
\begin{enumerate}
\item The number of different eigenvalues.
\item For each eigenvalue, the number of Jordan blocks and their sizes.
\end{enumerate}
It follows from \cref{l.acyc and pop1} that these data is sufficient to
determine $\acyc \Ad_A$.
\end{rem}

\subsection{Rigidity estimate for $\rm c$-rectangles}

\begin{lemma}\label{l.rig island}
For any $\rm c$-rectangle $\tC$,	
$$
\rig_+ (\Ad_A | \tC^\square) \le \frac{\pop_1 \prow(\tC) + \pop_1 \pcol(\tC)}{2} \, .
$$
\end{lemma}

In order to prove this lemma, it is convenient to consider separately
the cases of non-equatorial and equatorial $\rm c$-rectangles.

\begin{proof}[Proof of \cref{l.rig island} when $\tC$ is non-equatorial]
For each banner $\beta$ in $\tC$, let $n_\beta$ (resp.\ $s_\beta$) be the maximum of 
$\rig_+(\Ad_A | \tE^\square)$ over nonnegative (resp.\ negative) argument $\rm e$-rectangles $\tE$ in $\tC$ with banner $\beta$.
For each $\rm e$-rectangle $\tE$ with banner $\beta$, 
choose matrices $X_{\tE,1}$, \dots, $X_{\tE, n_\beta + s_\beta} \in \tE^\square$ such that:
\begin{itemize}
\item $\Lambda_\tE := \sorb_{\Ad_A} ( X_{\tE,1}, \dots, X_{\tE, m})$ is a transitive subspace of $\tE^\square$;
\item if $\tE$ has negative argument then $X_{1} = X_{2} = \cdots = X_{n_\beta} = 0$;
\item if $\tE$ has nonnegative argument then $X_{n_\beta+1} = \cdots = X_{n_\beta+s_\beta} = 0$.
\end{itemize}
Also, let $X_{\tE, j} = 0$ for $j>n_\beta+s_\beta$.

Next, define
\begin{equation}\label{e.Y_j}
Y_{\beta, j} = 
\sum_{\text{$\tE$ is an $\rm e$-rectangle}\atop\text{of $\tC$ with banner $\beta$}} X_{\tE, j} 
\end{equation}
and
\begin{equation}\label{e.Z_j}
Z_j = \sum_{\text{$\beta$ banner on $\tC$}} Y_{\beta, j} 
\end{equation}
Consider the space 
$$
\Delta = \sorb_{\Ad_A} (Z_1, \dots, Z_m), \quad 
\text{where} \quad
m = \max_{\text{$\beta$ banner on $\tC$}} (n_\beta + s_\beta)
$$
It follows from \cref{l.sum} that 
$$
\Delta = \sorb_{\Ad_A} \big\{Y_{\beta,j} ; \; \beta \text{ is a banner, } 1 \le j \le n_\beta + s_\beta \big\}.
$$

Recall notation~\eqref{e.sudoku space}.
We claim that 
\begin{equation}\label{e.inclusion}
\Lambda_\tE \subset \Delta^{[\tE]}.
\end{equation}
Indeed, given $M \in \Lambda_{\tE}$, write 
$M = \sum_j f_j(\Ad_A) X_{\tE, j}$,
where the $f_j$'s are polynomials and $f_j \equiv 0$ whenever $X_{\tE,j}=0$.
Consider $N = \sum_j f_j(\Ad_A) Y_{\beta, j}$, where $\beta$ is the banner of $\tE$.
Then it follows from \cref{l.geo} (part \ref{i.geo2}) that 
$N \in \Delta^{[\tE]}$. 
This shows \eqref{e.inclusion}.
So, by \cref{l.sudoku}, $\Delta$ is a transitive subspace of $\tC^\square$, showing that
$\rig_+ (\Ad_A | \tC^\square) \le m$.

To complete the proof of the lemma in the non-equatorial case, we show that
\begin{equation}\label{e.nonequat ineq}
m \le \frac{\pop_1 \prow(\tC) + \pop_1 \pcol(\tC)}{2} \, .
\end{equation}

Let $\beta$ be the banner for which $n_\beta + s_\beta$ attains the maximum $m$.
If $n_\beta>0$, let $\tE_+$ be a nonnegative argument $\rm e$-rectangle in $\tC$ with banner $\beta$
and $\rig_+(\Ad_A | \tE_+^\square) = n_\beta$. 
If $s_\beta>0$, let $\tE_-$ be a negative argument $\rm e$-rectangle in $\tC$ with banner $\beta$
and $\rig_+(\Ad_A | \tE_-^\square) = s_\beta$. 
Assume for the moment that both $\rm e$-rectangles exist.
Let $\tE_1$, $\tE_2$, $\tE_3$, $\tE_4$ be projected equatorial $\rm e$-rectangles as in \cref{f.projections1}.

\begin{figure}[htb] 
\begin{minipage}[t]{0.4\textwidth}
\centering
\setlength{\unitlength}{.25cm}
\begin{picture}(19,19)  
\thicklines  
\put(0,0){\framebox(19,19)} 
\put(.5,.5){$[1,d]^2$}

\put(1,11){\framebox(7,7)} 
\put(1.5,11.5){$\tC_1$}

\put(9,1){\framebox(9,9)} 
\put(9.5,1.5){$\tC_2$}

\put(9,11){\framebox(9,7)} 
\put(17,17){$\tC$}

\thinlines
\put(2.5,14){\framebox(2.5,2.5){$\tE_1$}} 
\put(5.5,11.5){\framebox(2,2){$\tE_2$}}
\put(10,5){\framebox(4,4){$\tE_3$}}
\put(15,2){\framebox(2,2){$\tE_4$}}
\put(10,11.5){\framebox(4,2){$\tE_-$}} 
\put(15,14){\framebox(2,2.5){ $\tE_+$}}

\put(5,16.5){\dashbox{.2}(10,0)}
\put(5,14){\dashbox{.2}(10,0)}
\put(7.5,13.5){\dashbox{.2}(2.5,0)}
\put(7.5,11.5){\dashbox{.2}(2.5,0)}

\put(10,9){\dashbox{.2}(0,2.5)}
\put(14,9){\dashbox{.2}(0,2.5)}
\put(15,4){\dashbox{.2}(0,10)}
\put(17,4){\dashbox{.2}(0,10)}
\end{picture}
\caption{The case of non-equatorial $\rm c$-rectangles: $\tE_1 = \prow(\tE_+)$, $\tE_2 = \prow(\tE_-)$, $\tE_3 = \pcol(\tE_-)$, $\tE_4 = \pcol(\tE_+)$.} 
\label{f.projections1}
\end{minipage}
\qquad 
\qquad
\begin{minipage}[t]{0.4\textwidth}
\centering
\setlength{\unitlength}{.25cm}
\begin{picture}(19,19)  
\thicklines  
\put(0,0){\framebox(19,19)} 
\put(.5,.5){$\tC$}

\thinlines
\put(2.5,14){\framebox(2.5,2.5){$\tE_1$}} 
\put(5.5,11.5){\framebox(2,2){$\tE_2$}}
\put(10,5){\framebox(4,4){$\tE_3$}}
\put(15,2){\framebox(2,2){$\tE_4$}}
\put(10,11.5){\framebox(4,2){$\tE_+$}} 
\put(2.5,2){\framebox(2.5,2){$\tE_-$}}

\put(5,2){\dashbox{.2}(10,0)}
\put(5,4){\dashbox{.2}(10,0)}
\put(7.5,13.5){\dashbox{.2}(2.5,0)}
\put(7.5,11.5){\dashbox{.2}(2.5,0)}

\put(10,9){\dashbox{.2}(0,2.5)}
\put(14,9){\dashbox{.2}(0,2.5)}
\put(2.5,4){\dashbox{.2}(0,10)}
\put(5,4){\dashbox{.2}(0,10)}
\end{picture}
\caption{The case of equatorial non-exceptional $\rm c$-rectangles: $\tE_1 = \pcol(\tE_-)$, $\tE_2 = \prow(\tE_+)$, $\tE_3 = \pcol(\tE_+)$, $\tE_4 = \prow(\tE_-)$. It is possible that $\tE_1 = \tE_2$ or $\tE_3 = \tE_4$.} 
\label{f.projections2}
\end{minipage}
\end{figure}
Then
\begin{multline*}
m = \rig_+ (\Ad_A | \tE_+) + \rig_+ (\Ad_A | \tE_-) 
\stackrel{\text{(i)}}{\le} \pop \tE_+ + \pop \tE_-  \\
\stackrel{\text{(ii)}}{\le} \tfrac{1}{2}\big( \pop \tE_1 + \dots + \pop \tE_4 \big)  
\le \tfrac{1}{2}\big(\pop_1 \tC_1 + \pop_1 \tC_2 \big), 
\end{multline*}
where (i) and (ii) follow respectively from \cref{l.crude rig city,l.city proj}.
This proves \eqref{e.nonequat ineq} in this case.
If there is no nonnegative argument $\rm e$-rectangle or no negative argument $\rm e$-rectangle within $\tC$ with banner~$1$
then the proof of \eqref{e.nonequat ineq} is easier.

So the lemma is proved for non-equatorial $\tC$.
\end{proof}

We now consider equatorial $\rm c$-rectangles.
There is a special kind of $\rm c$-rectangle 
for which the proof of the rigidity estimate has to follow a different strategy.
A $\rm c$-rectangle is called \emph{exceptional} if
it has only the banners $1$ and $-1$ (so it is equatorial and has $4$ $\rm e$-rectangles),
each $\rm e$-rectangle has a single $\rm j$-rectangle,
and all $\rm j$-rectangles have the same weight.

\begin{proof}[Proof of \cref{l.rig island} when $\tC$ is equatorial non-exceptional]
As in the previous case,  
let $n_\beta$ (resp.\ $s_\beta$) be the maximum of 
$\rig_+(\Ad_A|\tE^\square)$ over the nonnegative (resp.\ negative)  argument $\rm e$-rectangles $\tE$ in $\tC$ with banner $\beta$.

We claim that 
\begin{equation}\label{e.equat ineq}
n_\beta + s_\beta < \pop_1 \tC \quad \text{for all banners $\beta \neq 1$ in $\tC$.}
\end{equation}
Let us postpone the proof of this inequality and see how to conclude.

Let $M=\pop_1 \tC$.
In view of \cref{l.crude rig city} and relation~\eqref{e.equat ineq},
for each $\rm c$-rectangle $\tE$ we can take matrices $X_{\tE, 1}$, \dots, $X_{\tE, M} \in \tE^\square$
such that:
\begin{itemize}
\item $\Lambda_\tE := \sorb_{\Ad_A} (X_{\tE, 1}, \dots, X_{\tE, M})$ is a transitive subspace of  $\tE^\square$;
\item $X_{\tE, M} = 0$ if $\tE$ is non-equatorial;
\item $X_{\tE, M}$ is the identity in $\tE^\square$ if $\tE$ is equatorial.
\end{itemize}
Then define matrices $Z_j$ as before: by \eqref{e.Y_j} and \eqref{e.Z_j}.
Here we have that $Z_M$ is the identity matrix in $\tC^\square$.
As before, $\sorb_{\Ad_A}(Z_1, \dots, Z_M)$ is a transitive subspace of $\tC^\square$.
Hence $\rig_+(\Ad_A|\tC^\square) \le M = \pop_1 \tC$, as desired.

\medskip

Now let us prove \eqref{e.equat ineq}.
Consider a banner $\beta \neq 1$ in $\tC$.
Let $\tE_+$ (resp.\ $\tE_-$) be a nonnegative (resp.\ negative) argument $\rm e$-rectangle within $\tC$ with banner $\beta$ 
and of maximal weight;
assume for the moment that both $\rm e$-rectangles exist.
Let $\tE_1$, $\tE_2$, $\tE_3$, $\tE_4$ be projected equatorial $\rm e$-rectangles as in \cref{f.projections2}.
Then
\begin{align}
n_\beta + s_\beta  
&= \rig_+(\Ad_A | \tE_+) + \rig_+(\Ad_A | \tE_-)                \nonumber \\
&\le \pop \tE_+ + \pop \tE_-                                   \label{e.ineq1}\\
&\le  \tfrac{1}{2}\big( \pop \tE_1 + \dots + \pop \tE_4 \big)   \label{e.ineq2}\\
&\le \pop_1 \tC.                                                \label{e.ineq3}
\end{align}
Inequality~\eqref{e.ineq1} follows from \cref{l.crude rig city}, 
inequality~\eqref{e.ineq2} follows from \cref{l.city proj}, and
inequality~\eqref{e.ineq3} holds because the $\rm e$-rectangles $\tE_1$, \dots, $\tE_4$ are equatorial,
and any $\rm e$-rectangle can appear at most twice in this list.
So 
\begin{equation}\label{e.weaker}
n_\beta + s_\beta \le \pop_1 \tC.
\end{equation}
In the case that there is no nonnegative argument $\rm e$-rectangle or no negative argument $\rm e$-rectangle with banner $\beta$ 
(i.e., $n_\beta$ or $s_\beta$ vanishes),
a simpler argument shows that strict inequality holds in \eqref{e.weaker}.

Now assume by contradiction that \eqref{e.equat ineq} does not hold.
Then we must have equality in \eqref{e.weaker}.
By what we have just seen, both $\rm e$-rectangles $\tE_+$ and $\tE_-$ above exist.
Then the inequalities in \eqref{e.ineq1}--\eqref{e.ineq3} become equalities.
Since~\eqref{e.ineq3} is an equality, there must be exactly two equatorial $\rm e$-rectangles in $\tC$.
So the non-equatorial banner $\beta$ satisfies $\beta^{-1} = \beta$, that is, $\beta=-1$.
Since~\eqref{e.ineq2} is an equality, it follows from \cref{l.city proj} that both non-equatorial $\rm e$-rectangles 
have the same number of $\rm j$-rectangles in each column and each row.
So there is some $\ell$ such that all four $\rm e$-rectangles in $\tC$ have $\ell$ rows of $\rm j$-rectangles 
and $\ell$ columns of $\rm j$-rectangles.
Since~\eqref{e.ineq1} is an equality, \cref{l.crude rig city} implies that $\ell=1$.
That is, $\tC$ is a exceptional $\rm c$-rectangle, a situation which we excluded a priori.
This contradiction proves \eqref{e.equat ineq} and \cref{l.rig island}
in the present case.
\end{proof}

Let us now deal with exceptional $\rm c$-rectangles.
In all the previous cases, the transitive subspace we found had some vaguely Toeplitz form.
For exceptional $\rm c$-rectangles, however, this strategy is not efficient. 
What we are going to do is to find a transitive space of vaguely Hankel form, namely the following:
\begin{equation}\label{e.hankel like}
\Lambda_k = \left\{ 
\begin{pmatrix}
P & M \\
M & N
\end{pmatrix};
\; \text{$M$, $N$, $P$ are $k\times k$ matrices}
\right\}.
\end{equation}
Notice that $\Lambda_k = S_k \cdot \Gamma_k$, where
$$
S_k = \begin{pmatrix} 0 & \Id \\ \Id & 0 \end{pmatrix}
\quad \text{and} \quad
\Gamma_k = \left\{ 
\begin{pmatrix}
M & N \\
P & M
\end{pmatrix};
\; \text{$M$, $N$, $P$ are $k\times k$ matrices}
\right\}.
$$
Since $\Gamma_k$ is 
a generalized Toeplitz space, 
it follows from \cref{r.transitivity trick} that $\Lambda_k$ is transitive.

\begin{proof}[Proof of \cref{l.rig island} when $\tC$ is exceptional]
If $\tC$ is exceptional then it has size $2k \times 2k$ for some $k$,
and the operator $\Ad_A | \tC^\square$ is given by $X \mapsto \Ad_L(X)$, where
$$
L = \begin{pmatrix}  J & 0 \\ 0 & -J \end{pmatrix},
\quad \text{and $J = J_k(1)$ is the Jordan block~\eqref{e.block}.}
$$
Let $V$ be unique $\Ad_J$-invariant subspace of $\Mat_{k \times k}(\C)$ that 
has codimension $1$ and does not contain the identity matrix (which exists by \cref{l.1x1}).
Take matrices $X_1$, \dots, $X_k \in \Mat_{k \times k}(\C)$ such that
$X_1 = \Id$
and 
$V = \sorb_{\Ad_J} (X_2, \dots, X_k)$.
Define $Y_1$, \dots, $Y_k \in \Mat_{2k \times 2k} (\C)$ by 
$$
Y_1 = \begin{pmatrix} \Id & 0 \\ 0 & \Id \end{pmatrix},
\quad 
Y_j = \begin{pmatrix} X_j & 0       \\ 0 & 0   \end{pmatrix} \text{ for $2 \le j \le k$,}
$$
Then
$$
\sorb_{\Ad_L}(Y_1,\dots,Y_k) = 
\left\{ 
\begin{pmatrix} x\Id + Z & 0 \\ 0 & x \Id \end{pmatrix} ; \;
x \in \C, \ Z \in V
\right\}.
$$
For $j=k+1$, \dots, $2k$, define
$$
Y_j = \begin{pmatrix} 0   & X_{j-k} \\ X_{j-k} & X_{j-k} \end{pmatrix}. 
$$
Then, by \cref{l.sum}, 
$$
\sorb_{\Ad_L}(Y_{k+1},\dots,Y_{2k}) = 
\left\{ 
\begin{pmatrix} 0 & M \\ M & N \end{pmatrix} ; \;
M, \ N \in \Mat_{k \times k}(\C)
\right\}.
$$
Therefore $\sorb_{\Ad_L}(Y_1,\dots,Y_{2k})$ is the transitive space given by \eqref{e.hankel like}.
Since $Y_1$ is the identity on $\tC$, this shows that 
$\rig_+(\Ad_A|\tC^\square) \le 2k = \pop_1 \tC$,
concluding the proof of \cref{l.rig island}.
\end{proof}

\subsection{The final rigidity estimate}

Let $c = c(A)$ be the number of equivalence classes mod~$T$ of eigenvalues of $A$.

\begin{lemma}\label{l.rig world} 
If $c<d$ then
$$
\rig_+ \Ad_A \le  \pop_1 [1,d]^2 - c + 1 \, . 
$$
\end{lemma}

\begin{proof} 
Let $m = \pop_1 [1,d]^2 - c + 1$.
For each $\rm c$-rectangle $\tC$, let 
$$
r(\tC) =  \big\lfloor \tfrac{1}{2} (\pop_1 \prow(\tC) + \pop_1 \pcol(\tC)) \big\rfloor.
$$
We claim that 
\begin{equation}\label{e.crude}
r(\tC) \le 
\begin{cases}
m   &\text{if $\tC$ is an equatorial $\rm c$-rectangle,} \\
m-1 &\text{if $\tC$ is a non-equatorial $\rm c$-rectangle.}
\end{cases}
\end{equation}
Let us postpone the proof of this inequality and see how it implies the lemma.

In view of \cref{l.rig island} and relation~\eqref{e.crude},
for each $\rm c$-rectangle $\tC$ we can take matrices $X_{\tC, 1}$, \dots, $X_{\tC, m} \in \tC^\square$
such that:
\begin{itemize}
\item $\Lambda_\tC := \sorb_{\Ad_A} (X_{\tC, 1}, \dots, X_{\tC, m})$ is a transitive subspace of  $\tC^\square$;
\item $X_{\tC, m} = 0$ if $\tC$ is non-equatorial;
\item $X_{\tC, m}$ is the identity in $\tC^\square$ if $\tC$ is equatorial.
\end{itemize}

Define matrices:
\begin{alignat*}{2}
Y_{\alpha, j} &= 
\sum_{\text{$\tC$ is a $\rm c$-rectangle}\atop\text{with banner class $\alpha$}} X_{\tC, j}
&\qquad &(\alpha \text{ is a banner class}, \ 1 \le j \le m), \\
Z_j &= \sum_{\alpha \text{ is a banner class}} Y_{\alpha, j} 
&\qquad &(1 \le j \le m).
\end{alignat*}
So $Z_m$ is the $d \times d$ identity matrix.
Consider the space 
$$
\Delta = \sorb_{\Ad_A} (Z_1, \dots, Z_m).
$$
It follows from \cref{l.sum} that 
$$
\Delta = \sorb_{\Ad_A} \big\{ Y_{\alpha,j} ; \; \alpha \text{ is a banner class, } 1 \le j \le m \big\}.
$$
We claim that every $\rm c$-rectangle $\tC$, 
\begin{equation}\label{e.inclusion2}
\Lambda_\tC \subset \Delta^{[\tC]}.
\end{equation}
Indeed, if $M \in \tC$ then we can write $M = \sum_j f_j(\Ad_A) X_{\tC,j}$,
where the $f_j$'s are polynomials.
Consider $N = \sum_j f_j(\Ad_A) Y_{\alpha,j}$, where $\alpha$ is the banner class of $\tC$.
It follows \cref{l.geo} (part \ref{i.geo0}) that $N \in \Delta^{[\tC]}$.
This proves \eqref{e.inclusion2}.
So, by \cref{l.sudoku}, $\Delta$ is a transitive subspace of $\Mat_{d \times d}(\C)$, 
showing that $\rig_+ \Ad_A \le m$.

\medskip

To conclude the proof we have to show estimate~\eqref{e.crude}.
First consider a equatorial $\rm c$-rectangle~$\tC$.
Since there are $c$
equatorial $\rm c$-rectangles, and each of them has a nonzero $\pop_1$ value,
we conclude that $r(\tC) \le m$, as claimed.

Now take a non-equatorial $\tC$.
Applying what we have just proved for the equatorial $\rm c$-rectangles $\prow(\tC)$ and $\pcol(\tC)$,
we conclude that $r(\tC) \le m$.
Now assume that \eqref{e.crude} does not hold for $\tC$, that is, $r(\tC) = m$.
Then 
$$
\pop_1 \prow(\tC) = \pop_1 \pcol (\tC) = m = \pop_1 [1,d]^2 - c + 1.
$$
Since $\pop_1 [1,d]^2 \ge \pop_1 \prow(\tC) + \pop_1 \pcol (\tC) + c - 2$,
we have $m=1$ and $\pop_1 [1,d]^2 = c$.
This means that $\pop_1 \tilde \tC = 1$ for all equatorial $\rm c$-rectangles $\tilde \tC$,
which is only possible if $c=d$.
However, this case was excluded by hypothesis.

This proves \eqref{e.crude} and hence \cref{l.rig world}.
\end{proof}

\begin{example}
If $A$ is the matrix of \cref{ex.big matrix} then \cref{l.rig world}
gives the estimate $\rig_+ \Ad_A \le 28$.
A more careful analysis (going through the proofs of the lemmas) would give $\rig_+ \Ad_A \le 7$
(see \cref{ex.continued}).
\end{example}

\begin{proof}[Proof of part~\ref{i.rig hard} of \cref{t.rig}]
Apply \cref{l.rig world,l.acyc and pop1}.
\end{proof}

\section{Proof of the hard part of the codimension $m$ theorem}\label{s.cod proof}

We showed in \cref{p.cod_data_easy_half} that $\codim \cP_m^{(\K)} \le m$.
In this section, we will prove the reverse inequalities.
More precisely, we will first prove \cref{t.cod_data_C}
and then deduce \cref{t.cod_data_R} from it.

\subsection{Preliminaries on elementary algebraic geometry}

\subsubsection{Quasiprojective varieties} 

An algebraic subset of $\C^n$ is also called an \emph{affine variety}.
A \emph{projective variety} is a subset of $\CP^n$ that can be expressed 
as the zero set of a family of homogeneous polynomials in $n+1$ variables.
The \emph{Zariski topology} on an (affine or projective) variety $X$ is 
the topology whose closed sets are the (affine or projective) subvarieties of $X$.

An open subset $U$ of a projective variety $X$ is called a \emph{quasiprojective variety}.
We consider in $U$ the induced Zariski topology.
The affine space $\C^n$ can be identified with a quasiprojective variety,
namely its image under the embedding $(z_1, \dots, z_n) \mapsto (1: z_1 : \dots : z_n)$.

If $X$ and $Y$ are quasi-projective varieties
then the product $X \times Y$ can be identified with a quasiprojective variety, 
namely its image under the Segre embedding; see \cite[\S~5.1]{Shafa}.

Recall the following property from \cite[p.~58]{Shafa}:
 
\begin{prop}\label{p.projection}
If $X$ is a projective variety and $Y$ is a quasiprojective variety 
then the projection $p \colon X \times Y \to Y$ takes Zariski closed sets to Zariski closed sets.
\end{prop}

A quasiprojective variety is called \emph{irreducible} if it cannot be written as a nontrivial union of two quasiprojective varieties (that is, none contains the other).

\subsubsection{Dimension}
The dimension $\dim X$ of an irreducible quasiprojective variety $X$
may be defined in various equivalent ways (see for instance~\cite[p.~133ff]{Harris}). 
It will be sufficient for us to know that there exists an (intrinsically defined) 
subvariety $Y$ of the \emph{singular points of $X$}
such that in a neighborhood of each point of $X \setminus Y$, the set 
$X$ is a complex submanifold of dimension (in the classical sense of differential geometry) $\dim X$;
moreover, each irreducible component of $Y$ has dimension strictly less than $\dim X$.

The dimension of a general quasiprojective variety is by definition the maximum of the dimensions of the irreducible components.

The following lemma is useful to estimate the codimension of an algebraic set $X$
from information about the fibers of a certain projection $\pi \colon X \to Y$.

\begin{lemma}\label{l.pret_a_porter}
Let $Y$ be a quasiprojective variety.  
Let $X \subset Y \times \CP^n$ be a nonempty algebraically closed set.
Let $\pi \colon X \to Y$ be the projection along $\CP^n$.
Then:
\begin{enumerate}
\item
For each $j \ge 0$, the set
$$
C_j = \{ y \in \pi(X) ; \; \codim \pi^{-1}(y) \le j \}
$$
is algebraically closed in $Y$.
\item
The dimension of $X$ is given in terms of the dimensions of the $C_j$'s by:
\begin{equation}\label{e.formula_pret_a_porter}
\codim X = \min_{j ; \; C_j \neq \emptyset} \big( j + \codim C_j \big) \, .
\end{equation}
\end{enumerate}
\end{lemma}

In the above, the codimensions of $\pi^{-1}(Y)$, $X$ and $C_j$
are taken with respect to $\CP^n$, $Y \times \CP^n$ and $Y$, respectively.
The proof of the lemma is given in \cref{a.algebraic}.

\begin{rem}\label{r.homogeneous}
\Cref{l.pret_a_porter} works with the same statement if $\CP^{n}$ is replaced by $\C^{n+1}$,
provided one assumes that $X \subset Y \times \C^{n+1}$ is homogeneous in the second factor
(i.e., $(y,z) \in X$ implies $(y,tz)\in X$ for every $t\in \C$).
Indeed, this follows from the fact that the projection $\C^{n+1}\setminus\{0\} \to \CP^{n}$
preserves codimension of homogeneous sets.
\end{rem}

\subsubsection{Dimension estimates for sets of vector subspaces}\label{sss.sets_of_spaces}

%

If $M \in \Mat_{n \times m}(\K)$, let $\col M \subset \K^n$ denote the column space of $M$.
A set $X \subset \Mat_{n \times m}(\K)$ is called \emph{column-invariant}
if 
$$
\left.
\begin{array}{c}
M \in X \\ 
N \in \Mat_{n \times m}(\K)\\
\col M = \col N
\end{array}
\right\} \ \Rightarrow \ 
N \in X.
$$
So a column-invariant set $X$ is characterized by its set of column spaces.
We enlarge the latter set by including also subspaces, thus defining:
\begin{equation}\label{e.bracket_notation}
\ldbrack X \rdbrack := \big\{ E \text{ subspace of } \K^n ; \; E \subset \col M \text{ for some } M \in X \big\}.
\end{equation}
In \cref{a.algebraic} we prove:

\begin{thm}\label{t.schubert}
Let $X \subset \Mat_{n \times m}(\C)$ be an algebraically closed, column-invariant set.
Suppose $E$ is a vector subspace of $\C^n$
that does not belong to $\ldbrack X \rdbrack$.
Then
$$
\codim X \ge m + 1 - \dim E \, .
$$
\end{thm}



\subsubsection{The real part of an algebraic set}

Let $X$ be an algebraically closed subset of $\C^n$.
The \emph{real part} of $X$ is defined as $X \cap \R^n$.
This is an algebraically closed subset of $\R^n$. 
Indeed, generators of the corresponding ideal $f_1,\ldots, f_k$ in $\C[T_1,\ldots, T_n]$ can be replaced by the corresponding real and imaginary parts polynomials.

As in the complex case, there are many equivalent algebraic-geometric definitions of dimensions of real algebraic or semialgebraic sets. We just point out that a real algebraic or semialgebraic set admits a stratification into real manifolds such that the maximal differential-geometric dimension of the strata coincides with the algebraic-geometric dimension (see~\cite[\S~3.4]{BR} or \cite[p.~50]{BCR}).

The following is an immediate consequence of \cite[Prop.~3.3.2]{BR}:

\begin{prop}\label{p.real complex dimension}
If $X$ is an algebraically closed subset of $\C^n$ then
$\dim_\R (X \cap \R^n) \le \dim_\C X$.
\end{prop}

\subsection{Rigidity and the dimension of the poor fibers}

For simplicity of notation, let us write $\cP_m = \cP_m^{(\C)}$.
Also, for $A \in \GL(d,\C)$, write:
$$
r(A) := \rig_+ \Ad_A  - 1  \, .
$$

We decompose the set $\cP_m$ of poor data in fibers:
\begin{equation}\label{e.fiber decomp}
\cP_m = \bigcup_{A \in \GL(d,\C)} \{A\} \times \cP_m(A),
\quad \text{where} \quad
\cP_m(A) \subset \gl(d,\C)^m \, .
\end{equation}

\begin{lemma}\label{l.cod fiber}
For any $A \in \GL(d,\C)$, 
the codimension of $\cP_m(A)$ in $\gl(d,\C)^m$
is at least $m + 1 - r(A)$.  
\end{lemma}

The lemma follows easily from \cref{t.schubert} above:

\begin{proof} 
Fix $A \in \GL(d,\C)$, and write $r=r(A)$.
We can assume that $r \le m$, otherwise there is nothing to prove.
By definition, there exists a $r$-dimensional subspace $E \subset \gl(d,\C)^m$
such that $\sorb_{\Ad_A}(\Id \vee E)$ is transitive.
Identify $\gl(d,\C)$ with $\C^{d^2}$ and thus
regard $\cP_m(A)$ as a subset of $\Mat_{d^2 \times m} (\C)$.
Since the set $\cP_m$ is algebraically closed and saturated (recall \cref{ss.poor_set}),
the fiber $\cP_m(A)$ is algebraically closed and column-invariant,
as required by \cref{t.schubert}.
In the notation \eqref{e.bracket_notation}, 
we have $E \not\in \ldbrack \cP_m(A) \rdbrack$.
So \cref{t.schubert} gives the desired codimension estimate.
\end{proof}

\subsection{How rare is high rigidity?}

For simplicity of notation, let us write:
$$
a(A) := \acyc \Ad_A  \quad \text{for $A \in \GL(d,\C)$.}
$$
So \cref{t.rig} says that $r(A) \le a(A)-c(A)$ provided $c(A) < d$.

\begin{lemma}\label{l.cod rig}
For any integer $k \ge 1$, the set
$$
M_k = \big\{ A \in \GL(d,\C)  ; \; r(A) \ge k \big\};
$$
is algebraically closed in $\GL(d,\C)$;
moreover if $M_k \neq \emptyset$ then
$$
\codim M_k
\begin{cases}
= 0   &\text{if $k=1$,} \\
\ge k &\text{if $k \ge 2$.}
\end{cases}
$$
\end{lemma}


\Cref{l.cod rig} is basically a consequence of \cref{t.rig},
using the following construction:

\begin{lemma}\label{l.countable cover}
There is a family $\cG(A)$ of subsets of $\GL(d,\C)$, indexed by $A\in \GL(d,\C)$,
such that the following properties hold:
\begin{enumerate}
\item \label{i.contains}
Each $\cG(A)$ contains $A$.
\item \label{i.manifold}
Each $\cG(A)$ is an immersed manifold of codimension $a(A)-c(A)$.
\item \label{i.countable}
There are only countably many different sets $\cG(A)$.
\end{enumerate}
\end{lemma}
 

\begin{proof}
Fix any $A \in \GL(d,\C)$.
Then $A$ is conjugate to a matrix in Jordan form:
$$
\tilde{A} =
\left(
\begin{array}{ccc}
J_{t_1}(\lambda_1) &        &                   \\[-2mm]
                   & \ddots &                   \\[-2mm]
                   &        & J_{t_n}(\lambda_n)      
\end{array}
\right), 
$$
where $J_\lambda(t)$ denotes Jordan block as in \eqref{e.block}.
Let $U$ be the set of matrices of the form  
$$
\left(
\begin{array}{ccc}
J_{t_1}(\mu_1) &        &                   \\[-2mm]
               & \ddots &                   \\[-2mm]
               &        & J_{t_n}(\mu_n)      
\end{array}
\right),
$$
where $\mu_1$, \dots, $\mu_n$ are nonzero complex numbers such that
$$
\lambda_i = \lambda_j \ \Leftrightarrow \ \mu_i = \mu_j
\qquad \text{and} \qquad
\lambda_i \asymp \lambda_j \ \Leftrightarrow \ \frac{\lambda_i}{\lambda_j} = \frac{\mu_i}{\mu_j} \, .
$$
Then $U$ is an embedded submanifold of $\GL(d,\C)$ of dimension $c(A)$.
Every $Y \in U$ has the same Jordan type as $A$,
and so, by \cref{r.jordan type and acyc}, $a(Y) = a(A)$.
We define the set $\cG(A)$ as the image of the map 
$\Psi = \Psi_A \colon \GL(d,\C) \times U \to \GL(d,\C)$ given by $\Psi(X,Y) = \Ad_X (Y)$.
Notice that $\cG(A)$ does not depend on the choice of $\tilde{A}$.
Actually $\cG(A)$ is characterized by the sizes of the Jordan blocks $t_1$, \dots, $t_n$,
the pairs $(i,j)$ such that $\lambda_i \asymp \lambda_j$ and the corresponding
roots of unity;
in particular there are countably many such sets $\cG(A)$.

Let us check that property~\ref{i.manifold} holds.
Let $\partial_1 \Psi$ and $\partial_2 \Psi$ denote the partial derivatives with
respect to $X$ and $Y$, respectively.
As we have seen in \cref{r.conjugacy_class},
the rank of $\partial_1 \Psi (X,Y)$ is equal to $d^2 - a(Y) = d^2 - a(A)$ 
for every $(X,Y)$.
On the other hand, $\partial_2 \Psi (X,Y)$ is one-to-one and therefore of rank $c(A)$.
We claim that
\begin{equation}\label{e.partials}
(\text{image of } \partial_1 \Psi(X,Y)) \cap (\text{image of } \partial_2 \Psi(X,Y)) = \{0\} ;
\end{equation}
To see this, consider the map $F \colon \Mat_{d\times d}(\C) \to \C^d$ 
that associates to each matrix the coefficients of its characteristic polynomial.
Then $\partial_1 (F \circ \Psi)(X,Y) = 0$,
while $\partial_2 (F \circ \Psi)(X,Y)$ is one-to-one. 
So \eqref{e.partials} follows.
As a result, at every point the rank of the derivative of $\Psi$ is equal to 
the sum of the ranks of the partial derivatives, that is, $d^2 - a(A) + c(A)$.
Therefore, by the Rank Theorem, the image of $\Psi$ is an immersed manifold of codimension 
$a(A) - c(A)$.
\end{proof}


\begin{proof}[Proof of \cref{l.cod rig}]
If $k=1$ then $M_1 = \GL(d,\C)$ (since $d \ge 2$), so there is nothing to prove.
Consider $k \ge 2$.
We have already shown in \cref{ss.poor_set}
that $\cP_k$ is algebraic.
Since $M_k = \{ A \in \GL(d,\C) ; \; \forall \hat{X} \in \gl(d,\C)^{k} , \ (A, \hat{X}) \in \cP_k \}$,
it is evident that $M_k$ is algebraically closed as well. 
We are left to estimate its dimension.

Take a nonsingular point $A_0$ of $M_k$ where the local dimension is maximal.
Let $D$ be the intersection of $M_k$ 
with a small neighborhood of $A_0$; it is an embedded disk. 
Each $A \in D$ has $r(A) \ge 2$;
therefore by (both parts of) \cref{t.rig},
we have $a(A) - c(A) \ge r(A) \ge k$. 
So, in terms of the sets from \cref{l.countable cover},
$$
D \subset \bigcup_{A \text{ s.t.\ } a(A) - c(A) \ge k} \cG(A).
$$
The right hand side is a countable union of immersed manifolds of codimension at least $k$.
It follows (e.g.\ by Baire Theorem)
that $D$ (and hence $M_k$) has codimension at least $k$. 
\end{proof}

\subsection{Proof of \cref{t.cod_data_C,t.cod_data_R}}

Now we apply \cref{l.cod fiber,l.cod rig} to prove one of our major results:

\begin{proof}[Proof of \cref{t.cod_data_C}]
The set $\cP_m \subset \GL(d,\C) \times [\gl(d,\C)]^m$
is homogeneous in the second factor.
Using \cref{l.pret_a_porter} together with \cref{r.homogeneous}, we
obtain that the sets 
\begin{equation}\label{e.C_j}
C_j = \big \{A \in \GL(d,\C) ; \; \codim  \cP_m (A) \le j \big\}
\end{equation}
are algebraically closed in $\GL(d,\C)$, and
$$
\codim \cP_m = \min_{j ; \; C_j \neq \emptyset} \big( j + \codim C_j \big) \, .
$$
By \cref{l.cod fiber}, we have $C_j \subset M_{m+1-j}$.
Therefore, by \cref{l.cod rig},
\begin{equation}\label{e.ultraimportant}
C_j \neq \emptyset \quad \Rightarrow \quad
\codim C_j 
\begin{cases}
\ge 0      &\text{if $j=m$,} \\
\ge m-j+1  &\text{if $j \le m-1$.}
\end{cases}
\end{equation}
So $\codim \cP_m \ge m$, as we wanted to show.
\end{proof}

The proof above only used that $\codim C_j \ge m-j$.
On the other hand, using the full power of \eqref{e.ultraimportant} we obtain: 

\begin{scho}\label{scholium}
The set of poor data in ``fat fibers'', namely
$$
\cF_m := \big\{ (A, B_1, \dots, B_m)  \in \cP_m^{(\C)} ; \; \codim \cP_m(A) \le m-1 \big\},
$$
has codimension at least $m+1$ in $\GL(d,\C) \times [\gl(d,\C)]^m$.
\end{scho}

\begin{proof}
The projection of $\cF_m$ on $\GL(d,\C)$ is $C_{m-1}$.
Use \cref{l.pret_a_porter} (together with \cref{r.homogeneous}) and \eqref{e.ultraimportant}.
\end{proof}

Next, let us consider the real case:

\begin{proof}[Proof of \cref{t.cod_data_R}]
The real part of $\cP^{(\C)}_m$ is a real algebraic set
which, in view of \cref{p.real complex dimension}, has codimension at least $m$.
Recall from \cref{ss.poor_set} that this set contains the semialgebraic set $\cP^{(\R)}_m$,
which therefore has codimension at least $m$.
Since we already knew from \cref{p.cod_data_easy_half} that $\codim \cP^{(\R)}_m \le m$,
the theorem is proved.
\end{proof}

\section{Proof of the main result} \label{s.main proof}
We now use \cref{t.cod_data_R} and transversality theorems to prove our main result.
For precise definitions and statements on the objects used in this section, see \cref{a.strat_trans}.

A \emph{stratification} is a filtration by closed subsets of a smooth manifold $X$
$$
\Si = \Si_n \supset \Si_{n-1} \supset \cdots \supset \Si_0
$$ 
such that for each $i$, the set $\Gamma_i= \Si_i \setminus \Si_{i-1}$ (where $\Sigma_{-1}:=\emptyset$) is a smooth submanifold of $X$ without boundary, and the dimension of $\Gamma_i$ decreases strictly with increasing $i$. 

We say that a $C^1$-map is {\em transverse} to that stratification if it is transverse to each of the submanifolds $\Gamma_i$. 
There are explicit, so-called {\em Whitney conditions} that guarantee that a stratification behaves nicely with respect to transversality, as the next proposition shows. A stratification satisfying those conditions is called a {\em Whiney stratification}.
By the classical~\cref{t.canonical_stratif} stated in \cref{a.strat_trans} (see for instance~\cite{GWPL}), any semi-algebraic subset of an affine space admits a canonical Whitney stratification.

We refer the reader to~\cref{a.strat_trans} for the definitions of jets, jet extensions and for a proof of the following: 
\begin{prop}\label{p.stratifiedtransversality}
Let $X$, $Y$ be $C^\infty$-manifolds without boundary. 
Let $\Si$ be a Whitney stratified closed subset of the set of $1$-jets from $X$ to $Y$. 
Then the set of maps $f \in C^2(X,Y)$ whose $1$-jet extension $j^1f$ is transverse to $\Si$ is $C^2$-open and $C^\infty$-dense in $C^2(X,Y)$ (i.e., its intersection with $C^r(X,Y)$ is $C^r$-dense, for every $2\le r\le \infty$).
\end{prop}


By \cref{t.cod_data_R}, $\cP_m^{(\R)}$ is a closed semialgebraic subset of $\GL(d,\R) \times \gl(d,\R)^m$ of codimension $m$.
The closure $\overline{\cP_m^{(\R)}}$ of $\cP_m^{(\R)}$ in  $[\Mat_{d\times d}(\R)]^{1+m}$ is a closed semialgebraic set of the affine space  $[\Mat_{d\times d}(\R)]^{1+m}$. As mentioned above, it admits a canonical Whitney stratification
$$
\overline{\cP_m^{(\R)}} = \hat \Gamma_n \supset \cdots \supset  \hat \Gamma_0 \, .
$$
The differentiable codimension of that stratification is also $m$. By locality of the Whitney conditions (see \cref{p.whitney_properties} of \cref{a.strat_trans}), this stratification restricts to a Whitney stratification of codimension~$m$:
\begin{equation}\label{e.poor_data_stratif}
\cP_{m}^{(\R)} = \Gamma_n \supset \cdots \supset \Gamma_0 \, .
\end{equation}
Since that stratification of $\overline{\cP_m^{(\R)}}$ is canonical, 
the stratification \eqref{e.poor_data_stratif} is invariant under 
polynomial automorphisms of $\GL(d,\R) \times \gl(d,\R)^m$ that preserve $\cP_{m}^{(\R)}$.

\begin{proof}[Proof of Theorem~\ref{t.main}]
Let $\cU$ be a smooth manifold without boundary and of dimension $m$. 
Given local coordinates on an open set $U\subset \cU$, the set $J^1(U,\GL(d,\R))$ of $1$-jets from $U$ to $\GL(d,\R)$ may be identified with the set $$U \times \GL(d,\R) \times \gl(d,\R)^m.$$ 
Indeed, a jet $\bJ$ represented by a pair $(u,A)$ can be identified with the point 
$$(u,A(u),B_1, \dots ,B_m)\in U \times \GL(d,\R) \times  \gl(d,\R)^m,$$
where $B_i\in \Mat_{d \times d}(\R)$ is the normalized derivative 
of $A$ at $u$, 
along the $i^\text{th}$ coordinate. 
Let us say that the $1$-jet $\bJ$ is \emph{rich}
if the datum $\bA = (A(u),B_1, \dots ,B_m)$ is rich,
or equivalently, if for sufficiently large $N$,
the input $(u,\ldots,u)\in \cU^N$ is universally regular for the system \eqref{e.proj semilin CS}.
If the jet is not rich then it is called \emph{poor}.

Define a filtration 
\begin{equation}\label{e.poor_jets_stratif}
\Sigma_n  \supset \cdots \supset \Sigma_0
\end{equation}
of the set of poor jets from $\cU$ to $\GL(d,\R)$ as follows:
a jet $\bJ$ represented as above in local coordinates by $(u,A(u),B_1, \dots ,B_m)$
belongs to $\Sigma_i$ if and only if $(A(u), B_1, \dots ,B_m)$ belongs to 
the set $\Gamma_i$ in \eqref{e.poor_data_stratif}.
We need to check that this definition does not depend on the choice of the local coordinates.
Indeed, this follows from $\cP_m^{(\R)}$ being a saturated set
(see \cref{ss.poor_set}) and from the invariance of \eqref{e.poor_data_stratif} by polynomial automorphisms.

We claim that the filtration \eqref{e.poor_jets_stratif} is a Whitney stratification 
of codimension $m$.
Indeed, the intersection of the filtration 
with the open subset $J^1(U,\GL(d,\R))$ of $J^1(\cU,\GL(d,\R))$
is identified (through a smooth diffeomorphism) with the filtration 
$$
U \times \Gamma_n \supset \cdots \supset U \times \Gamma_0.
$$
Such a filtration is still a Whitney stratification (see~\cref{p.whitney_properties} of  \cref{a.strat_trans}) of codimension $m$ in $J^1(U,\GL(d,\R))\approx U \times \GL(d,\R) \times \gl(d,\R)^m$. Covering $\cU$ by open sets $U$, we deduce that \eqref{e.poor_jets_stratif} is a Whitney stratification of codimension $m$ in $J^1(\cU,\GL(d,\R))$.


Applying Proposition~\ref{p.stratifiedtransversality}, we obtain a $C^2$-open
$C^\infty$-dense set $\cO \subset C^2(\cU,\GL(d,\C))$ formed by maps $A$ 
that are transverse to the stratification \eqref{e.poor_jets_stratif} 
of the set of poor jets.
Since the codimension of the stratification equals the dimension of $\cU$,
if $A \in \cO$ then the points $u$ for which $j^1 A(u)$ is poor form a $0$-dimensional set.
This proves~\cref{t.main}.
\end{proof}



\appendix  

\section{The case of one-dimensional input}\label{a.dim 1}

As we explained in \cref{ss.overview},
this appendix contains a basically independent discussion of the case where $m = \dim \cU$ equals $1$.
The prerequisites are all contained in \cref{s.prelim_poor,ss.acyclicity}. 

\medskip

\subsection{Elementary constraints}\label{ss.unconstrained}
The material of this subsection is also used in \cref{a.generic singular}.

\medskip

An \emph{elementary constraint} in the variables $\lambda_1$, \dots, $\lambda_d$ 
is a relation $p=0$ where
$p$ is an irreducible factor of a polynomial 
of the form $\lambda_i \lambda_\ell - \lambda_j \lambda_k$.
Every elementary constraint can be written, after a permutation of the indices $1,\dots,d$, 
as one of the following:
\begin{equation}\label{e.canonical_constr}
\lambda_1 \lambda_3 = \lambda_2^2 , \qquad
\lambda_1 \lambda_4 = \lambda_2 \lambda_3, \qquad  
\lambda_1 = -\lambda_2, \qquad 
\lambda_1 = \lambda_2 \, ,
\end{equation}
which will be called the \emph{canonical constraints} respectively of type $1$, $2$, $3$, $4$.
The \emph{type} of elementary constrained is defined as 
the (unique) type of the associated canonical constraint.

We say that a matrix $A\in \GL(d,\R)$ is \emph{unconstrained} if its eigenvalues, counted with multiplicity, satisfy no elementary constraint.
(Equivalently, $\Ad_A$ has the maximal possible number of distinct eigenvalues, namely, $d^2-d+1$.)
 

Let us see that the converse of \cref{l.easy_poor_data} holds for unconstrained matrices:

\begin{lemma}\label{l.easy_fiber}
Suppose that the datum 
$\bA = (A, B_1, \dots, B_m) \in \GL(d,\K) \times \gl(d,\K)^m$ is poor
and that the matrix $A$ is unconstrained.
Then $\bA$ is conspicuously poor.
\end{lemma}

\begin{proof}
Suppose $A$ is unconstrained.
In particular, $A$ has simple spectrum.
With a change of basis we can assume that $A$ is diagonal.

Now suppose that $\bA = (A, B_1, \dots, B_m)$ is not conspicuously poor.
This means that for each off-diagonal position there is at least of the matrices $B_k$
that has a non-zero entry in that position.
(Notice that this fact does not depend on the change of basis chosen before.)

Since $A$ is unconstrained, the values $\lambda_i \lambda_j^{-1}$, 
where $(i,j)$ runs on the matrix positions outside the diagonal, are pairwise different, 
and all different from $1$. 
Recall that one can always (using Lagrange formula) find a polynomial 
whose values at finitely many different points are prescribed.
Restricting to polynomials $f$ such that $f(1)=0$,
it follows from \eqref{e.poly entry}
that the space $\Lambda(\bA)$ contains all matrices $(y_{ij})$ 
with only zeros in the diagonal.
Since, by definition, $\Lambda(\bA)$ also contains the identity matrix,
it contains all Toeplitz matrices.
So $\Lambda(\bA)$ is transitive, i.e., $\bA$ is not poor.
This proves the \lcnamecref{l.easy_fiber}.
\end{proof}

\subsection{Effective richness criteria for the case $m=1$}
We will describe an explicit set of rich data $(A,B)$ whose complement has codimension $1$.
In order to avoid technicalities, we will be sometimes informal, especially regarding questions of transversality.

\medskip


Let us say that a matrix $A\in \GL(d,\R)$ is \emph{$(i)$-constrained}, where $1\leq i\leq 4$, if:
\begin{itemize}
	\item its eigenvalues, counted with multiplicity, satisfy exactly one elementary constraint, which is a type $i$ constraint,
	\item if there is a type $4$ constraint between the eigenvalues, then the matrix $A$ is \emph{not} diagonalizable.
\end{itemize}
Suppose that there is no $i$ for which the matrix $A$ is $(i)$-constrained; then:
\begin{itemize}
\item either $A$ is unconstrained, i.e., its eigenvalues (with multiplicity) satisfy no elementary constraint;
\item or the eigenvalues of $A$ satisfy at least two elementary constraints;
\item or $A$ has a (multiple) eigenvalue corresponding to at least two Jordan blocks.
\end{itemize}
If either of the last two cases hold, we say that $A$ is \emph{multiconstrained}.

\begin{prop}\label{p.cod_constraints}
\begin{enumerate}
\item The complement of the set of unconstrained matrices has codimension $1$ in $\GL(d,\R)$.
\item The set of multiconstrained matrices has codimension $2$ in $\GL(d,\R)$.
\end{enumerate}
\end{prop}

\begin{proof}[Informal proof]
Matrices that are not unconstrained have at least one constraint on their eigenvalues, so the corresponding set has codimension~$1$.

Matrices that are multiconstrained either have at least two constraints on their eigenvalues, or 
are derogatory, i.e., have an eigenvalue corresponding to at least two Jordan blocks.
In both cases, the corresponding set has codimension~$2$.
\end{proof}

Let us define \emph{adapted bases} for matrices $A$ that are not multiconstrained:
\begin{itemize}
\item If $A$ is unconstrained then an adapted basis is a basis of eigenvectors.
\item If $A$ is $(i)$-constrained, for $i = 1$, $2$, or $3$ 
then an adapted basis is an (ordered) basis of eigenvectors such that
the corresponding eigenvectors $\lambda_1,\ldots \lambda_d$ 
satisfy the canonical type $i$ constraint.
\item If $A$ is $(4)$-constrained then an \emph{adapted basis} for $A$ 
is a basis in which $A$ is written in the following \emph{modified Jordan form} 
$$
\left(
\begin{array}{cc|ccc}
\lambda_1 & \lambda_1 &          &       &           \\
0         & \lambda_1 &          &       &           \\
\hline
          &           &\lambda_3 &       &           \\
          &           &          &\ddots &           \\
          &           &          &       &\lambda_d  
\end{array}
\right).
$$
\end{itemize}
Obviously, such adapted bases always exist.

If a matrix $A$ is $(i)$-constrained then we say that a 
$d \times d$ matrix $B$ is a \emph{good match} for $A$, 
if there is an adapted basis for $A$ in which 
it writes as $B = (b_{ij})$, 
where all nondiagonal entries $b_{ij}$ are nonzero and 
if $b_{11}\neq b_{22}$, in the particular case where $A$ is $3$-constrained.


The usefulness of this definition is explained by the following 
\cref{p.rich_pair,p.cod_RH}.
(Actually, the definition of a good match matrix is stronger than necessary for the validity of the propositions below. But in order to avoid complications, we chose a condition that works for all types of constraints.)


\begin{prop}\label{p.rich_pair}
If $A$ is not multiconstrained and $B$ is a good match for $A$ then the pair $(A,B)$ is rich.
\end{prop}

In other words, $\cP^{(\C)}_1$ is contained in the following set:
\begin{multline}\label{e.explicit}
\cE:=\big\{ (A,B) \in \GL(d,\C) \times \gl(d,\C) ; \; 
\text{either $A$ is multiconstrained} \\
\text{or $A$ is not multiconstrained but $B$ is not a good match for $A$} \big\}.
\end{multline}

\begin{prop}\label{p.cod_RH}
\begin{enumerate}
\item The set $\cE$ has codimension $1$.
\item The set $\{(A,B) \in \cE; \; A \text{ is not unconstrained}\}$ has codimension $2$.
\end{enumerate}
\end{prop}

\begin{proof}[Informal proof]
\Cref{p.cod_RH} follows from \cref{p.cod_constraints}
and from the fact that for each matrix $A$ that is not multiconstrained,
the set of $B$'s that are not good matches for $A$
has positive codimension in $\gl(d,\C)$.
\end{proof}

\cref{t.cod_data_C} in the case $m=1$ follows from the propositions above.
Therefore the other main results (\cref{t.cod_data_R,t.main,t.addendum,t.main_C}) in the $m=1$ 
case also follow from  the propositions.
For any of these results, the propositions give extra information of practical value: 
with the explicit definition of the set $\cE$ in \eqref{e.explicit}, we know 
which $1$-jets should be avoided in \cref{t.main}, for example.
The discussion given in \cref{a.generic singular}
also applies; it gives explicit conditions on the $2$-jet extension of the map 
$A \colon \cU \to \GL(d,\R)$ that ensure that $A$ satisfies the conclusions of \cref{t.main,t.addendum}.

\begin{proof}[Proof of \cref{p.rich_pair}]
Let $A$ and $B$ satisfy the hypotheses.
We need to show that $\Lambda(A,B) = \sorb_{\Ad_A}(\Id,B)$ is a transitive subspace of $\gl(d,\C)$.
Let $\Gamma = \sorb_{\Ad_A}(B)$,
so that $\Lambda(A,B) = \{\Id\} \vee \Gamma$.

The matrix $A$ is not multiconstrained and so has an adapted basis as above.
We change the basis so that $A$ and $B$ are ``canonical''.

The proof is divided in cases according to the type of constraint.
Except for the $(4)$-constrained case, the matrix $A$ is diagonal,
and so the space $\Gamma$ is described by \eqref{e.poly entry}.

\medskip\noindent\emph{Unconstrained case:}
It follows from \cref{l.easy_fiber}
that if $A$ is unconstrained and diagonal then the only way for the pair $(A,B)$ to be poor
is that $B$ has an off-diagonal zero entry.
(The reader should review the proof of \cref{l.easy_fiber}.)

\medskip\noindent\emph{$(1)$-constrained case:}
We see that the adjoint $\Ad_A$ has two eigenvalues (different from $1$) 
of multiplicity $2$, namely
$\lambda_1 \lambda_2^{-1} = \lambda_2 \lambda_3^{-1}$ and 
$\lambda_2 \lambda_1^{-1} = \lambda_3 \lambda_2^{-1}$.
By the same reasoning as in the unconstrained case, it follows that $\{\Id\} \vee \Gamma$ contains the space
$$
\big \{ (y_{ij}) \in \gl(d,\C) ; \;  y_{11} = \cdots = y_{dd}\, , \;
b_{12}^{-1} y_{12} = b_{23}^{-1} y_{23}\, , \;  
b_{21}^{-1} y_{21} = b_{32}^{-1} y_{32}
\big \}.
$$
This is a generalized Toeplitz space, 
and so by \cref{ex.gen Toeplitz} it is transitive.

\medskip\noindent\emph{$(2)$-constrained case:}
The reasoning is very similar to that of the $(1)$-constrained case,
but now the adjoint has four eigenvalues (different from $1$) of multiplicity $2$.
The space $\Lambda(A,B)$ contains the following subspace:
\begin{align*}
\big \{ (y_{ij}) \in \gl(d,\C) ; \;	
&y_{11} = \cdots = y_{dd}\, , \; 
 b_{13}^{-1} y_{13} = b_{24}^{-1} y_{24}\, , \\
&b_{12}^{-1} y_{12} = b_{34}^{-1} y_{34}\, , \; 
 b_{21}^{-1} y_{21} = b_{43}^{-1} y_{43}\, , \;
 b_{31}^{-1} y_{31} = b_{34}^{-1} y_{34}
\big \}.
\end{align*}
Again, this is a generalized Toeplitz space, 
and so it is transitive.

\medskip\noindent\emph{$(3)$-constrained case:}
This case is a little different from the two previous ones.
The adjoint has an eigenvalue $-1$ of multiplicity $2$.
Recalling that $b_{11}$ and $b_{22}$ are different, and making use of the identity matrix,
we see that $\Lambda(A,B)$ contains the following subspace:
$$
\tilde \Gamma = 
\big \{ (y_{ij}) \in \gl(d,\C) ; \;  y_{33} = \cdots = y_{dd}\, , \;
b_{12}^{-1} y_{12} = b_{21}^{-1} y_{21}
\big \}.
$$
This is not a generalized Toeplitz space.
However, consider the linear automorphism $S$ that swaps the first two elements of the canonical basis of $\C^n$, and fixes the others.
Then
$$
S \cdot \tilde \Gamma = 
\big \{ (z_{ij}) \in \gl(d,\C) ; \;  z_{33} = \cdots = z_{dd}\, , \;
b_{12}^{-1} z_{22} = b_{21}^{-1} z_{11}
\big \}
$$
is a generalized Toeplitz space.
By \cref{r.transitivity trick}, the space $S \cdot \tilde \Gamma$ is transitive, and so are $\tilde \Gamma$ and $\Lambda(A,B)$.

\medskip\noindent\emph{$(4)$-constrained case:}
This case is more involved because the operator $\Ad_A$ is not diagonalizable.
We will explain its Jordan form.
Let us explain visually how $\Ad_A$ acts: given any matrix, decompose it into blocks $C_{ij}$ as in the following picture $$\left(
\begin{array}{cc|c|c|ccc|c}
C_{22}& &\!\!C_{23}\!\!\! &\!\!C_{24}\!\!\! &  &\hdots & &\!\! C_{2d}\!\!\!\!\\
& & & & & & &\\
\hline C_{32} &&\!\!C_{33}\!\!\!& & &\hdots  &\\
\hline C_{42} & & &\!\!C_{44}\!\!\! & &\hdots  &\\
\hline & & & & & & &\\
\vdots& &\vdots &\vdots & &\ddots & &\\
& & & & & & &\\
\hline C_{d2}& & & & &\hdots & &\!\!C_{dd}\!\!\!\!\\
\end{array}
\right)
$$
where the block $C_{22}$ is a $2\times 2$ matrix, the blocks $C_{2j}$ are $2\times 1$, the blocks $C_{i2}$ are $1\times 2$ and the  others are $1\times 1$. Then, the operator $\Ad_A$ leaves invariant the space $\Gamma_{ij}$ of matrices whose nonzero coefficients lie inside the block $C_{ij}$.

Let us use notations $J_t(\lambda)$ from \eqref{e.block} and  $E_{i,j}$ from \eqref{e.Eij}.
It is easily computed that the operator $\Ad_A$ has the following properties:
\begin{itemize}
	
\item the matrix of ${\Ad_A}{|\Gamma_{11}}$ with respect to the basis formed by 
$M_1 = -2 E_{12}$, $M_2 = E_{11}-E_{12}-E_{22}$, $M_3 = E_{21}$, $M_4 = E_{11}+E_{22}$
is $\begin{pmatrix}J_3(1) & 0 \\ 0 & 1 \end{pmatrix}$.


\item For any $j\geq 3$, the matrix of ${\Ad_A}{|\Gamma_{2j}}$ 
with respect to the basis formed by $\lambda_2 \lambda_j^{-1}E_{1,j}$ and $E_{2,j}$
(where we use the notation $E_{i,j}$ from \eqref{e.Eij})
is $J_2(\lambda_2 \lambda_j^{-1})$.

\item For any $i\geq 3$, the matrix of ${\Ad_A}{|\Gamma_{i2}}$ 
with respect to the basis formed by $-\lambda_i \lambda_2^{-1}E_{i,1}$ and $E_{i,2}$
is $J_2(\lambda_i \lambda_2^{-1})$.

\item For $3\leq i,j\leq d$, the matrix of ${\Ad_A}{|\Gamma_{ij}}$ 
with respect to the basis formed by the single vector $E_{ij}$ is $(\lambda_i\lambda_j^{-1})$.
 
\item The spaces $\Gamma_{ij}$, for $2\leq i,j\leq d$ have  respective spectra $\{\lambda_i\lambda_j^{-1}\}$, which for $i\neq j$ are pairwise disjoint and different from $\{1\}$. 
\end{itemize}

The concatenation of the bases described above gives a Jordan basis for ${\Ad_{A}}$. 
Now take a matrix $B$ that is a good match for $A$,
and consider its expression as a linear combination of the elements of that Jordan basis.
One easily checks that all coefficients in this linear combination are nonzero, except possibly
the coefficients of the vectors $M_1$, $M_2$, $M_4$ and the vectors $E_{ii}$, for all $3\leq i\leq d$.
Consider now the splitting $\Mat_{d \times d}(\C) =  V \oplus \Delta$, where $\Delta$ is the subspace $\C M_4 \oplus E_{33} \oplus \ldots \oplus E_{dd}$ of the space of diagonal matrices, and 
$V$ is the space spanned by all other elements of the above Jordan basis. Note that
\begin{align*}
V=(\C M_1+\C M_2+\C M_3)\oplus \left(\bigoplus_{2\leq i,j\leq d\atop i\neq j}\Gamma_{ij}\right)
\end{align*}
is a decomposition of $V$ into $\Ad_A$-invariant subspaces with pairwise disjoint spectra. 
Let $\pi$ be the projection onto $V$ along $\Delta$.
Using \cref{l.sum}, we see that $\pi(B)$ is a cyclic vector for $\Ad_A | V$.
So, using the $\Ad_A$-invariance of the spaces $V$ and $\Delta$, we have
$$
\pi (\Gamma) = 
\pi \big( \sorb_{\Ad_A} (B) \big) = 
\sorb_{\Ad_A} (\pi(B)) = V. 
$$
Note that $V$ contains the matrices $E_{ij}$, for all $i\neq j$, hence $\{\Id\} \vee V$ is a generalized Toeplitz space. As $\pi$ projects along a subspace of diagonal matrices,  $\{\Id\} \vee \Gamma$ is again a generalized Toeplitz space and in particular is a transitive space.

\medskip

We have considered the four types, and \cref{p.rich_pair} is proved.
\end{proof}

\section{Some general facts on dimensions of algebraic sets}\label{a.algebraic}

In this appendix we prove \cref{l.pret_a_porter,t.schubert},
which were used in \cref{s.cod proof}.
\Cref{l.pret_a_porter} is a simple consequence of standard theorems in algebraic geometry, 
but for the reader's convenience let us spell out the details. 
\Cref{t.schubert} follows from intersection theory of the Grassmannians (``Schubert calculus'').
We tried to make the exposition the least technical as possible, to make it accessible to non-experts (like ourselves).

\subsection{Fiberwise dimension estimate}


\begin{proof}[Proof of \cref{l.pret_a_porter}]
In what follows, all topologies are Zariski.
We will prove the equivalent ``dual form'' of the lemma, namely, that
the sets 
$$
Y_k = \big\{ y \in \pi(X) ; \; \dim \pi^{-1}(y) \ge k \big\}
$$
are algebraically closed in $Y$,
and 
\begin{equation}\label{e.dual_formula}
\dim X = \max_{k ; \;  Y_k \neq \emptyset}
\big( k + \dim Y_k \big).
\end{equation}

First, the sets $X_k = \{x \in X ; \; \dim \pi^{-1}(\pi(x)) \ge k\}$ are closed.
(see \cite[Thrm.~11.12]{Harris}).
So, by \cref{p.projection}, $Y_k = \pi(X_k)$ is closed.

For each $k$ with $X_k \neq \emptyset$, 
let $X_{k,i}$ denote the irreducible components of $X_k$.
Let 
$$
\mu(k,i) = \min_{x \in X_{k,i}} \dim \pi^{-1}(\pi(x)) \, .
$$
Then, by \cite[Thrm.~11.12]{Harris} (and the fact that 
taking closures does not affect dimension) 
we have 
$$
\dim X_{k,i} = \mu(k,i) + \dim \pi(X_{k,i}) \, .
$$
By definition, $\mu(k,i) \ge k$;
moreover equality holds unless $X_{k,i} \subset X_{k+1}$.
So 
$$
X_{k,i} \not \subset X_{k+1} \ \Rightarrow \ 
\dim X_{k,i} 
=  k + \dim \pi(X_{k,i})  
\le k+ \dim Y_k \, .
$$
Since $X = \bigcup_{X_{k,i} \not \subset X_{k+1}} X_{k,i}$,
this proves the $\le$ inequality in \eqref{e.dual_formula}.

To prove the converse inequality, fix any $k$ with $Y_k \neq \emptyset$.
Find $i$ such that $\dim \pi(X_{k,i}) = \dim Y_k$.
Then
$$
\dim X \ge \dim X_{k,i} =  \mu(k,i) + \dim Y_k \ge k + \dim Y_k.
$$
This proves \eqref{e.dual_formula} and hence the \lcnamecref{l.pret_a_porter}.
\end{proof}

\subsection{A particular case of \cref{t.schubert}}\label{ss.particular}

Let us begin the proof of \cref{t.schubert}.
For the reader's convenience we recall the notations and the statement.

If $M \in \Mat_{n \times m}(\C)$, let $\col M \subset \C^n$ denote the column space of $M$.
A set $X \subset \Mat_{n \times m}(\C)$ is called \emph{column-invariant}
if 
$$
\left.
\begin{array}{c}
M \in X \\ 
N \in \Mat_{n \times m}(\C)\\
\col M = \col N
\end{array}
\right\} \ \Rightarrow \ 
N \in X.
$$
So a column-invariant set $X$ is characterized by its set of column spaces.
We enlarge the latter set by including also subspaces, thus defining:
$$
\ldbrack X \rdbrack := \big\{ E \text{ subspace of } \C^n ; \; E \subset \col M \text{ for some } M \in X \big\}.
$$
Then we have:

\begin{repeatedthm} 
Let $X \subset \Mat_{n \times m}(\C)$ be an algebraically closed, column-invariant set.
Suppose $E$ is a vector subspace of $\C^n$
that does not belong to $\ldbrack X \rdbrack$.
Then
$$
\codim X \ge m + 1 - \dim E \, .
$$
\end{repeatedthm}

It is obvious that the algebraicity hypothesis is indispensable. 

\medskip

Define
\begin{equation}\label{e.R_k}
R_k := \big \{ A \in \Mat_{n \times m}(\C) ; \; \rank A \le k \big\} \, .
\end{equation}
We recall (see \cite[Prop.~12.2]{Harris}) that 
this is an irreducible algebraically closed set of codimension
\begin{equation}\label{e.cod_R_k}
\codim R_k = (m-k)(n-k) \qquad \text{if } 0 \le k \le \min(m,n).
\end{equation}

\begin{proof}[Proof of \cref{t.schubert} in the case $E = \C^n$]
If $E = \C^n$ then the hypothesis $\C^n \not\in \ldbrack X \rdbrack$
means that $X \subset R_{n-1}$.
We can assume that $n-1 \le m$, otherwise the conclusion of the 
\lcnamecref{t.schubert} is vacuous.
Thus $\codim X \ge \codim R_{n-1} = m + 1 - n$, as we wanted to show.
\end{proof}


\subsection{Reduction to a property of Grassmannians} \label{ss.reduction}

As we will see, to prove \cref{t.schubert} 
it is sufficient to prove a dimension estimate (\cref{t.schubert2} below) 
for certain subvarieties of a Grassmaniann.

\subsubsection{Grassmannians}

Given integers $n > k \ge 1$, the \emph{Grassmanniann} $G_k(\C^n)$ 
is the set of the vector subspaces of $\C^{n}$ of 
dimension $k$.

The Grassmannian can be interpreted as a subvariety of a higher dimensional complex projective space
using the \emph{Pl\"ucker embedding} $G_k(\C^n) \to P(\bigwedge^k \C^n)$,
which maps each $V \in G_k(\C^n)$ to $[v_1 \wedge \cdots \wedge  v_k]$,
where $\{v_1, \dots, v_k\}$ is any basis of $V$/
This is clearly an one-to-one map.
It can be shown (see e.g.~\cite[p.~61ff]{Harris}) that the image is 
an algebraically closed subset 
of $P(\bigwedge^k \C^n)$.
Its dimension is 
\begin{equation}\label{e.dim_G}
\dim G_k(\C^n) = k(n-k).
\end{equation}

If $E \subset \C^n$ is a vector space with $\dim E = e \le k$ then 
we consider the following subset of $G_k(\C^n)$: 
\begin{equation}\label{e.special schubert}
S_k(E) := \big\{V \in G_k(\C^n) ; \; V \supset E \big\}.
\end{equation}
(This is a Schubert variety of a special type, as we will see later.)
Since any $V \in S_k(E)$ can be written as $E \oplus W$ for some $V \subset W^\perp$,
we see that $S_k(E)$ is homeomorphic to $G_{k-e}(\C^{n-e})$.

We will 
show that an algebraic set that avoids $S_k(E)$ cannot be too large:

\begin{thm}\label{t.schubert2}
Fix integers $1 \le e \le k < n$. 
Suppose that $Y$ is an algebraically closed subset of $G_k(\C^n)$
that is disjoint from $S_k(E)$,
for some $e$-dimensional subspace $E \subset \C^n$.
Then $\codim Y \ge k + 1 - e$.
\end{thm}

\subsubsection{Proof of \cref{t.schubert} assuming \cref{t.schubert2}} \label{sss.reduction}

Assuming \cref{t.schubert2} for the while, let us see how it yields \cref{t.schubert}.

Recalling notation \eqref{e.R_k}, define the quasiprojective variety
$$
\hat{R}_k := R_k \setminus R_{k-1} \, .
$$
We define a map $\pi_k \colon \hat{R}_k \to G_k(\C^n)$
by $A \mapsto \col A$.

\begin{lemma}\label{l.projection}
If $X$ is an algebraically closed column-invariant subset of $\hat{R}_k$
then $Y = \pi_k(X)$ is algebraically closed subset of $G_k(\C^n)$,
and the codimension of $Y$ inside $G_k(\C^n)$ is the same as the codimension of $X$ inside $\hat{R}_k$.
\end{lemma}

\begin{proof}
First, let us see that $\pi_k \colon \hat{R}_k \to G_k(\C^n)$ is a regular map.
We identify $G_k(\C^n)$ with the image of the Pl\"ucker embedding.
In a Zariski neighborhood of each  matrix $A \in \hat{R}_k$, 
the map $\pi_k$ can be defined as $A \mapsto [a_{j_1} \wedge \dots \wedge a_{j_k}]$
for some $j_1 < \dots < j_k$, where $a_j$ is the $j^\text{th}$ column of $A$.
This shows regularity.
	
Next, let us see that $Y = \pi_k (X)$ is closed with respect to the classical (not Zariski)
topology.
Consider the subset $K$ of $X$ formed by the matrices $A \in \hat{R}_k$ 
whose first $k$ columns form an orthonormal set, and whose $m-k$ remaining columns are zero.
Then $K$ is compact (in the classical sense), and thus so is $\pi_k(K)$.
But column-invariance of $X$ implies that $\pi_k(K) = Y$, so $Y$ is closed (in the classical sense).

It follows (see e.g.~\cite[p.39]{Harris}) 
from regularity of $\pi_k$ is regular that the set $Y$ is constructible, i.e., 
it can be written as
$$
Y = \bigcup_{i=1}^{p} Z_i \setminus W_i \, ,
$$
where $Z_i \varsupsetneq W_i$ are algebraically closed subsets of $G_k(\C^n)$.
We can assume that each $Z_i$ is irreducible.
It follows from \cite[Thrm.~2.33]{Mumford} that $\overline{Z_i \setminus W_i} = Z_i$,
where the bar denotes closure in the classical sense.
In particular, $Y = \overline{Y} = \bigcup_{i=1}^{p} Z_i$,
showing that $Y$ is algebraically closed.

We are left to show the equality between codimensions.
Since the codimension of an algebraically closed set equals the minimum of the codimensions of its components, we can assume that $X$ is irreducible.

By column-invariance of $X$,
for each $y\in Y$, the whole fiber $\pi^{-1}(y)$ is contained in $X$.
All those fibers have the same dimension $\mu = km$.
By \cite[Thrm.~11.12]{Harris}, $\dim X = \dim Y + km$.
By \eqref{e.cod_R_k} and \eqref{e.dim_G}, we have $\dim \hat{R}_k - \dim G_k = km$,
so the claim about codimensions follows.
\end{proof}

\begin{proof}[Proof of \cref{t.schubert}]
Let $X \subset \Mat_{n \times m}(\C)$ be a nonempty algebraically closed, column-invariant set.
Suppose $E$ is a vector subspace of $\C^n$ that does not belong to $\ldbrack X \rdbrack$.
Let $e = \dim E$.
We can assume $e > 0$ (otherwise the result is vacuously true), 
and $e<n$ (because the case $e=n$ was already considered in \S~\ref{ss.particular}).

Notice that $X \subset R_{n-1}$.
Let 
$$
X_k := X \cap \hat{R}_k \quad \text{and} \quad 
Y_k := \pi_k(X_k) , \quad \text{for } 0 \le k \le \min(m,n-1).
$$
For every $k$ with $e \le k < n$, the set $Y_k$ is disjoint from 
the set $S_k(E)$ defined by \eqref{e.special schubert}.
In view of \cref{l.projection} and \cref{t.schubert2}, we have
$$
\codim_{\hat{R}_k} X_k =  \codim Y_k \ge k + 1 - e \, .
$$
So the codimension of $X_k$ as a subset of $\Mat_{n\times m}(\C)$ is
\begin{align*}
\codim X_k &=   \codim \hat{R}_k + \codim_{\hat{R}_k} X_k \\
           &\ge (m-k)(n-k) + k + 1 - e =: f(k) \, .
\end{align*}
The function $f(k)$ is decreasing on the interval $0 \le k \le \min(m,n-1)$.
Therefore:
\begin{multline*}
\codim X 
=   \min_{0 \le k \le \min(m,n-1)} \codim X_k  
\ge \min_{0 \le k \le \min(m,n-1)} f(k) \\
= f(\min(m,n-1)) 
= m + 1 - e,
\end{multline*}
as claimed.
This proves \cref{t.schubert} modulo \cref{t.schubert2}.
\end{proof}

The proof of \cref{t.schubert2} will be given in \S~\ref{ss.end},
after we explain the necessary tools in \S\S~\ref{ss.schubert}, \ref{ss.intersection}.

\subsection{Schubert calculus} \label{ss.schubert}

Here we will outline some facts about the intersection of Schubert varieties.
The readable expositions \cite{Blasiak,Vakil} contain more information.

\medskip

A (complete) flag 
in $\C^{n}$ is a sequence of subspaces $F_0 \subset F_1 \subset \cdots \subset F_{n}$
with $\dim F_j = j$. We denote $F_\bullet = \{F_i\}$.

Given $V \in G_k (\C^n)$, 
its \emph{rank table} (with respect to the flag $F_\bullet$)
is the datum $\dim (V \cap F_j)$, $j=0,\dots,n$.
The \emph{jumping numbers} are
the indexes $j \in \{1,\dots,n\}$ such that 
$\dim (V \cap F_j) - \dim (V \cap F_{j-1})$ is positive (and thus equal to $1$).
Of course, if one knows the jumping numbers, one know the rank table and vice-versa.
Let us define a third way to encode this information:
Consider a rectangle of height $m$ and width $n-m$, divided in $1 \times 1$ squares.
We form a path of square edges:
Start in the northeast corner of the rectangle.
In the $j^\text{th}$ step ($1 \le j \le n$),
if $j$ is a jumping number then we move one unit in the south direction,
otherwise we move one unit in the west direction.
Since there are exactly $k$ jumping numbers, 
the path ends at the southwest corner of the rectangle.
The \emph{Young diagram} of $V$ with respect to the flag $F_\bullet$ is 
the set of squares in the rectangle that lie northwest of the path.
We denote a Young diagram by $\lambda = (\lambda_1, \lambda_2, \dots, \lambda_k)$,
where $\lambda_i$ is the number of squares in the $i^\text{th}$ row (from north to south).
Its \emph{area} $\lambda_1+\cdots+\lambda_k$ is denoted by $|\lambda|$.

\begin{example}\label{ex.Young}
Here is a possible rank table with $k=5$, $n=12$;
the jumping numbers are underlined:
\begin{center}
\begin{tabular}{rrrrrrrrrrrrrr}
$j = $                 & 0 & 1 & 2 & \underline{3} & 4 & 5 & \underline{6} & 7 & \underline{8} & \underline{9} & 10 & \underline{11} & 12 \\
$\dim (W \cap F_j)=$ & 0 & 0 & 0 &       1 & 1 & 1 &       2 & 2 &        3 &      4 &  4 &        5 &  5 \\
\end{tabular}
\end{center}
The associated path in the rectangle is:
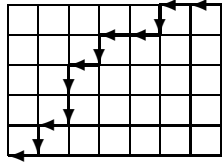
\setlength{\unitlength}{.4cm}
\begin{center}
\raisebox{-3\unitlength}{
\begin{picture}(7,5)
\thinlines
\put(0,0){\grid(7,5)(1,1)}
\thicklines
\put(7,5){\vector(-1,0){1}} 
\put(6,5){\vector(-1,0){1}} 
\put(5,5){\vector(0,-1){1}} 
\put(5,4){\vector(-1,0){1}} 
\put(4,4){\vector(-1,0){1}} 
\put(3,4){\vector(0,-1){1}} 
\put(3,3){\vector(-1,0){1}} 
\put(2,3){\vector(0,-1){1}} 
\put(2,2){\vector(0,-1){1}} 
\put(2,1){\vector(-1,0){1}} 
\put(1,1){\vector(0,-1){1}} 
\put(1,0){\vector(-1,0){1}} 
\end{picture}
}
\end{center}
and so the Young diagram is $$\lambda=\tiny{\yng(5,3,2,2,1)} = (5,3,2,2,1).$$
\end{example}

In general, we have:
\begin{itemize}
\item $\lambda = (\lambda_1, \dots, \lambda_k)$ is a possible Young diagram if and only if
$n-k \ge \lambda_1 \ge \dots \ge \lambda_k \ge 0$.
\item If $j_1 < \dots < j_k$ are the jumping numbers then $\lambda_i = n-k-j_i+i$.
\end{itemize}

The set of $V \in G_k(\C^n)$ that have a given Young diagram 
$\lambda$ 
is called a 
\emph{Schubert cell}, denoted by $\Omega(\lambda)$ or $\Omega(\lambda,F_\bullet)$.
Each Schubert cell is a topological disk of real codimension $2|\lambda|$.
The Schubert cells (for a fixed flag) give a CW decomposition of the space $G_k(\C^n)$.
The closure of $\Omega(\lambda)$ (in either classical or Zariski topologies) is
the set of $V \in G_k(\C^n)$ such that $\dim (V \cap F_{j_i}) \ge i$ for each $i=1,\ldots,n$
(where $j_1 < \dots < j_k$ are the jumping numbers associated to $\lambda$).
These sets are closed irreducible varieties,
called \emph{Schubert varieties}. (See e.g.~\cite[\S9.4]{Fulton}.)

\begin{example}\label{ex.special schubert}
If $E \subset \C^n$ is a subspace with $\dim E = e \le k$ then 
the set $S_k(E)$ defined by \eqref{e.special schubert} is 
a Schubert variety $\bar\Omega(\lambda,F_\bullet)$,
where $F_\bullet$ is any flag with $F_e = E$ and
\setlength{\unitlength}{.25cm}
\begin{equation}\label{e.special young}
\lambda = \big( \underbrace{n-k,\dots,n-k}_{e \text{ times}}, \underbrace{0,\dots,0}_{k-e \text{ times}} \big) = 
\raisebox{-4\unitlength}{
\begin{picture}(12,8)
\thinlines
\put(0,0){\grid(12,8)(1,1)}
\multiput(0,5)(0,1){3}{\multiput(0,0)(1,0){12}{\lightdiagonalpattern}}
\end{picture}
}
\end{equation}
\end{example}

Let $A^*(k,n)$ denote the set of formal linear combinations with integer coefficients of
Young diagrams in the $k \times (n-k)$ rectangle.
This is by definition an abelian group.

\begin{prop}
There is a second binary operation called the \emph{cup product} and denoted by the symbol $\cupro$ 
that makes $A^*(k,n)$ a commutative ring, and 
is characterized by the following properties:

If $\lambda$ and $\mu$ are Young diagrams with respective areas $r$ and $s$
then their cup product is of the form:
$$
\lambda \cupro \mu = \nu_1 + \cdots + \nu_N \, .
$$
where $\nu_1$, \dots, $\nu_N$ are Young diagrams with area $r+s$
(possibly with repetitions, possibly $N=0$).
Moreover, there are flags $F_\bullet$, $G_\bullet$, $H^{(i)}_\bullet$ such that
the manifolds $\bar\Omega(\lambda,F_\bullet)$ and $\bar\Omega(\mu,G_\bullet)$ are transverse
and their intersection is $\bigcup \bar\Omega(\nu_i,H^{(i)}_\bullet)$. 
\end{prop}

\begin{example}
Working in $A^*(2,4)$,
let us compute the products of the Young diagrams 
$\lambda = {\tiny \yng(2)}$ and $\mu={\tiny \yng(1,1)}$.
Fix a flag $F_\bullet$.
Then
$\bar\Omega(\lambda, F_\bullet)$ is the set of $W \in G_2(\C^4)$ that contain $F_1$,
and $\bar\Omega(\mu, F_\bullet)$ is the set of $W \in G_2(\C^4)$ that are contained in $F_3$.
Take another flag $G_\bullet$ which is in general position with respect to $F_\bullet$,
that is $F_i \cap G_{4-i} = \{0\}$.
Then:
\begin{itemize}
\item The set
$\bar\Omega(\lambda, F_\bullet) \cap \bar\Omega(\lambda, G_\bullet)$ 
contains a single element, namely $F_1 \oplus G_1$,
and thus equals $\bar\Omega((2,2),H_\bullet) = \{H_2\}$ for an appropriate flag $H_\bullet$.
This shows that
$\lambda \cupro \lambda = {\tiny \yng(2,2)}$.

\item The space $F_3 \cap G_3$ is $2$-dimensional and thus
is the single element of 
$\bar\Omega(\mu, F_\bullet) \cap \bar\Omega(\mu, G_\bullet)$.
So
$\mu \cupro \mu = {\tiny \yng(2,2)}$.

\item The set 
$\bar\Omega(\lambda, F_\bullet) \cap \bar\Omega(\mu, G_\bullet)$
is empty, thus $\lambda \cupro \mu = 0$.
\end{itemize}
However, if we work in $A^*(4,8)$  
then it can be shown that:
$$
{\tiny \yng(2)}   \cupro {\tiny \yng(2)}   = {\tiny \yng(2,2)} + {\tiny \yng(4)} + {\tiny \yng(3,1)}, \quad
{\tiny \yng(1,1)} \cupro {\tiny \yng(1,1)} = {\tiny \yng(2,2)} + {\tiny \yng(2,1,1)} + {\tiny \yng(1,1,1,1)}, \quad
{\tiny \yng(2)}   \cupro {\tiny \yng(1,1)} = {\tiny \yng(3,1)} + {\tiny \yng(2,1,1)}.
$$
If we drop the terms that do not fit in a $2 \times 2$ rectangle, we reobtain the results for $G_2(\C^4)$. 
\end{example}

The general computation of the product $\lambda \cupro \mu$ is 
not simple and can be done in various ways. 
For our purposes, however, it will be sufficient to know when 
the product is zero or not.
The answer is provided by the following simple \lcnamecref{l.overlap}: 

\begin{lemma}[\cite{Fulton}, p.~148--149]\label{l.overlap}
Let $\lambda$ and $\mu$ be Young diagrams in the $k \times (n-k)$ rectangle.
The following two conditions are equivalent:
\begin{enumerate}

\item\label{i.nonzero} 
$\lambda \cupro \mu \neq 0$.

\item\label{i.nonoverlap} 
If one draws inside the $k \times (n-k)$ rectangle
the Young diagrams of $\lambda$ and $\mu$,
being the later rotated by $180^{\circ}$ and put in the southeast corner,
then the two figures do not overlap
(see \cref{f.nooverlap}). 
Equivalently, $\lambda_i + \mu_{k+1-i} \le n-k$ for every $i=1, \ldots, n$.
\end{enumerate}
\end{lemma}

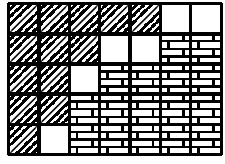
\begin{figure}[hbt]  
\setlength{\unitlength}{.4cm}
\begin{picture}(7,5)

\thicklines
\put(0,0){\grid(7,5)(1,1)}

\thinlines
\put(0,0){\diagonalpattern}
\put(0,1){\diagonalpattern}
\put(1,1){\diagonalpattern}
\put(0,2){\diagonalpattern}
\put(1,2){\diagonalpattern}
\put(0,3){\diagonalpattern}
\put(1,3){\diagonalpattern}
\put(2,3){\diagonalpattern}
\put(0,4){\diagonalpattern}
\put(1,4){\diagonalpattern}
\put(2,4){\diagonalpattern}
\put(3,4){\diagonalpattern}
\put(4,4){\diagonalpattern}

\thinlines
\put(2,0){\brickpattern}
\put(3,0){\brickpattern}
\put(4,0){\brickpattern}
\put(5,0){\brickpattern}
\put(6,0){\brickpattern}
\put(2,1){\brickpattern}
\put(3,1){\brickpattern}
\put(4,1){\brickpattern}
\put(5,1){\brickpattern}
\put(6,1){\brickpattern}
\put(3,2){\brickpattern}
\put(4,2){\brickpattern}
\put(5,2){\brickpattern}
\put(6,2){\brickpattern}
\put(5,3){\brickpattern}
\put(6,3){\brickpattern}

\end{picture}
\caption{{\footnotesize 
Consider $k=5$, $n=12$, $\lambda=(5,3,2,2,1)$, and $\mu = (5,5,4,2,0)$.
The picture shows that the non-overlap condition~(\ref{i.nonoverlap}) from \cref{l.overlap}
is satisfied, and in particular $\lambda \cupro \mu \neq 0$. (This example is reproduced from 
\cite[p.~150]{Fulton}.)}}
\label{f.nooverlap}
\end{figure}


\subsection{Intersection of subvarieties of the Grassmannian}\label{ss.intersection}

Next we explain how the Schubert calculus sketched above can be used to obtain information
about intersection of general subvarieties of the Grassmannian,
by means of cohomology and Poincar\'{e} duality.
See \cite[Appendix~B]{Fulton} and \cite{Hutchings} for further details.

\medskip

Any topological space $X$ has singular homology groups $H_i X$ and cohomology groups $H^i X$ (here taken always with integer coefficients). With the cup product $H^i X \times H^j X \to H^{i+j} X$,
the cohomology $H^* X = \bigoplus H^i X$ has a ring structure.

If $X$ is a real compact oriented manifold of dimension $d$ then 
the homology group $H_d X$ is canonically isomorphic to $\Z$, with a generator $[X]$
called the \emph{fundamental class} of $X.$
In addition, there is \emph{Poincar\'{e} duality isomorphism}
$H^i X \to H_{d-i} X$, which is given by 
$\alpha \mapsto \alpha \smallfrown [X]$
(taking the cap product with the fundamental class).
Let us denote by $\omega \mapsto \omega^*$ the inverse isomorphism.

Next suppose $Y$ and $Z$ are compact oriented submanifolds of $X$, of codimensions $i$ and $j$ respectively.
Also suppose that $Y$ and $Z$ have transverse intersection
$Y \cap Z$, which therefore is either empty or a compact submanifold of codimension $i + j$,
which is oriented in a canonical way.
The images of the fundamental classes of $Y$, $Z$, and $Y\cap Z$ under the inclusions into $X$
define homology classes that we denote (with a slight abuse of notation) by
$[Y] \in H_{d-i} X$, $[Z] \in H_{d-j} X$, $[Y\cap Z] \in H_{d-i-j} X$.
Then their Poincar\'e duals 
$[Y]^*\in H^i X$, $[Z]^* \in H^j X$, and $[Y \cap Z]^* \in H^{i+j} X$
are related by: 
$$
[Y]^* \cupro [Z]^* = [Y \cap Z]^* \, .
$$
That is, \emph{cup product is Poincar\'{e} dual to intersection.}

Now consider the case where $X$ is a projective nonsingular (i.e., smooth) complex variety,
and $Y$ and $Z$ are irreducible  subvarieties of $X$.
Obviously, the fundamental class $[X]$ makes sense, because $X$ is a compact manifold
with a canonical orientation induced from the complex structure.
A deeper fact (see \cite[Appendix~B]{Fulton})
is that fundamental classes $[Y]$ and $[Z]$ can also be canonically associated to 
the (possibly singular) subvarieties $Y$ and $Z$,
and the Poincar\'{e} duality between cup product and intersection 
works in this situation.
More precisely,
suppose that $Y$ and $Z$ are transverse in the algebraic sense:
$Y \cap Z$ is a union of subvarieties $W_1$, \dots, $W_\ell$
whose codimensions are the sum of the codimensions of $Y$ and $Z$,
and for each $i=1,\dots,\ell$, the tangent spaces
$T_w Y$ and $T_w Z$ are transverse 
for all $w$ in a Zariski-open subset of $W_i$.
Then each $W_i$ has its canonical fundamental class, 
and the following duality formula holds:
$$
[Y]^* \cupro [Z]^* = [W_1]^* + \cdots + [W_\ell]^* \, .
$$

\medskip

In our application of this machinery, $X$ will be the Grassmannian $G_k(\C^n)$.
In this case:
\begin{itemize}
\item The fundamental classes of the Schubert varieties $[\bar\Omega(\lambda, F_\bullet)]$
do not depend on the flag $F_\bullet$.
\item Let $\sigma_\lambda$ denote the Poincar\'e dual of $[\bar\Omega(\lambda, F_\bullet)]$.
Then $H^{2r} G_k(\C^n)$ is a free abelian group and the elements $\sigma_\lambda$
with $|\lambda| = r$ form a set of generators.
(The cohomology groups of odd codimension are zero.)
\item The cup product on cohomology agrees with the ``cup'' product of Young diagrams 
explained in the previous section.
\end{itemize}

\subsection{End of the proof}\label{ss.end}

We are now able to prove \cref{t.schubert2}. 

\begin{proof}[Proof of \cref{t.schubert2}]
Let $1 \le e \le k < n$.
Let $E \subset \C^n$ be a subspace of dimension $e$,
and consider the set $S_k(E)$ defined by \eqref{e.special schubert}.
Recall from \cref{ex.special schubert} that 
this is the Schubert variety for the Young diagram $\lambda$ given by \eqref{e.special young}.

Now consider a (nonempty) subvariety $Y \subset G_k(\C^n)$
that is disjoint from $S_k(E)$.
We want to give a lower bound for the codimension $c$ of $Y$.
We can of course assume that $Y$ is irreducible.

Let $[Y]^*$ be the dual of fundamental class of $Y$.
This is a nonzero element of $H^{2c} G_k(\C^n)$.
It can be expressed as  
$\sum n_i \sigma_{\mu_i}$,
where $\mu_i$ are Young diagrams with area $|\mu_i|=c$,
and $n_i$ are nonzero integers.
In fact we have $n_i>0$, because of the canonical
orientations induced by complex structure. 

Since the intersection between $S_k(E)$ and $Y$ is empty (and in particular transverse),
Poincar\'{e} duality gives $[S_k(E)]^* \cupro [Y]^* = 0$.
Therefore we have $\sigma_\lambda \cupro \sigma_{\mu_i} = 0$ for each $i$.

By \cref{l.overlap},
if we draw the Young diagram of $\mu_i$
rotated by $180^{\circ}$ and put in the southeast corner
of the $k \times (n-k)$ rectangle,
then it overlaps the Young diagram $\lambda$ pictured in \eqref{e.special young}.
This is only possible if $c \ge k-e+1$;
indeed the Young diagram $\mu$ with least area such that $\lambda \cupro \mu \neq 0$ is
$$
\mu = \big( \underbrace{1,\dots,1}_{k-e+1 \text{ times}}, \underbrace{0,\dots,0}_{e-1 \text{ times}} \big), 
$$
for which the overlapping picture becomes:
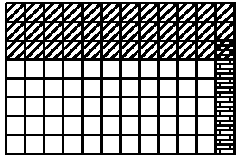
\setlength{\unitlength}{.25cm}
\begin{center}
\begin{picture}(12,8)
\thinlines
\put(0,0){\grid(12,8)(1,1)}
\multiput(0,5)(0,1){3}{\multiput(0,0)(1,0){12}{\lightdiagonalpattern}}
\multiput(11,0)(0,1){6}{\brickpattern}
\thicklines
\put(11,5){\grid(1,1)(1,1)}
\end{picture}
\end{center}
This concludes the proof of  \cref{t.schubert2}.
\end{proof}

As explained in \S~\ref{sss.reduction}, \cref{t.schubert} follows.

%
%
%
%
%
%
%
%

\section{Stratifications and transversality}\label{a.strat_trans}

\subsection{Stratifications}
This appendix contains fundamental for the understanding of \cref{s.main proof}. We recall a few notions about stratifications and transversality, and prove \cref{p.stratifiedtransversality}.
We refer the reader to \cite{GWPL,Mather_71} for more details and proofs.

\medskip

Let $X$ be a smooth (i.e., $C^\infty$) manifold. 
A \emph{smooth stratification} of a closed subset $\Si\subset X$ is a filtration by closed subsets 
$$
\Si = \Si_n \supset \Si_{n-1} \supset \cdots \supset \Si_0
$$ 
such that for each $i$, the set $\Gamma_i = \Si_i \setminus \Si_{i-1}$ (where $\Sigma_{-1}:=\emptyset$) is a smooth submanifold of $X$ without boundary and the dimension of $\Gamma_i$ decreases strictly with increasing $i$. 
Each connected component of $\Gamma_i$ is called a \emph{stratum}. 
The \emph{codimension} in $X$ of a stratification is the codimension of the stratum of largest dimension. 
A stratification of  a set $\Si$ is not unique, but this codimension in $X$ does not depend on the choice of the stratification.

Actually, apart for discrete subsets $\Si\subset X$, if there is one smooth stratification, 
then there are infinitely many others. 
However, the subsets we deal with 
are endowed with certain \emph{canonical} stratifications:


\begin{otherthm}[Existence of canonical stratifications]\label{t.canonical_stratif}
Any algebraic set $\Si\subset \C^N$ admits a canonical smooth stratification
whose strata are complex submanifolds of $\C^N$.
Any closed semialgebraic set $\Si\subset \R^N$ admits a canonical smooth stratification
whose strata are semialgebraic submanifolds of $\R^N$.
\end{otherthm}

In the case of an irreducible algebraic set $\Si\subset \C^n$, 
the canonical stratification can be obtained as follows:
The connected components of the set of regular (i.e., non-singular) points form the higher-dimensional 
strata; then one decomposes the set of singular points of $\Si$ into irreducible components and proceeds by 
induction.

In any case, those canonical stratifications are uniquely characterized by a certain minimality property.
In particular, the canonical stratifications are equivariant under polynomial automorphisms of 
the ambient space.

Another important property of the canonical stratifications 
is that they satisfy the so-called \emph{Whitney conditions $(a)$ and $(b)$}: 
\medskip

For any  sequence of points $x_n$ in a stratum $\Gamma$ of dimension $i$ converging to a point $y$ in a stratum $\Delta$ of dimension $<i$, if the sequence of tangent spaces $T_{x_n}\Gamma$ converges to an $i$-space $E\subset T_yX$, then we have
\begin{itemize}
\item[(a)] $E$ contains $T_y\Delta$,

\item[(b)] in a local chart, if a sequence $y_n\in \Delta$ converges to $y$ and if the lines $x_ny_n$ converge to a line $L\subset T_y\Delta$,  then $L\subset E$.
\end{itemize}

\medskip

A smooth stratification that satisfies the Whitney conditions is called a \emph{Whitney stratification}. Let us write down some properties.


\begin{prop}[Basic properties of Whitney stratifications]\label{p.whitney_properties} 
Let $X$, $Y$ be smooth manifolds.
Let 
\begin{equation}\label{e.filtration}
\Sigma_n \supset \cdots \supset \Sigma_0
\end{equation}
be a filtration of a set $\Sigma \subset X$.
Then:
\begin{enumerate}
\item \label{i.whitney_local}
Being a Whitney stratification is a local property of a filtration:
So if \eqref{e.filtration} is a Whitney stratification then
$\Sigma_n \cap U \supset \cdots \supset \Sigma_0 \cap U$ 
is a Whitney stratification, 
and conversely if each point in $\Sigma$ has an open neighborhood $U \subset X$ such that
$\Sigma_n \cap U \supset \cdots \supset \Sigma_0 \cap U$ is a Whitney stratification
then \eqref{e.filtration} is a Whitney stratification.
\item \label{i.whitney_product}
If  \eqref{e.filtration} is a Whitney stratification of codimension $m$ in $X$,
then $\Sigma_n \times Y \supset \cdots \supset \Sigma_0 \times Y$  
is a Whitney stratification of codimension $m$ in $X \times Y$.
\item \label{i.whitney_invariance}
If  \eqref{e.filtration} is a Whitney stratification 
and $f \colon X \to Y$ is a smooth diffeomorphism 
then $f(\Sigma_n) \supset \cdots \supset f(\Sigma_0)$ 
is a Whitney stratification in $Y$. 
\end{enumerate}
\end{prop}

Let us now discuss how stratifications behave with respect to transversality. 
Let $f \colon X \to Y$ be a $C^1$ map.
Let $\Si=\Si_d \supset\cdots \supset\Si_0$ be a stratification of a closed subset $\Si$ of $Y$. 
One says that $f$ is \emph{transverse} to that stratification 
(in symbols, $f\pitchfork \Si$)
if it is transverse to each of its strata.
Transversality to a general stratification is not an open condition.
However, we obtain openness if the stratification is Whitney:


\begin{prop}[Transversality is open]\label{p.transversality open}
Let $X$, $Y$ be $C^\infty$ manifolds without boundary. Let
$\Si=\Si_d \supset\cdots \supset\Si_0$ be a Whitney stratification of a closed subset of $Y$. 
Then the set $\cO = \{f \in C^1(X,Y) ; \;  f\pitchfork \Si\}$ is open in $C^1(X,Y)$ (with respect to the 
strong topology). 
\end{prop}

Actually, only Whitney condition $(a)$ is necessary here
(use the (1)$\Rightarrow$(3) implication of Trotman's theorem \cite{Trotman_79}).

\subsection{Jets and jet transversality}

We recall the basic notions on jets 
and state the transversality theorems we will need;
see~\cite{Hirsch} for details.

Let $X$, $Y$ be smooth manifolds without boundary. 
If $1 \le r < \infty$, an \emph{$r$-jet from $X$ to $Y$} is an equivalence class of pairs $(x,f)$, where $x\in X$, $f$ is a $C^r$ map from a neighborhood of $x$ to $Y$, and where $(x,f)$ is equivalent to $(x',f')$ if $x=x'$ and $f$ and $f'$ have same derivatives at $x$ up to order $r$. 
We denote by $J^r(X,Y)$ the space of $r$-jets from $X$ to $Y$. 
It is a smooth manifold.

For all $1 \leq s \leq \infty$, we denote by $C^s(X,Y)$ the space of $C^s$-maps from $X$ to
$Y$, endowed with the 
strong topology.

Given $1\leq r < s \leq \infty$ and a map $g\in C^s(X,Y)$, 
the \emph{$r$-jet extension} is the map $j^r g \colon X \to J^r(X,Y)$ that sends $x$ 
to the equivalence class $j^r g(x)$ of $(x,g)$. 
Then the mapping 
$$
j^r\colon C^s(X,Y) \to C^{s-r}\left( X, J^r(X,Y) \right)
$$
is continuous.


\begin{otherthm}[Jet transversality]\label{t.jtransversality}
Let $1\leq r < s \leq \infty$.
Let $X$ and $Y$ be $C^\infty$ manifolds without boundary.
Let $W\subset J^r(X,Y)$ be a $C^\infty$ submanifold without boundary. 
Then the $C^s$-maps
$g\colon X\to Y$ for which the $r$-jet 
extension $j^r g$ is transverse to $W$
form a residual subset of $C^s(X,Y)$.
\end{otherthm}

We finally prove the proposition stated in \S~\ref{s.main proof}:

\begin{proof}[Proof of Proposition~\ref{p.stratifiedtransversality}]
By Proposition~\ref{p.transversality open}, the set $\{F\colon X\to j^1(X,Y) ;\; F\pitchfork\Si\}$ is open in $C^1\left(X,J^1(X,Y)\right)$. Hence the set $\cO := \{f\colon X\to Y ;\;j^1f\pitchfork\Si\}$ is open in $C^2(X,Y)$. 
	
Fix $r \ge 2$.
Given a Whitney stratification $\Sigma_n \supset \cdots \supset \Sigma_0$ of $\Sigma$,
let $Z_i = \Sigma_i \setminus \Sigma_{i-1}$ be the corresponding decomposition into smooth submanifolds. By the jet transversality theorem (\cref{t.jtransversality}),
each set $\mathcal{R}_i = \{f\in C^r(X,Y) ; \; j^1 f \pitchfork Z_i\}$ is residual. 
Thus $\cO \cap C^r(X,Y) = \bigcap_i\mathcal{R}_i$ is $C^r$-dense. 
This concludes the proof.
\end{proof}

\section{Proof of the result in the holomorphic setting}\label{a.complex}





\begin{proof}[Proof of~\cref{t.main_C}]
Let $\cU\subset \C^m$ be an open subset. We may identify the set of $1$-jets from $\cU$ to $\GL(d,\C)$ with 
$$
\cU\times \GL(d,\C) \times  \gl(d,\C)^m.
$$
As we did in \cref{s.main proof}, 
and using \cref{t.cod_data_C} instead of \cref{t.cod_data_R}, we obtain that the set of poor $1$-jets from $\cU$ to $\GL(d,\C)$ is the algebraic subset $\cU\times\cP_m^{(\C)}$ of the space of $1$-jets. Hence it admits a stratification
$$
\cU\times\cP_m^{(\C)}=\cU\times \Si_n\supset \cdots \supset\cU\times \Si_0.
$$
Write $\cU\times\cP_m^{(\C)}$ as the disjoint union $\bigsqcup_{0\leq i\leq n} X_i$ where each $X_i$ is a smooth submanifold  of dimension $i$ in the jet space $J^1\left(\cU,\GL(d,\C)\right)$, and $X_n$ has codimension $m$.

Fix now a map $A\in \mathcal{H}(\cU,\GL(d,\C))$. For all $v=(a,b_1,\ldots b_m)\in \C^{m+1}$ and $u=(u_1, \dots ,u_m)\in \C^m$, write 
 $$P_{v}(u)=a+\sum_{i=1}^m b_ku_k.$$
For all $\mathtt{v}= (v_{i,j})_{1\leq i,j\leq d}\in\left(\C^{m+1}\right)^{d^2}$, write $P_\mathtt{v}=\left[P_{v_{i,j}}\right]_{1\leq i,j\leq d}$ and define the map $\Phi_\mathtt{v}= A+P_\mathtt{v}$.
One can write the $1$-jet extension $j^1 A$ at the point $u\in \cU$ as 
$$j^1 A(u)= \left[u,A(u),B_1, \dots ,B_m\right]\in \cU \times \GL(d,\C) \times \left[\mathrm{Mat}_{d \times d}(\C)\right]^m.$$
The same way, if we put $v_{i,j}=(a_{i,j},b_{1,i,j},\ldots,b_{m,i,j})$, we have
$$j^1P_\mathtt{v}(u)=\left[u,P_\mathtt{v}(u),(b_{1,i,j})_{1\leq i,j\leq d},\ldots,(b_{m,i,j})_{1\leq i,j\leq d}\right].$$
Define the map $F\colon \mathtt{v} \mapsto F_\mathtt{v}=j^1\Phi_\mathtt{v}$. The evaluation map of $F$ is:
$$F^\mathrm{ev}\colon \begin{cases} \left(\C^{m+1}\right)^{d^2} \times \cU &\to \; \cU\times \mathrm{Mat}_{d \times d}(\C) \times \left[\mathrm{Mat}_{d \times d}(\C)\right]^m\\
                                                          (\mathtt{v},u) &\mapsto \; F_\mathtt{v}(u)
                                     \end{cases}.$$ 
Hence,
\begin{align*}
F^\mathrm{ev}(\mathtt{v},u)&=j^1(A+P_\mathtt{v})\\
&=\left[u,\left(A+P_\mathtt{v}\right)(u),\left(b_{1,i,j}\right)_{1\leq i,j\leq d},\ldots,\left(b_{m,i,j}\right)_{1\leq i,j\leq d}\right]
\end{align*}

\begin{claim} For all $u$, the map $F^\mathrm{ev}$ restricts to a submersion from the $(\cdot,u)$-fiber to the $\left[u,\cdot\right]$-fiber. 
\end{claim}

\begin{proof}
We want to prove that $$\mathtt{v}\mapsto \left[(A+P_\mathtt{v})(u),\left(b_{1,i,j}\right)_{1\leq i,j\leq d},\ldots,\left(b_{m,i,j}\right)_{1\leq i,j\leq d}\right]$$
is a submersion, or equivalently that
$$\mathtt{v}\mapsto \left[P_\mathtt{v}(u),\left(b_{1,i,j}\right)_{1\leq i,j\leq d},\ldots,\left(b_{m,i,j}\right)_{1\leq i,j\leq d}\right]$$
is a submersion. Noting that $\mathtt{v}=(a_{i,j},b_{k,i,j})_{1\leq i,j\leq d \atop 1\leq k\leq m}$, this comes easily from the fact that $(a_{i,j})\mapsto P_\mathtt{v}(u)$ is a submersion, for any fixed set of coefficients $(b_{k,i,j})_{1\leq i,j\leq d \atop 1\leq k\leq m}$.
\end{proof}

That claim immediately implies that $F^\mathrm{ev}$ is a submersion. In particular it is transverse to each $X_i$. By the parametric transversality theorem (see \cite[p.~79]{Hirsch}), 
there is a residual subset of parameters $\mathtt{v}$ in $\left(\C^{m+1}\right)^{d^2}$ such that $F_\mathtt{v}=j^1\Phi_\mathtt{v}$ is transverse to $X_i$, for all $i$.

When $\mathtt{v}$ goes to $0$, $\Phi_\mathtt{v}$ tends to $A$ in $\mathcal{H}\left(\cU,\GL(d,\C)\right)$. 
This shows the denseness in $\mathcal{H}\left(\cU,\GL(d,\C)\right)$ of the maps $\hat{A}$ such that $j^1\hat{A}$ is transverse to $X_i$, for all $i$.
Take such a map $\hat{A}$: for all $i$, the image of $j^1\hat{A}$ does not intersect $X_0\sqcup  \dots  \sqcup X_{n-1}$ and intersects $X_n$ (which has codimension $m$) only in a discrete subset.

Fix $K'\subset \cU$ a compact set that contains $K$ in its interior. The image $j^1\hat{A}$ restricted to $K'$ can only intersect $X_n$ in a finite set $\Gamma$: indeed, any accumulation point of that intersection set would have to be in $X_0\sqcup  \dots  \sqcup X_{n-1}$, since $X_0\sqcup\ldots\sqcup X_n$ is closed, and this would contradict the fact that $j^1\hat{A}$ does not intersect $X_0\sqcup  \dots  \sqcup X_{n-1}$.

By the choice of our topology, a small perturbation $\tilde{A}$ of $\hat{A}$ is $C^0$ close to $\hat{A}$ by restriction to $K'$. By Cauchy's formula, the map $\tilde{A}$ is $C^2$ close to $\hat{A}$ over the set $K$. Hence, the (compact) image of $j^1\tilde{A}$ restricted to $K$ is still far from $X_0\sqcup  \dots  \sqcup X_{n-1}$, and intersects $X_n$ transversally in some $\epsilon$-neighborhood of $\Gamma$ inside $X_n$. Thus it also has to intersect $X_n$ only on a finite set. 

So we have found an open and dense subset of holomorphic maps whose $1$-jets above $K$ intersect the set of $N$-poor jets only on a finite number of points. As a consequence, for such maps,
there are only finitely many constant singular inputs  in $K^N$ for the system~\ref{e.proj semilin CS}. 
This concludes the proof of Theorem~\ref{t.main_C}.
\end{proof}

\section{Singular constant inputs of generic type} \label{a.generic singular}

In this appendix we prove \cref{t.addendum}
and the other assertions made at the end of \S~\ref{ss.main_statements}.
We also discuss other control-theoretic properties of generic semilinear systems
that are related to universal regularity.

\subsection{The poor data of generic type}\label{ss.additional}

Recall from \S~\ref{ss.unconstrained} the definition of 
an unconstrained matrix.
Let $(e_1, \dots, e_d)$ denote the canonical basis of $\C^d$.

\begin{lemma}\label{l.good_poor}
Suppose that the datum $\bA = (A, B_1, \dots, B_m) \in \GL(d,\C) \times \gl(d,\C)^m$
has the following properties:
\begin{enumerate}
\item\label{i.good_hyp_1}
$A$ is an unconstrained diagonal matrix;
\item\label{i.good_hyp_2}
there are indices $i_0$, $j_0 \in \{1,\dots,d\}$ with $i_0 \neq j_0$ 
such that for each $k \in \{1,\dots,m\}$, the $(i_0,j_0)$ entry of the matrix $B_k$ vanishes;
\item\label{i.good_hyp_3}
the off-diagonal vanishing entry position $(i_0,j_0)$ above is unique.
\end{enumerate}
Then:
\begin{enumerate}
\item\label{i.good_conclusion_1}
There is a single direction $[v] \in \CP^{d-1}$ such that $\Lambda(\bA) \cdot v \neq \C^d$,
namely~$[e_{j_0}]$.
\item\label{i.good_conclusion_2}
The space $\Lambda(\bA) \cdot e_{j_0}$ has codimension $1$; 
in fact, it equals $\spa \{e_i ; \; i \neq i_0\}$.
\end{enumerate}
\end{lemma}

If the datum $\bA$ satisfies the assumptions of the lemma
then it is conspicuously poor (see \cref{ss.cod_data_easy_half})
and thus the constant input $(0,\dots,0)$ of length $d^2$
for the associated bilinear control system on $\CP^{d-1}$
is not universally regular.
However, the conclusions of the lemma say that this universal regularity fails in 
the weakest possible way: there is exactly one non-regular state,
which can be moved in all directions but one.
We will show in \cref{l.bad_set} below that the generic poor data
satisfy the hypotheses of \cref{l.good_poor} after a change of basis.

\begin{proof}[Proof of \cref{l.good_poor}]
By arguments as in the proof of \cref{l.easy_fiber},
we see that 
$$
\Lambda(\bA) \supset
\big\{ (y_{ij}) \in \gl(d,\C) ; \; y_{11} = \cdots = y_{dd}, \ y_{i_0 j_0} = 0 \big\}.
$$	
The conclusions follow easily.
\end{proof}


\medskip

Recall from \cref{ss.poor_set} that a set $\cZ \subset [\Mat_{d\times d}(\K)]^{1+m}$
is called \emph{saturated} if
$(A, B_1, \dots, B_m) \in \cZ$ implies that:
\begin{itemize}
\item 
for all $P \in \GL(d,\K)$ we have $(P^{-1}AP, P^{-1}B_1 P, \dots, P^{-1}B_m P) \in \cZ$;
\item 
for all $Q = (q_{ij}) \in \GL(m,\K)$, letting $B'_i = \sum_j q_{ij} B_j$,
we have $(A, B_1', \dots, B_m') \in \cZ$.
\end{itemize}

\begin{rem}\label{r.saturation_properties}
\begin{enumerate}
\item\label{i.saturation_properties_1}
A subset $[\Mat_{d\times d}(\K)]^{1+m}$ is saturated if and only if it is invariant under a certain action of the group $\GL(d,\K)\times\GL(m,\K)$.
\item\label{i.saturation_properties_2}
The real part of a complex saturated set is saturated (in the real sense).
\end{enumerate}
\end{rem}

\begin{lemma}\label{l.bad_set}
There exists a saturated algebraically closed set 
$\cS_m^{(\C)} \subset \GL(d,\C)\times [\Mat_{d\times d}(\C)]^m$        
of codimension at least $m+1$
such that for all $(A, B_1, \dots, B_m) \in \cP_m^{(\C)} \setminus \cS_m^{(\C)}$, the following properties hold:
\begin{enumerate}
\item\label{i.good_1}
$A$ is unconstrained;
\item\label{i.good_2}
if $P\in \GL(d,\C)$ is such that $P^{-1} A P$ is a diagonal matrix then
there are indices $i_0$, $j_0 \in \{1,\dots,d\}$ with $i_0 \neq j_0$ 
such that for each $k \in \{1,\dots,m\}$, the $(i_0,j_0)$ entry of the matrix $P^{-1} B_k P$ vanishes;
\item\label{i.good_3}
for each choice of $P$ above,
the off-diagonal vanishing entry position $(i_0,j_0)$ is unique.
\end{enumerate}
\end{lemma}


\medskip

In order to prove the \lcnamecref{l.bad_set},
we begin by checking algebraicity of the constraints:

\begin{lemma}\label{l.alg_constraint}
The set $K \subset \GL(d,\C)$ of constrained matrices
is an algebraically closed subset of codimension $1$.
\end{lemma}

\begin{proof}
Multiply all constraints, obtaining a polynomial in the variables $\lambda_1$, \dots, $\lambda_d$.
This polynomial is symmetric, 
and therefore (see e.g.\ \cite[Thrm.~IV.6.1]{Lang})
can be written as a polynomial function of the elementary symmetric polynomials in the variables 
$\lambda_1$, \dots, $\lambda_d$.
Now substitute each elementary symmetric polynomial in this expression by the corresponding
coefficient of the characteristic polynomial of the matrix $A$.
This gives a polynomial function on the entries of the matrix $A$
that vanishes if and only if $A$ is constrained.
It is obvious that the corresponding algebraic set $K$ has codimension $1$.
\end{proof}

Now we check algebraicity of double vanishing:

\begin{lemma}\label{l.double_zero}
There exists a saturated algebraically closed subset 
$\cD$ of $\GL(d,\C)\times [\Mat_{d\times d}(\C)]^m$ 
such that
if $(A, B_1, \dots, B_m) \in \cD$ and $A$ has simple spectrum then 
property \ref{i.good_2} from \cref{l.bad_set} is satisfied, 
but property \ref{i.good_3} is not.
\end{lemma}

\begin{proof}
First, consider the subset $X \subset [\Mat_{d\times d}(\C)]^{1+m} \times (\CP^{d-1})^2$
formed by tuples $(A,B_1,\dots,B_m,[v],[w])$ such that
$$
[Av] = [v], \quad [A^* w] = [w], \quad w^* v = 0, \quad w^* B_k v = 0 \text{ for each $k=1,\dots,m$,}
$$
where $v$ and $w$ are regarded as column-vectors and the star denotes transposition.
The set $X$ is obviously algebraic; thus, by \cref{p.projection}, so is its projection
$Y$ on $[\Mat_{d\times d}(\C)]^{1+m}$.

Let $A$ be a matrix with simple spectrum.
Then $(A, B_1, \dots, B_m)$
belongs to $Y$ if and only if property \ref{i.good_2} from \cref{l.bad_set} is satisfied.
In particular, the fiber of $Y$ over $A$ is 
a union of affine subspaces of $[\Mat_{d\times d}(\C)]^m$.
Intersections of those affine spaces correspond to points where the uniqueness 
property \ref{i.good_3} is not satisfied.
These points of intersection are singular points of $Y$.
Conversely, it is clear that the variety $Y$ is smooth at the points on the fiber over $A$
where property \ref{i.good_3} is satisfied.

So let $Z$ be the (algebraically closed) set of singular points of $Y$.
It is straightforward to see that the set $Y$ is saturated.
Recalling \cref{r.saturation_properties} (part~\ref{i.saturation_properties_1})
and the fact that a group acting on a variety preserves singular points,
we see that the set $Z$ is saturated as well.

We define $\cD$ as the set $Z$ minus the tuples $(A, B_1, \dots, B_m)$ with $\det A = 0$.
Then $\cD$ has all the required properties.
\end{proof}

Now we combine the facts above with \cref{scholium}
to prove \cref{l.bad_set}:

\begin{proof}[Proof of \cref{l.bad_set}]
For simplicity of writing we will omit the $m$ subscripts and the $(\C)$ superscripts.

Let $\pi: \cP \to \GL(d,\C)$ be the projection on the first matrix.
Define
$$
\cS = \pi^{-1}(K) \cup (\cD \cap \cP),
$$
where $K$ and $\cD$ come respectively from \cref{l.alg_constraint,l.double_zero}.
Then $\cS$ is a saturated algebraically closed subset of $\cP$. 
If $\bA = (A, B_1, \dots, B_m) \in \cP \setminus \cS$ then:
\begin{itemize}
\item $A \not\in K$, which is property~\ref{i.good_1};
\item since $\bA \in \cP$, it follows from \cref{l.easy_fiber} that $\bA$ is conspicuously poor,
and so property~\ref{i.good_2} holds;
\item since $\bA \not\in \cD$, property~\ref{i.good_3} also holds.
\end{itemize}

To complete the proof of the \lcnamecref{l.bad_set}, we need to show that $\codim \cS \ge m+1$.
We will use the following inclusion:
\begin{equation}\label{e.union}
\cS \subset \cF \cup \underbrace{\big( \pi^{-1}(K) \setminus \cF \big)}_{\cF'} \cup 
\underbrace{\big( (\cD \cap \cP) \setminus \pi^{-1}(K)\big)}_{\cF''} \, .
\end{equation}
where $\cF$ comes from \cref{scholium}.
Recall that $\cF$ equals $\pi^{-1}(C_{m-1})$, where $C_j$ is given by \eqref{e.C_j},
and it has codimension at least $m+1$.

We apply \cref{l.pret_a_porter,r.homogeneous}
to the set $\cF' \subset Y' \times [\gl(d,\C)]^m$, 
where $Y' = \GL(d,\C) \setminus C_{m-1}$.
Since $K$ has codimension at least $1$ in $Y'$,
and the fibers of $\cF'$ all have codimension at least $m$, 
we conclude that that $\codim \cF' \ge m+1$.

Next, we want to apply \cref{l.pret_a_porter,r.homogeneous}
to the set $\cF'' \subset Y'' \times [\gl(d,\C)]^m$, 
where $Y'' = \GL(d,\C) \setminus K$.
For each $A \in Y''$, it follows from \cref{l.double_zero}
that the fiber of $\cF''$ over $A$ (which is the same as the fiber of $\cD$ over $A$)
has codimension $2m$ in $[\gl(d,\C)]^m$,
corresponding to the $2m$ different matrix entries that must vanish.
We conclude that $\codim \cF'' \ge 2m$.

We have seen that each of the three sets on the right-hand side of \eqref{e.union}
has codimension at least $m+1$.
So the same is true for $\cS$, as we wanted to prove.
\end{proof}

\subsection{Proof of the addendum to the Main \cref{t.main}}

\begin{proof}[Proof of \cref{t.addendum}]
Consider the set $\cS_m^{(\C)}$ given by \cref{l.bad_set},
and let $\cS_m^{(\R)}$ be its real part.
This is an algebraically closed saturated subset of $\GL(d,\R) \times [\gl(d,\R)]^m$
which, by \cref{p.real complex dimension}, has codimension at least $m+1$.

Consider the set $\tilde \Gamma$ of $1$-jets $\bJ \in J^1(\cU,\GL(d,\C))$ 
that have a local expression $(u, A(u),B_1,\dots,B_m)$
with $(A(u),B_1,\dots,B_m) \in \cS_m^{(\R)}$.
This does not depend on the choice of the local coordinates, because $\cS_m^{(\R)}$ is saturated.
By the same arguments as in the proof of \cref{t.main}, 
the set $\tilde \Gamma$ admits a Whitney stratification.
Its codimension is at least $m+1$.
Applying Proposition~\ref{p.stratifiedtransversality}, we obtain a $C^2$-open
$C^\infty$-dense set $\tilde \cO \subset C^2(\cU,\GL(d,\C))$ formed by maps $A$ 
that are transverse to the stratification.

Let $\cO$ be the set provided by \cref{t.main}.
and consider a map $A \in \cO \cap \tilde \cO$.
Then whenever a jet $j^1 A (u)$ is poor,
it does not belong to $\tilde \Gamma$.
Recalling \cref{l.bad_set}, we see that the local expression of 
$j^1 A (u)$ satisfies (after a change of basis) the hypotheses of \cref{l.good_poor}.
Therefore parts \ref{i.addendum_1} and \ref{i.addendum_2} of the theorem
follow respectively from conclusions \ref{i.good_conclusion_1} and \ref{i.good_conclusion_2}
of the \lcnamecref{l.good_poor}.
\end{proof}

\begin{rem}\label{r.addendum_of_addendum}
The proof of \cref{t.addendum} also gives more information
about the $1$-jets that appear generically for singular constant inputs $(u,\ldots,u)$:
any associated matrix datum is conspicuously poor 
and the matrix $A(u)$ is unconstrained. 
\end{rem}

\begin{rem}
Properties \ref{i.addendum_1} and \ref{i.addendum_2} in \cref{t.addendum}
are in fact dual to each other.
If $\bA$ is the datum representing the $1$-jet of $A$ at $u$,
and $\Lambda = \Lambda(\bA)$,
then property \ref{i.addendum_1} means that there is an unique direction 
$[v] \in \RP^{d-1}$ such that $\Lambda \cdot v \neq \C^d$.
Then property \ref{i.addendum_2} means that there is an unique direction 
$[w] \in \RP^{d-1}$ such that $\Lambda^* \cdot w \neq \C^d$,
where $\Lambda$ is the set of the transposes of the matrices in $\Lambda$.
This fact can be proved easily using the dual characterization of \cref{l.duality}. 
\end{rem}

\subsection{Local persistence of singular inputs}

Let $A \in C^r( \cU, \GL(d,\R))$, $r \ge 1$.
We will work upon \cref{l.easy_poor_data} in order to 
obtain a more practical way to detect that the $1$-jet of $A$ at a point
corresponds to a conspicuously poor datum
(which as mentioned in \cref{r.addendum_of_addendum} is the only type of poor data that appear generically).
For example, in the $m=1$, $d=2$ case, we will see that conspicuous poorness means that
the angular velocity of one of the eigendirections vanishes (see \cref{r.speed} below).

\medskip

Suppose that $u_0 \in \cU$ is such that the matrix $A(u_0)$ is diagonalizable over $\R$ 
and with simple eigenvalues only.
By \cref{p.eigen_smooth}, there is a neighborhood $\cU_0$ of $u_0$
 and $C^r$-maps $\lambda_1$, \dots, $\lambda_d\colon \cU_0\to \C$ such that for all $u\in \cU_0$, the complex numbers $\lambda_i(u)$ are all distinct, and form the spectrum of $A(u)$; moreover there exist
a $C^r$ map $P \colon \cU_0 \to \GL(d,\R)$ 
such that for all $u \in \cU_0$,
\begin{equation}\label{e.diagonalize}
A(u) = P(u) \, \Delta(u) \, P^{-1}(u) \text{ , where }
\Delta(u) = \Diag (\lambda_1(u), \dots, \lambda_d(u)).
\end{equation}

For simplicity, 
let us consider first case where $\cU$ is an interval in $\R$ (in particular $m=1$).
Then the normalized derivative of $A$ at a point $u$ 
can be identified with $N(u):= A'(u) \, A^{-1}(u)$.
Consider the expression of $N(u)$ in the basis that diagonalizes $A(u)$, that is,
$B(u) := P^{-1}(u) \, N(u) \, P(u)$.
Since
$\frac{\mathrm{d}}{\mathrm{d}u}P^{-1}(u) = - P^{-1}(u) \, P'(u) \, P^{-1}(u)$, 
we compute that
$$
B(u) = \Delta'(u) \, \Delta^{-1}(u) 
+ Q(u) - \Delta(u) \, Q(u) \, \Delta^{-1}(u) \, , 
$$ 
where
$$
Q(u) := P^{-1}(u) \, P'(u) \, .
$$
So the off-diagonal entries of the matrices $B(u)$ and $Q(u)$ are related by
$$
b_{ij}(u) = \big(1 - \lambda_i(u) / \lambda_j(u)\big) \, q_{ij}(u)
\quad (i\neq j).
$$
In view of \cref{l.easy_poor_data}, we conclude the following:
if for some $u_* \in \cU_0$
\begin{equation}\label{e.easy_poor_cond}
\text{there is an off-diagonal entry position $(i,j)$ such that $q_{ij}(u_*) = 0$}
\end{equation}
then the $1$-jet $j^1 A(u_*)$ is poor.

\begin{rem}\label{r.speed}
Let us give a geometrical interpretation of condition~\eqref{e.easy_poor_cond}.
The columns of $P$ form a basis $(v_1,\dots,v_d)$ of eigenvectors of $A$,
and the rows of $P^{-1}$ form a basis $(f_1,\dots,f_d)$ of eigenfunctionals of $A$
(in the sense that $f_i \circ A = \lambda_i f_i$); 
these two bases are related by $f_i(v_j) = \delta_{ij}$.
So $q_{ij} = f_i \left(\frac{\mathrm{d} v_j}{\mathrm{d} u}\right)$
is the component of the velocity of $v_j$ in the direction of $v_i$.
For example, for $d=2$,  condition~\eqref{e.easy_poor_cond} means 
that one of the eigendirections of $A$ has zero angular speed at instant $u=u_*$.
\end{rem}

It is trivial to adapt the previous calculations to the higher dimensional case and then conclude the following:

\begin{prop}\label{p.easy_poor_derivative}
Let $(u_1, \dots, u_m)$ be coordinates in a chart domain $\cU_0\subset \cU$
where expression~\eqref{e.diagonalize} holds. 
Consider matrices
\begin{equation}\label{e.Q_k}
Q_k (u) := P^{-1}(u) \, \frac{\partial P}{\partial u_k}(u) \, .
\end{equation}
If for some $u_* \in \cU_0$ there is an off-diagonal entry position $(i,j)$
such that
\begin{equation}\label{e.easy_poor_cond_multdim}
\text{for each $k=1,\dots,m$, the $(i,j)$-entry of the matrix $Q_k(u_*)$ vanishes}
\end{equation} 
then the $1$-jet $j^1 A(u_*)$ is poor, that is, 
the constant input $(u_*,\dots,u_*)$ (of any length) is singular.
\end{prop}

In the situation of \cref{p.easy_poor_derivative},
assume additionally that the map
\begin{equation}\label{e.2nd_derivative}
\Phi\colon \begin{cases}
\cU_0& \to \Im \Phi \subset \K^m\\
u&\mapsto \left[ \text{the $(i,j)$-entry of } Q_k (u)
\right]_{1\leq k \leq m}
\end{cases} \mbox{ is a diffeomorphism. } 
\end{equation}
In that case, the existence of a poor jet is persistent in the following way:
If $\tilde A$ is sufficiently $C^2$-close to $A$
then by \cref{p.eigen_smooth} we can express 
$\tilde A (u) = \tilde P(u) \, \tilde \Delta(u) \, \tilde P^{-1}(u)$
for $u$ close to $u_*$, where $\tilde P$ and $\tilde \Delta$ are $C^2$-close to $P$ and $\Delta$ respectively, and $\tilde \Delta$ is diagonal.
The corresponding matrices $\tilde Q_k = \tilde P^{-1} \, \frac{\partial \tilde P}{\partial u_k} $ are $C^1$-close to $Q_k$ and the map 
\begin{align*}
\tilde \Phi \colon u \mapsto \left[ \text{the $(i,j)$-entry of } \tilde Q_k (u)
\right]_{1\leq k \leq m}
\end{align*}
is $C^1$-close to $\Phi$. By~\eqref{e.2nd_derivative} the fact that $\Phi(u_*)=(0,...,0)$,  there is $\tilde u$ close to $u_*$ such that $\tilde \Phi(u)=(0,...,0)$.
In particular the $1$-jet $j^1 \tilde A(\tilde u)$ is poor.

Now, concerning existence:
It is evident that a domain $\cU_0$ and $2$-jets $j^2 P(u_*)$ satisfying conditions \eqref{e.easy_poor_cond_multdim} and \eqref{e.2nd_derivative} actually exist; moreover we can always find a map $P \colon \cU \to \GL(d,\R)$ with a prescribed $2$-jet at a point $u_*$. In view of the discussion above, we conclude the following:

\begin{prop}[Persistence of singular inputs]\label{p.persistent_poorness}
For any $d\geq 1$ and any $d$-dimensional smooth manifold $\cU$,
there exists a $C^2$-open nonempty subset of maps $A\in C^2(\cU,\GL(d,\R))$
with the following property:
there exists $u \in \cU$ such that the constant inputs $(u,\dots,u)$ of any length are all singular for the system~\eqref{e.proj semilin CS}.
\end{prop}

That is, one cannot improve \cref{t.main} replacing ``discrete set'' by ``empty set''.

Given any map $A$ such that \eqref{e.easy_poor_cond_multdim} holds at some point,
we can $C^1$-perturb $A$ (by $C^0$-perturbing $P$) in a way such that \eqref{e.easy_poor_cond_multdim} now holds for a non-discrete set of points.
This shows that the statement of \cref{t.main} with ``$C^2$-open'' replaced by ``$C^1$-open'' is not true.
Using the same idea and Baire's theorem, one can also show that the conclusion of \cref{t.main} is not true for $C^1$-generic maps $A$; actually for $C^1$-generic $A$, the points $u \in \cU$ corresponding to singular constant controls form a perfect set.

\subsection{Other control-theoretic properties}

We now introduce a few control-theoretic notions related to accessibility and regularity,
and discuss the validity of statements similar to \cref{t.main} for these notions.

Consider a general control system \eqref{e.general CS}.
Fix a time length $N$, and let $\phi_N$ denote the response map as in \eqref{e.final state}.
We say that a trajectory determined by $(x_0; u_0, \dots, u_{N-1})$ is:
\begin{itemize}
\item \emph{locally accessible} 
if for every neighborhood $V$ of $(u_0, \dots, u_{N-1})$ in $\cU^N$,
the set $\phi_N(\{x_0\} \times V)$ has nonempty interior.
\item \emph{strongly locally accessible} if for every neighborhood $V$ of $(u_0, \dots, u_{N-1})$ in $\cU^N$, the set $\phi_N(\{x_0\} \times V)$ contains in its interior the final state $\phi_N(x_0;u_0,\dots,u_{N-1})$.
\end{itemize}
The following implications are immediate:
$$
\text{regular $\Rightarrow$ strongly locally accessible $\Rightarrow$ locally accessible.}
$$
We say that an input $(u_0, \dots, u_{N-1})$ is \emph{universally locally accessible} 
(resp.\ \emph{universally strongly locally accessible}) if the trajectory determined by  
$(x_0; u_0, \dots, u_{N-1})$ is locally accessible (resp.\ strongly locally accessible). 

Now we come back to the context of 
projective semilinear control systems \eqref{e.proj semilin CS}.
A (relatively weak) corollary of \cref{t.main} is
that for generic maps~$A$, universal local accessibility holds at all 
constant inputs:

\begin{prop}\label{p.local acc}
Let $N \in \N$ and $\cO \subset C^2(\cU,\GL(d,\R))$ be as in \cref{t.main}.
For any $A \in \cO$, every constant input sequence of length $N$ is universally locally accessible.
\end{prop}

\begin{proof}
If $A\in \cO$ then for every constant input sequence of length $N$
we can find a regular input sequence nearby.
\end{proof}

As we have shown in \cref{p.persistent_poorness}, 
it is not possible to improve \cref{p.local acc} by replacing ``local accessible'' by ``regular''.
Neither it is possible to replace ``local accessible'' by ``strongly local accessible'',
as the following simple example (in dimensions $m=1$, $d=2$) shows:

\begin{example}
For $u \in \R$, define
$$
P(u) = \begin{pmatrix} 1 & u \\ u^2 & 1 \end{pmatrix}, \quad \Delta(u) = \Diag(2,1).
$$
Let $\cU$ be an small open interval containing $0$,
and define $A \colon \cU \to \GL(2,\R)$ by \eqref{e.diagonalize}.
Let $\xi_0 \in \R\P^1$ correspond to the direction of the vector $(1,0)$.
Then for any subinterval $V \ni 0$, and any $N>0$, the set 
$$
\phi_N (\{\xi_0\} \times V^N) = \big\{ A(u_{n-1})\cdots A(u_0)\cdot \xi_0 \; u_i \in V \big\}
$$
is an interval of $\R\P^1$ containing $\xi_0 = \phi_N (\xi_0; 0,\dots,0)$ in its boundary.
Therefore the input $(0,\dots,0)$ is not universally strongly locally accessible.
A similar situation occurs for any $C^2$-perturbation of $A$.
\end{example}

\bigskip

\begin{ack}
We are grateful for the hospitality of Institute Mittag--Leffler, 
where this work begun to take form.
We thank R.~Potrie, L.~San~Martin, S.~Tikhomirov, and C.~Tomei for valuable discussions.
We thank the referees for corrections, references to the literature,
and other suggestions that helped to improve the exposition.
\end{ack}


%
%
%
%
%
%
%
%
%
%
%
%
%

\end{document}